\documentclass[a4paper]{amsart}

\usepackage{amsmath,amsfonts,amssymb,amsxtra,amscd,verbatim,multicol}
\usepackage{tikz}

  \RequirePackage{tikz-cd}
 \usepackage{color}
 \allowdisplaybreaks
\usepackage{xcolor} 
 \usepackage[colorlinks=true,linkcolor=blue,urlcolor=black,citecolor=purple]{hyperref}

\usepackage{url}

   \usepackage{color}
\definecolor{note}{rgb}{0.1,0.1,0.4}
\definecolor{note}{rgb}{0.1,0.1,0.4}

\numberwithin{equation}{section}
 \theoremstyle{plain}

\newtheorem{theorem}[equation]{Theorem}

\newtheorem{sublemma}[equation]{Sublemma}
\newtheorem{lemma}[equation]{Lemma}

\newtheorem{corollary}[equation]{Corollary}
\newtheorem{proposition}[equation]{Proposition}

\theoremstyle{definition}
\newtheorem{definition}[equation]{Definition}
\newtheorem{notation}[equation]{Notation}

\newtheorem{remark}[equation]{Remark}
\newtheorem{remarks}[equation]{Remarks}

\newtheorem{example}[equation]{Example}
\newtheorem{examples}[equation]{Examples}
\newtheorem{question}[equation]{Question}

 \makeatletter
\DeclareRobustCommand{\qed}{%
  \ifmmode
    \eqno \def\@badmath{$$}
    \let\eqno\relax \let\leqno\relax \let\veqno\relax
    \hbox{\openbox}%
  \else
    \leavevmode\unskip\penalty9999 \hbox{}\nobreak\hfill
    \quad\hbox{\openbox}%
  \fi
}
\makeatother  

\usepackage[mathscr]{euscript} 
\DeclareMathAlphabet{\euls}{U}{eus}{m}{n}

 \newcommand{\eD}{{\euls{D}}}
 
\newcommand{\eG}{{\euls{G}}}

 \newcommand{\eS}{{\euls{S}}}

\newcommand{\eu}{{\mathrm{eu}}}

\newcommand{\degeu}{{\deg_{\eu}}}

\newcommand{\vt}{\vartheta}

\DeclareMathOperator{\ord}{{ord}}

\DeclareMathOperator{\Lie}{{Lie}}
\DeclareMathOperator{\Der}{{Der}}

\DeclareMathOperator{\adj}{{ad}}

\newcommand{\lto}{\longrightarrow}

\newcommand{\calA}{\mathcal{A}}
\newcommand{\calC}{\mathcal{C}}
\newcommand{\calS}{\mathcal{S}}
\newcommand{\calO}{\mathcal{O}}
\newcommand{\C}{\mathbb{C}}
\newcommand{\D}{\mathcal{D}}

\DeclareMathOperator{\Hom}{{Hom}}
\DeclareMathOperator{\End}{{End}}

\newcommand{\mc}{\mathcal}
\DeclareMathOperator{\gr}{{gr}}
\DeclareMathOperator{\Sym}{{Sym}}
\DeclareMathOperator{\GKdim}{{GKdim}}

\newcommand{\half}{\frac{1}{2}}

\newcommand{\BZ}{\mathbb{Z}}

\newcommand{\git}{/\!/}

\newcommand{\modmod}{/ \hskip -3pt /}

\newcommand{\rr}{\varrho}
\newcommand{\rrr}{\widetilde{\rr}}
\newcommand{\rrp}{\rr^*}
\newcommand{\rrpp}{\xi}
\newcommand{\gtilde}{\widetilde{\mathfrak{g}}}
\newcommand{\Gtilde}{\widetilde{G}}

 \newcommand{\rad}{\operatorname{rad}}

\newcommand{\Kdim}{\operatorname{Kdim}}
\newcommand{\Ker}{\operatorname{ker}}
\renewcommand{\Im}{\operatorname{Im}}
\renewcommand{\o}{\otimes}

\newcommand{\reg}{\operatorname{reg}}

\newcommand{\rank}{\operatorname{rank}}

\newcommand{\g}{\mathfrak{g}}
\newcommand{\p}{\mathfrak{p}}
\newcommand{\h}{\mathfrak{h}}
\newcommand{\mf}{\mathfrak}
\newcommand{\ds}{\dots}
\newcommand{\idot}{\bullet} 
\newcommand{\mr}{\mathrm}

\newcommand{\R}{\mathbb{R}}

\newcommand{\balp}{\boldsymbol{\alpha}}

\newcommand{\Hk}{H_{\kappa}}
\newcommand{\Ak}{A_{\kappa}}
\newcommand{\dd}{\eD}
\newcommand{\bpsi}{\boldsymbol{\psi}}

\newcommand{\Z}{\mathbb{Z}}
\newcommand{\vs}{\varsigma}
\newcommand{\iso}{\stackrel{\sim}{\rightarrow}}
\newcommand{\isom}{\stackrel{\sim}{\longrightarrow}}

  \newcommand{\deltah}{{h}}
\newcommand{\deltav}{{\delta}}

\usepackage{verbatim,rotating,lscape}

\def\preisomto{\vbox{\hbox to
                 14pt{\hfill$\sim$\hfill}\nointerlineskip\vskip -0.5pt
                 \hbox to 14pt{\rightarrowfill}}}

\def\isomto{\mathop{\preisomto}}

\def\prelongisomto{\vbox{\hbox to
                17pt{\hfill$\sim$\hfill}\nointerlineskip\vskip -0.5pt
                \hbox to 17pt{\rightarrowfill}}}
\def\longisomto{\mathop{\prelongisomto}}

\def\boplus{\mathbin{\boldsymbol{\oplus}}}

\def\trait{\hbox to 2mm{\hrulefill}}

\newcommand{\N}{{\mathbb{N}}}
\newcommand{\fS}{{\mathfrak{S}}}

\newcommand{\fso}{{\mathfrak{so}}}

\newcommand{\fsl}{{\mathfrak{sl}}}
\newcommand{\fsp}{{\mathfrak{sp}}}

\newcommand{\fs}{{\mathfrak{s}}}

\newcommand{\Asf}{{\mathsf{A}}}
\newcommand{\Bsf}{{\mathsf{B}}}
\newcommand{\Csf}{{\mathsf{C}}}
\newcommand{\Dsf}{{\mathsf{D}}}
\newcommand{\Esf}{{\mathsf{E}}}
\newcommand{\Fsf}{{\mathsf{F}}}
\newcommand{\Gsf}{{\mathsf{G}}}
\newcommand{\Isf}{{\mathsf{I}}}

\newcommand{\Vsf}{{\mathsf{V}}}

\newcommand{\csh}{c_{\mathrm{sh}}}
\newcommand{\clg}{c_{\mathrm{lg}}}

\newcommand{\cher}[2]{H_{{#1}}({#2})}
\newcommand{\cheW}[1]{{H_{{#1}}(W)}}
\newcommand{\chec}{{H_{c}(W)}}
\newcommand{\etriv}{e_{\triv}}

\newcommand{\cchi}{{c^{\chi}}}

\newcommand{\esgn}{{e_{\sgn}}}

\newcommand{\scal}[2]{{\langle {#1} \, , \, {#2} \rangle}}

\theoremstyle{definition}

\theoremstyle{remark}

\theoremstyle{plain}
\newtheorem{mainthm}{Theorem}

\newtheorem{mainhypo}[mainthm]{Hypothesis}

\newcommand{\II}{\mathsf{I\vspace{-.02truein}I}}
\newcommand{\III}{\mathsf{I\vspace{-.02truein}I\vspace{-.02truein}I}}
\newcommand{\IV}{\mathsf{I\vspace{-.02truein}V}}
\newcommand{\VI}{\mathsf{V\vspace{-.02truein}I}}
\newcommand{\VII}{\mathsf{V\vspace{-.02truein}I\vspace{-.02truein}I}}
\newcommand{\VIII}{\mathsf{V\vspace{-.02truein}I\vspace{-.02truein}I\vspace{-.02truein}I}}
\newcommand{\IX}{\mathsf{I\vspace{-.02truein}X}}

\DeclareMathOperator{\GL}{{GL}}

\DeclareMathOperator{\sgn}{{sgn}}
\DeclareMathOperator{\triv}{{triv}}
\DeclareMathOperator{\rk}{{rank}}
\DeclareMathOperator{\ann}{{ann}}
\DeclareMathOperator{\Iring}{{Q}}
\DeclareMathOperator{\st}{{s}}

\newcommand{\integral}{{integral}}
\newcommand{\gainly}{{robust}}

\newcommand{\dims}{{d}}
\newcommand{\NV}{{\mbN(\p)}}
\newcommand{\NVreg}{{\mbN(\p)^{\mathrm{reg}}}}
\newcommand{\Vdist}{{\NV^{\mathrm{dist}}}}

\newcommand{\sminus}{\smallsetminus}
\newcommand{\mbN}{{\mathbf{N}}}
\newcommand{\LL}{{\euls{L}}}
\newcommand{\KK}{{\euls{K}}}

\DeclareMathOperator{\Supp}{{\mathrm{Supp}}}
\DeclareMathOperator{\Ch}{{\mathrm{Ch}}}
\DeclareMathOperator{\HC}{{\mathrm{HC}}}

\begin{document}

 \title[Quantum Hamiltonian Reduction]{Quantum Hamiltonian
   Reduction for Polar Representations}
 \author{G. Bellamy} \address{(Bellamy) School of Mathematics and
   Statistics, University of Glasgow, University Gardens, Glasgow
   G12 8QW, Scotland. }\email{gwyn.bellamy@glasgow.ac.uk}
 \author{T.~Levasseur} \address{(Levasseur) Laboratoire de
   math\'ematiques, CNRS UMR 6205, Universit\'e de Brest, 29238 Brest
   cedex~3, France} \email{thierry.levasseur@univ-brest.fr}
 \author{T. Nevins} \address{(Nevins) Department of Mathematics,
   University of Illinois at Urbana-Champaign, Urbana, IL 61801,
   USA.}
  \author{J. T. Stafford} \address{(Stafford) School of
   Mathematics, The University of Manchester, Manchester M13 9PL,
   England.}  \email{toby.stafford@manchester.ac.uk}

 \dedicatory{\it Dedicated to the memory of our friend and
   coauthor Tom Nevins} 


 \begin{abstract}
        Let $G$ be a reductive complex algebraic group with Lie algebra $\g$ and suppose that $V$ is a  polar $G$-representation. 
   We prove the existence of a radial parts map $\rad: \dd(V)^G\to \Ak$ from the
   $G$-invariant differential operators  on $V$
   to the spherical subalgebra $\Ak$ of a rational Cherednik
   algebra. Under mild
   hypotheses $\rad$ is shown to be surjective.  
   
    If $V$ is a symmetric space, then $\rad$ is surjective, and
   we  determine exactly when $\Ak$ is a simple ring.  When $\Ak$ is simple,  we also show that 
   $\Ker(\rad)=\left(\dd(V)\tau(\g)\right)^G$, where $\tau:\g\to \dd(V)$ is the differential of the $G$-action.

 \end{abstract} 
 
 \maketitle
 
 \tableofcontents

 \section{Introduction} 
\label{introduction}
  
 Throughout the paper except  stated to the contrary, the base field will be
$\C$.  We fix a connected reductive complex algebraic  group $G$ with Lie algebra $\g=\mathrm{Lie}(G)$. 
    
 The idea of using radial parts to understand invariant differential operators goes back (at least) to
 Harish-Chandra \cite{HC2, HC3}, where he defined a \emph{radial parts homomorphism}
 $\HC: \dd(\g)^G \to \dd(\h)^W$. Here $\h$ is a 
 Cartan subalgebra of $\g$ with Weyl group $W$. This was the
 fundamental algebraic step in proving his theorems on the regularity of
 invariant  eigendistributions, and hence the regularity of the characters of
 representations of real forms of $G$. In \cite{Wallach} and \cite{LSSurjective}
 the map $\HC$ was actually shown to be surjective, enabling the authors to
 provide much easier proofs of those key results of
 Harish-Chandra. Further applications of these ideas to Springer
 theory and Weyl group representations appear in \cite{HK, AJM, Wallach}. 
 
 Harish-Chandra's radial parts map (and resulting regularity
 results) was then generalised in \cite{Se, LS3, GL} to symmetric
 space representations $V$ of $G$ satisfying Sekiguchi's
 ``niceness'' condition.  The definitions of symmetric spaces and
 of Sekiguchi's niceness condition can be found in
 Section~\ref{Sec:examples}. Again for nice symmetric spaces it
 was shown in \cite{LS3} that, for a certain abelian subalgebra
 $\mathfrak{a}\subset V$ with associated little Weyl group $W$,
 there is a radial parts map
 $\rad \colon \dd(V)^G\to \dd(\mathfrak{a}/W)$, although the
 precise nature of the image was unknown.
  
  Much more recently, the notion of radial parts maps has been very
 fruitful in the theory of Cherednik algebras or, more
 precisely, of their spherical subalgebras. Here one can use the radial parts map to relate (via the functor of Hamiltonian reduction) equivariant $\dd$-modules on certain $G$-representations to representations of the (spherical) Cherednik algebra. This interplay is useful, both for understanding the equivariant $\dd$-modules on $V$ and the representation theory of Cherednik algebras; see for example 
  \cite{BG,  EG,   GG, GGS,  Lo0, MN, OblomkovHC}. 
  
 \bigskip  
  
Given the applications one can derive from the existence of a radial parts map, it is important to understand how generally these result will hold. This is the goal of the article -- we describe what we expect is the most general setting in which the theory can be developed. As explained below, developing the theory in this generality allows us to prove new results even in the previously ``well-understood'' case of symmetric spaces. 

\subsection*{Polar representations}
 
 To motivate the general setting, let us consider the classical example of rotation invariant differential operators on $\R^3$. Here the radial parts decomposition allows one to understand easily the action of the Laplacian on invariant functions. This is possible because restriction of functions to a line through the origin is an isomorphism between rotation invariant functions on $\R^3$ and functions on the line invariant under the involution $v \mapsto -v$. This latter isomorphism holds precisely because the line is orthogonal to every $SO(3)$-orbit. As a consequence, in order to understand the action of the Laplacian on rotation invariant functions, we are reduced to studying differential operators on the line that are invariant  under $\mathbb{Z}_2$. 
 
 In general, we begin with a slice $\h \subset V$ intersecting
 transversely the (closed) $G$-orbits in $V$. Then the key to the
 existence of a radial parts map is that restriction to $\h$
 should provide a version of the Chevalley isomorphism
 $\C[V]^G \iso \C[\h]^W$, where $W = N_G(\h) / Z_G(\h)$ is a
 \textit{finite} group. The representations where one can find a
 slice with this property are precisely the \emph{polar
   representations} $V$ of \cite{DadokKac}. Thus, it is natural
 to ask whether a radial parts map holds for $\dd(V)^G$ when $V$
 is a polar $G$-representation.  In this paper we show that this
 indeed the case. Indeed, we prove that there exists such a
 radial parts map $\rad$ for any polar representation $V$ (see
 Theorem~\ref{thm:intro-existence}) and, under mild conditions,
 $\rad$ is surjective (see
 Theorem~\ref{thm:intro-surjection}). In this generality
 the image of $\rad$ will now be some spherical subalgebra
 $\Ak(W)$ of a Cherednik algebra, realised as a subalgebra of
 $\dd(\h_{\reg})^W$ via the Dunkl embedding.

   
   In order to state these results precisely, we need some definitions. 
    Following Dadok and Kac \cite{DadokKac},  a 
representation $V$ of $G$ is called  \textit{polar}  if there exists a semisimple element $v \in V$ such that 
$$
\h := \{ x \in V \, | \, \mf{g} \cdot x \subset \mf{g} \cdot v \}
$$
satisfies $\dim \h = \dim V \git G$.  One of the fundamental properties of polar representations is that the finite group $W = N_G(\h)/ Z_G(\h)$ acts as a complex reflection group on $\h$ and, by restriction, defines a Chevalley isomorphism $\rr \colon V \git G \cong \h / W$; see  \cite[Theorem~2.9]{DadokKac}. The group  $W$ is   called 
\emph{the Weyl group} associated to this data. Thus, polar representations are a natural situation---perhaps the most general one---where one can hope to have a radial parts map.
 
It is immediate that invariant differential operators act on
invariant functions, and so there is a natural map
$\phi : \dd(V)^G\to \dd(V\git G) \cong \dd(\h/ W)$, but its
image is almost never in $\dd(\h)^W$. Instead, Harish-Chandra's
morphism $\HC$ differs from $\phi$ by conjugation with
$\delta^{1/2}$ for an appropriate element $\delta$ (the
discriminant).
 
As was already apparent in previous applications to the study of Cherednik algebras, it is important to be able to twist the radial parts map and our version of the the radial parts map will be defined in this context. This is constructed in Section~\ref{sec:radialpartsmap1}, but in outline the definition is as follows.
For any complex reflection group $(W,\h)$, there exists a discriminant $h\in \C[\h]^W$ such that 
 the regular locus $\h_{\reg} = (h\not=0) \subset \h$ is precisely the locus where $W$ acts freely.
 We then define the \textit{discriminant} $\deltav=\rr^{-1}(h)$ on $V$ with regular locus $V_{\reg} =(\deltav\not=0) \subset V$. Let $\tau : \g \to \mathrm{Der} (\C[V])$ be the
   differential  of the action of $G$ on $V$, where $\Der(\C[V]) $
   denotes the space of $\C$-linear derivations.
 If $\chi$ is a character of $\g$ we set
 \[\g_{\chi} = \{ \tau(x) - \chi(x) : x \in \g \}.\]
  One would
 expect to define a radial parts map such that the localization  to the regular locus $V_{\reg}$ is a  surjection
 \begin{equation}\label{eq:localizatioradiso}
   (\dd(V_{\reg}) / \dd(V_{\reg}) \g_{\chi})^G
  \twoheadrightarrow \dd(\h_{\reg})^W \cong \dd(\h_{\reg}/ W).   
 \end{equation}
 It is soon apparent that any twisted radial parts map satisfying
 \eqref{eq:localizatioradiso} is governed by the factorisation
 $\deltav = \deltav_1^{m_1} \cdots \deltav_k^{m_k}$ of the
 discriminant into irreducible, pairwise distinct
 polynomials. Then, for complex numbers $(\vs_1,\dots,\vs_k) \in \C^k$, we
 define $\rad_{\vs} \colon \dd(V)^G \to \dd(\h_{\reg})^W$ by
$$
\rad_{\vs}(D)(z) = (\deltav^{-\vs} D(\rr^{-1}(z) \deltav^{\vs}))
|_{\h_{\reg}},
$$
where
$\deltav^{\vs} := \delta_1^{\vs_1} \cdots \delta_k^{\vs_k}$ and $D \mapsto \deltav^{-\vs}D\deltav^{\vs}$ is the automorphism of $\dd(V_{\reg})^G$ given by conjugation by
$\deltav^{-\vs}$. Twisting by $\delta^{\vs}$ can also be thought of as identifying $\dd(V_{\reg})$ with the ring of differential operators on the free $\C[V_{\reg}]$-module $\C[V_{\reg}]\delta^{\vs}$, as is described in more detail in Section~\ref{sec:radialpartsmap1}.
Since $G$ is connected, each $\deltav_i$ is
a $G$-semi-invariant of weight $\theta_i$ say. This means that the character $\chi=\chi(\vs)$ associate to $\vs$ in \eqref{eq:localizatioradiso} must be $\sum_{i=1}^k \vs_i d \theta_i$.

 With this notation to hand, we can state the first main result of this paper. 

\begin{theorem}\label{thm:intro-existence}
{\rm (Theorem~\ref{thm:surjectivereducerank1} and Lemma~\ref{lem:genericregfiltredsurj})}
 Assume that  $V$ is a polar representation for $G$.  Then there exists a parameter $\kappa= \kappa(\vs)$ for the
  spherical subalgebra $\Ak(W)$ of the rational Cherednik algebra $H_{\kappa}(W)$ such that
  $\mr{Im}(\rad_{\vs}) \subseteq \Ak(W)$. 
  
  The localization of
  $\rad_{\vs}$ to $V_{\reg}$ induces the  surjection
  \eqref{eq:localizatioradiso}.
\end{theorem}

Fundamental to the proof of this theorem is the fact that the
parameter $\kappa$ can be computed by taking rank one slices in
$V$. Following \cite{Le3}, the rank one case can be understood explicitly, and is described in Section~\ref{sec:radialpartsmap1}.  Using those slices, the proof of the theorem is then completed in Section~\ref{sec:radialpartsmap}. In explicit examples, the method of reduction to rank one slices provides a powerful, and quick, way to compute the parameter $\kappa$. This completely avoids the arduous calculations one normally needs to perform to compute the image of $\rad_{\vs}$. 

We remark that, if the representation $V$ is stable, then 
 \eqref{eq:localizatioradiso} is actually an isomorphism; see Corollary~\ref{cor:unstablepolaropen}. This will be important in \cite{BNS}.

 Naturally, our second main results
describes situations in which the radial parts map
$\rad_{\vs} \colon \dd(V)^G \to \Ak(W)$ is surjective.

\begin{theorem}\label{thm:intro-surjection} 
{\rm(Theorem~\ref{thm:radsurjects} and Corollary~\ref{thm:Levsurjectiverank1}.)}
Assume that  $V$ is a polar representation for $G$. 
  Then the  map $\rad_{\vs} \colon \dd(V)^G \to \Ak(W)$ is surjective
 in each of the following cases:
  \begin{enumerate}
  \item when $\Ak(W)$ is a simple algebra;
  \item when $V$ is a symmetric space;
  \item when the associated complex reflection group $W$ is a Weyl
    group with no summands of type $\mathsf{E}$ or $\mathsf{F}$;
  \item when the rank of $(W,\h)$ is at most one.
  \end{enumerate}
\end{theorem} 

We do not know of any  polar representations where
$\rad_{\vs}$ is not surjective. 

If $\vs$ and $\vs'$ are two choices of twist then the images of $\rad_{\vs}$ and $\rad_{\vs'}$ are not, in general, equal as subalgebras of $\dd(\h_{\reg})^W$. However, as shown in Corollary~\ref{lem:invconjugationvs}, they are \textit{isomorphic} if $\chi(\vs) = \chi(\vs')$ since the kernel of $\rad_{\vs}$ equals that of $\rad_{\vs'}$ in this case.   

\subsection*{Symmetric spaces}

In the companion paper \cite{BNS} we will apply the map  $\rad_{\vs}$ to study the representation theory of admissible 
 $\dd$-modules on $V$ and of modules over $\Ak(W)$.
  The theory developed there only works when $\Ak(W)$ is a simple ring, and so it is is important to
  understand when that is the case. In   Sections~\ref{Sec:examples} and~\ref{Sec:kernel} we study the case where
   $V$ is a symmetric space, at least for the most important case when the character $\chi$ is zero. 
   The precise   result is given in Theorem~\ref{simplicity-irred-symm-space} and depends upon a
    case-by-case analysis. But to summarise, 
  \emph{if $V$ is an irreducible symmetric space and $\vs=0$,  then $\Ak$ is simple if and only if 
  either $V$ is nice or the parameter $\kappa$ takes only integer values.} The symmetric spaces 
  satisfying these properties are itemised in the tables in Appendix~\ref{app-tables}.

It is easy to see that the kernel of $\rad_{\vs}$
contains $(\dd(V) \g_{\chi})^G$. Much harder than surjectivity is
to decide when the kernel of $\rad_{\vs}$ equals
$(\dd(V) \g_{\chi})^G$. This is known for the adjoint representation of  a reductive Lie algebra $G$ on $\g=\mathrm{Lie}(G)$ \cite{LS} and, more generally,  for nice symmetric spaces \cite[Theorem~A]{LS3}. 
In Theorem~\ref{mainthm} we generalise this to prove:

\begin{theorem}\label{intro-mainthm} Let $V$ be a symmetric space and take $\vs=0$. If   the algebra 
$\Ak=\Im(\rad_0)$ is simple, then $\Ker(\rad_0)= (\dd(V) \tau(\g))^G$.
\end{theorem}

  One significance of this theorem is that it has the following vanishing result as an easy consequence.  In the special case when  
$V=\g$ is the adjoint representation, this is a classic result of Harish-Chandra \cite[Theorem~3]{HC2}.

\begin{corollary}
  \label{intro-maincor2}  
  {\rm (Corollary~\ref{maincor2}.)}  Assume that the   hypotheses of Theorem~\ref{intro-mainthm} are 
  satisfied 
 and pick a real form $ (G_0,\p_0)$. Let $T$ be a locally
invariant eigendistribution on an open subset $U \subset \p_0$
with support contained in $U \sminus U'$, where $U'$ is the set
of regular semisimple elements in $U$. Then $T=0$.\qed
\end{corollary}

\subsection*{Applications}

We stated earlier in the introduction that our motivation for
constructing the radial parts map in this generality was so that
we can peruse the many resulting applications. The applications
of this work are indeed numerous, to the extent that we have left
them to the companion paper \cite{BNS}. The main application
explored there is to use the radial parts map to the
understanding of the Harish-Chandra $\dd$-module $\eG$ (whose
distributional solutions are precisely the invariant
eigendistributions)  on $V$. For instance, under the assumption that $V$ is stable and $\Ak$ is simple, we show that:
\begin{enumerate}
        \item the Harish-Chandra $\dd$-module $\eG$ is the minimal extension of its restriction to $V_{\reg}$; and
        \item $\eG$ is semisimple if and only if the associated Hecke algebra $\mc{H}_{q(\vs)}(W)$ is semisimple. 
\end{enumerate}  
Moreover, we study the dependence of $\eG$ on the character $\chi$ when $\eG$ is semisimple. It is shown that under their natural labelling, the simple summands of $\eG$ are permuted by the KZ-twist when we vary the parameter $\vs$.  
  
\medskip
\noindent {\bf Acknowledgements:} The third  author was partially
supported by NSF grants DMS-1502125 and DMS-1802094 and a Simons
Foundation Fellowship.  The fourth author is partially supported
by a Leverhulme Emeritus Fellowship.

 Part of this material is based upon work supported by the National Science Foundation under   Grant   DMS-1440140, while the third 
author was   in residence at the Mathematical Sciences Research Institute in Berkeley, California 
during the Spring 2020 semester. 

\medskip Finally Gwyn Bellamy, Thierry Levasseur and Toby Stafford
would like to acknowledge the great debt they hold for Tom Nevins
who tragically passed away while this paper was in
preparation. He was a dear friend and a powerful mathematician
who taught us so much about so many parts of mathematics.


\section{Spherical rational Cherednik
        algebras}\label{Sec:Cherednik}

In this section we define and prove various basic facts, many of
which are well-known, about rational Cherednik algebras and their
spherical subalgebras.  Throughout the paper, $\otimes$ will implicitly mean $\otimes_{\C}$.

First, if $V$ is an affine algebraic variety, we denote by
$\C[V]$ the algebra of regular functions on $V$.   
The symmetric algebra of a vector space $U$ is denoted
$\Sym U$.  Then, following \cite{GGOR} or \cite{EG}, the
Cherednik algebra is defined as follows.

\begin{definition}\label{defn:Cherednik}
        Let $W$ be a complex reflection group with reflection
        representation $\h$ and reflections $\calS\subset W$. To each
        $s\in \calS$ is associated a reflecting hyperplane
        $H = \ker (1 - s)\subset \h$. Let $\calA$ denote the set of all
        reflecting hyperplanes and for each $H \in \calA$, let 
        $W_H \subseteq W$\label{WsubH} denote the point-wise stabiliser
        of $H$. This is a cyclic group of order $\ell_H$ for which 
        we fix a generator $s_H$.

        The elements
        $$
        e_{H,i} = \frac{1}{\ell_H} \sum_{s \in W_H} \mr{det}_{\h}(s)^i s,
        $$
        form an idempotent basis of $\C W_H$ with
        $s e_{H,i} = \mr{det}_{\h}(s)^{-i} e_{H,i}$. Choose
        $\alpha_H \in \h^*$ such that $H = \ker(\alpha_H)$ and
        $\alpha_H^{\vee} \in \h$ is an eigenvector of $s_H$ with
        eigenvalue not equal to one. Then
        $\langle \alpha_H^{\vee},\alpha_H \rangle \neq 0$. 
        For $H \in \calA$ and
        $0\leq i\leq  \ell_H - 1$,    choose  $\kappa_{H,i}\in \C$ subject to
        $\kappa_{w(H),i} = \kappa_{H,i}$ for all $w \in W$, $H$ and
        $i$. Finally, let $T(\h^*\oplus
        \h)$ denote the free algebra on the vector space $\h^*\oplus
        \h$ with skew group algebra $T(\h^*\oplus \h)\rtimes
        W$. 
        
        Following \cite[(3.1)]{GGOR}, the \emph{rational
                Cherednik algebra} (with parameter
        $t=1$) is defined to be  the quotient algebra $\Hk=\Hk(W,\h)$ of $T(\h^*\oplus
        \h)\rtimes W$ by the relations $[x,x'] = [y,y'] = 0$ and
        \begin{equation}\label{eq:relaltRCAnon}
                [y,x] = \langle y,x\rangle + \sum_{H \in \calA} \ell_H
                \frac{\langle \alpha_H^{\vee},x  \rangle \langle y,\alpha_H
                        \rangle}{\langle \alpha_H^{\vee},\alpha_H  \rangle} \sum_{j =
                        0}^{\ell_H-1} (\kappa_{H,j+1} - \kappa_{H,j}) e_{H,j},  
        \end{equation}
        for all $x,x'\in \h^*$ and $y,y'\in
        \h$. This algebra was first introduced in \cite{EG} and, as a vector
        space, $\Hk\cong \C[\h]\otimes_{\C}\C[W]\otimes_{\C} \Sym \h$, by \cite[Theorem~1.3]{EG}.
\end{definition}

\begin{remark}
        We   impose neither the condition
        $\sum_{i = 0}^{\ell_H-1} \kappa_{H,i} = 0$ nor
        $\kappa_{H,0} = 0$. This will be important later when it comes
        to identifying the image of the radial parts map.
\end{remark}

We always identify $\C[\h^*]$ with $\Sym \h$ using the
natural pairing, and regard $\h^*$ as the natural space of
generators for $\C[\h]$. However, it will be more convenient to
use $\Sym \h$ in place of $\C[\h^* ]$ and, similarly, $\Sym V $
in place of $\C[V^*]$. Let $e=\frac{1}{|W|} \sum_{w\in W } w$ be
the trivial idempotent in $\C W$. The  \emph{spherical algebra} or, more formally, the \emph{spherical
        subalgebra}\label{spherical-defn} of $\Hk(W)$ is the algebra $\Ak(W)= e \Hk(W) e$. 
        For brevity, $\Ak(W)$ will often be written
$\Ak$.  
Note that
\[
\C[\h]^W\cong e\C[\h]e \subset \Ak \qquad\text{and}\qquad (\Sym
\h)^W\cong e(\Sym \h)e \subset \Ak.
\]

We note for later the following result.

\begin{lemma}\label{pdQ}
        The algebra $\Ak$ is a free $\left(\C[\h]^W,
        (\Sym \h)^W\right)$-bimodule of finite rank equal to $|W|$. \end{lemma}

\begin{proof} 
        Set  $E=  (\Sym \h)^W$. Since $(W,\h)$ is a complex reflection group, $\C[\h\oplus \h^*]$ is a free module of rank $|W|^2$ over the
        subring $\C[\h]^W\otimes_{\C} E$. As $\C[\h \oplus \h^*]^W$ is a
        graded direct summand of $\C[\h \oplus \h^*]$ as a
        $\left(\C[\h]^W, E\right)$-bimodule, it follows that $\C[\h
        \oplus \h^*]^W$ is itself a free $\left(\C[\h]^W,
        E\right)$-bimodule of finite rank.  Since $\Ak$ has associated
        graded ring $\C[\h \oplus \h^*]^W$, the ring $\Ak$ is also free of
        finite rank as a $\left(\C[\h]^W, E\right)$-bimodule.  
\end{proof}

\subsection*{The Dunkl embedding}

We begin by recalling some standard facts regarding complex
reflection groups; for references, see
\cite{Broue}. Recall that 
the point-wise stabiliser $W_H$ of $H \in \calA$ is a cyclic group of
order $\ell_H$. The element $\prod_{H\in \calA} \alpha_H$ acts
as the inverse determinant character
$\det_{\h}^{-1} \colon W \to \C^{\times}$; see
\cite[Theorem~2.3]{StanleySemiInv}. Moreover, the
\textit{discriminant} on $\h$ is defined to be
\begin{equation}\label{eq:discriminant}
        \deltah = \prod_{H \in \calA} \alpha_H^{\ell_H} \in \C[\h]^W;
\end{equation}
this polynomial is $W$-invariant, and is reduced as an element of
$\C[\h]^W$.

\begin{notation}\label{not:regularelementh}
        Let $\h_{\reg}$ denote the open subset of $\h$ where $W$ acts
        freely. It is a consequence of Steinberg's Theorem that
        $\h_{\reg}$ is precisely the non-vanishing locus of $\deltah$. A
        vector $x \in \h_{\reg}$ is called \textit{regular}.
\end{notation}

Set
\begin{equation}\label{eq:Hdiscriminant}
        \deltah_H = \alpha_H^{-\ell_H} \deltah. 
\end{equation}
Since both $\deltah$ and $\alpha_H^{\ell_H}$ are $W_H$-invariant,
$\deltah_H$ is a $W_H$-invariant.  Let
$$
\Ak(W_H)_{\deltah_H} = \Ak(W_H)[\deltah_H^{-1}]
$$
denote the localisation of $\Ak(W_H)$ at the Ore set
$\{ \deltah_H^k \}$.

As in \cite[Proposition~4.5]{EG}, the \textit{Dunkl embedding}
$\Hk(W) \hookrightarrow \dd(\h_{\reg}) \rtimes W$ is given by
$x \mapsto x, w \mapsto w$ and  \begin{equation}\label{eq:Dunkl}
   y \mapsto T_y := {\partial_y} + \sum_{H \in
                \calA} \frac{\langle y,\alpha_H \rangle}{\alpha_H} \sum_{i
                = 0}^{\ell_H - 1} \ell_H \kappa_{H,i} e_{H,i}. 
\end{equation}  
for all $x \in \h^*, y \in \h$ and $w \in W$, where $\partial_y
\in \Der \C[\h]$ is the vector field defined by $y$.
We remark that the Dunkl operator $T_y$ preserves
$\C[\h] \subset \C[\h_{\reg}]$ only if $\kappa_{H,0} = 0$ for
all $H \in \calA$. We do not impose this condition.

Using the Dunkl embedding, we will always regard  $\Hk(W)$ and $\Hk(W_H)$ as
subalgebras of $\dd(\h_{\reg}) \rtimes W$. Since
$e \Hk(W) e$ is realised as an algebra of
$\mr{ad}(\C[\h_{\reg}]^W)$-nilpotent operators in
$\End_{\C}(\C[\h_{\reg}]^W)$, it follows that $e \Hk(W) e$
acts on $\C[\h_{\reg}]^W$ as differential operators. This defines a map
$$
\Phi \colon \Ak(W) \ = \ e\Hk(W)e \ \ \hookrightarrow  \ \ \dd(\h_{\reg}/W)  \ =
\ \dd(\h_{\reg})^W \  \subset  \ \dd(\h_{\reg}).
$$
The map $\Phi_{H} \colon \Ak(W_H) \hookrightarrow \dd(\h_{\reg})$ is defined
similarly. Thus, we can and will consider $\Ak(W)$, $\Ak(W_H)$ and
$\Ak(W_H)_{\deltah_H}$ as subalgebras of $\dd(\h_{\reg})$.

The following result is standard.

\begin{lemma}  \label{lem:quotientetaleh}
                {\rm (1)}  Under the Dunkl embedding, 
                $\Hk(W)[h^{-1}]= \dd(\h_{\reg})\rtimes W$ and hence
                $\Ak(W)[h^{-1}] =\dd(\h_{\reg})^W$.  
                
                {\rm(2)} The canonical map
                $\dd(\h_{\reg})^W \to \dd(\h_{\reg}/ W)$, given
                by restriction of operators, is an isomorphism of
                filtered algebras.

{\rm (3)} The centre of $\Ak(W)$ reduces to $\C$.                 
\end{lemma}

\begin{proof} (1) Since the $\alpha_H$ are units in $\Hk(W)[h^{-1}]$,  the identity 
$\Hk(W)[h^{-1}] = \dd(\h_{\reg})\rtimes W$   follows from  \eqref{eq:Dunkl}.
        Thus    $$\Ak[h^{-1}] = e \Hk(W)[h^{-1}] e = e \dd(\h_{\reg}\rtimes W) e =
        \dd(\h_{\reg})^W.$$

        (2)  
        Since $W$ acts freely on $\h_{\reg}$, the quotient map $\h_{\reg} \to \h_{\reg}/W$ is \'etale. Therefore,
        \cite[Corollaire~4]{LevasseurDiff} says that, for each $\ell \ge 0$,
        there is an equivariant isomorphism
        $\dd_{\ell}(\h_{\reg}) \iso \C[\h_{\reg}] \o_{\C[\h_{\reg}]^W}
        \dd_{\ell}(\h_{\reg} / W)$. Taking $W$-invariants gives
        $\dd_{\ell}(\h_{\reg})^W \iso \dd_{\ell}(\h_{\reg} / W)$.

   (3) By (1), $A_\kappa(W)[h^{-1}] = \dd(\h_{\reg})^W$
        is a simple ring with centre $\C$. Thus
        the centre of any spherical algebra $A_\kappa(W)$ also
        reduces to~$\C$.
      \end{proof}

\begin{corollary}
  \label{cor:quotientetaleh}, 
Let $W= W_1\times W_2$ be the product of two complex reflection
groups and $\Ak(W_i)$, $i=1,2$, be  spherical algebras
attached to $W_i$. Then, if $\kappa=\kappa_1 \times \kappa_2$,
the spherical algebra $\Ak(W)= \Ak(W_1) \otimes \Ak(W_2)$ is simple if and only if
$\Ak(W_1)$ and $A_\kappa(W_2)$ are both simple. 
  \end{corollary}

  \begin{proof}
    By Lemma~\ref{lem:quotientetaleh}(3), the centre of a
    spherical algebra reduces to $\C$, the result is therefore consequence
    of \cite[Lemma~~9.6.9]{MR}.
  \end{proof}
 
\begin{example}\label{ex:Dunklcyclci}
        In the simplest case, $\h = \C \{ y \}$ and
        $W = \mathbb{Z}/ \ell \mathbb{Z} = \langle s \rangle$, where
        $s(y) = \omega y$, for $\omega$ a fixed primitive
        $\ell^{\text{th}}$ root of unity. If $x \in \h^*$ is such
        that $\langle y,x\rangle= 1$, the Dunkl embedding \eqref{eq:Dunkl} sends
        $y$ to 
        $$
        T_y = \frac{\partial}{\partial x} + \frac{\ell}{x}
        \sum_{i = 0}^{\ell-1} \kappa_i e_i
        $$
        where the idempotent $e_i \in \C W$ satisfies
        $s e_i = \omega^{-i} e_i$. The defining relation is
        $$
        [y,x] = 1 + \ell \sum_{i = 0}^{\ell-1} (\kappa_{i+1} -
        \kappa_i) e_i,
        $$      
        where the subscripts are interpreted modulo $\ell$; thus
        $\kappa_\ell=\kappa_0$.  In this case, $\Ak(W)$ is the
        algebra generated by
        $$
        X = e x^{\ell}, \quad Y = e y^{\ell}, \quad Z = e (xy -
        \ell \kappa_0),
        $$
        together with  $\Phi(X) = z = x^{\ell}$,\ 
        $\Phi_{}(Z) = \ell z \partial_z = x \partial_x$ and 
        \[   
        \Phi_{}(Y)  = \frac{\ell^{\ell}}{z} \prod_{i =
                0}^{\ell-1} \left( z \partial_z  +
        \kappa_{i} + \frac{i}{\ell} +(\delta_{i,0}
        - 1) \right)  = x^{-\ell} \prod_{i = 0}^{\ell-1} \left(
        x \partial_x  + \ell \kappa_{i} + i  +
        \ell(\delta_{i,0} - 1)\right),   
        \]
        where $\delta_{i,0}$ is $1$ if $i = 0$ and $0$ otherwise.
\end{example}

\begin{notation}\label{defn:order}
        Given a differential operator $D\in \dd(X)$, for some smooth
        variety $X$, we let $\ord_X(D)$ denote the order of $D$ as a
        differential operator, and write $\ord D=\ord_XD$ if it is unambiguous.  Set
        $\dd_\ell(X) = \ \left\{d\in \dd(X) : \ord d \leq \ell\right\}$
        and
        $$
        \gr \dd(X)= \gr_{\ord} \dd(X) =\bigoplus_{\ell} \dd_\ell
        (X)/\dd_{\ell-1}(X).
        $$
        Given $d\in \dd_\ell (X) \smallsetminus \dd_{\ell-1}(X)$,
        let
        $\sigma(d)= [p+\dd_{\ell-1}]\in \dd_\ell
        (X)/\dd_{\ell-1}(X)$ denote the \emph{symbol} of $d$.
\end{notation}
 
\begin{remark}\label{rem:orderfiltrationH}
        The algebra $H_{\kappa}(W)$ has \emph{an order filtration} given by
        putting $\C[\h] \rtimes W$ in degree zero and
        $\h \subset \Sym \h$ in degree one. As noted in \cite[p.~ 281]{EG} this agrees, via the Dunkl embedding, with the
        filtration   induced from the order filtration on
        $\dd(\h_{\reg}) \rtimes W$. Give $\dd(\h_{\reg})^W$  
        the  filtration induced from the order filtration on
        $\dd(\h_{\reg})$. Then  the filtration induced on
        $\Ak(W)$ as a subalgebra of $H_{\kappa}(W)$  (again called the  \emph{order filtration}) agrees with the
        filtration on $\Ak(W)$ considered as a subalgebra of
        $\dd(\h_{\reg})^W$ via  $\Phi$; see \cite[page     282]{EG}.  
        The PBW Theorem for
        $H_{\kappa}(W)$, as described in~\cite[Theorem~1.3]{EG}, then identifies their associate graded rings:
        \[
        \gr H_{\kappa}(W) = (\C[\h] \otimes \Sym \h) \rtimes W \quad \text{ and} \quad  \gr
        \Ak(W) = (\C[\h] \otimes \Sym \h)^W.
        \]
\end{remark}
    
\subsection*{A localisation of $\Hk$}

Fix a reflecting hyperplane $H \in \calA$ and recall that $W_H$ denotes the  pointwise stabiliser of $H$. 
In order to  define the radial parts map in the next section, we will use  Luna slices to reduce the problem to the spherical algebra $\Ak(W_H)$ for suitably defined parameters. For this to work we need to  prove that, after localising at a suitable Ore set $\euls{C}$, we have $\Ak(W)\subseteq \Ak(W_H)_{\euls{C}}$;  see Proposition~\ref{prop:sphericalinrankone} for  the precise result. The proof of this result follows the ideas of \cite{BE} and we begin with the appropriate notation.

     Let  $\C W_H \subseteq R \subseteq S$ be $\C$-algebras.   As in \cite[Section~3.2]{BE}, let $F(S)
:= \mathrm{Fun}_{W_H}(W,S)$ denote the space of all $W_H$-equivariant maps
$W \to S$; these are the maps $f$ with $f(nw) = n f(w)$ for all $n \in W_H,
w \in W$.   Set $E(S) = \End_{S}(F(S))$.      Fixing a set of left coset representatives of $W_H$ in $W$  shows that  $F(S)$ is a free right $S$-module  and hence  identifies $E(S)$ with the matrix  ring $\mathrm{Mat}_{|W/W_H|}(S)$.

We denote by $e$ and $e_0$ the trivial idempotents in $\C W$ and $\C W_H$, respectively. 
  Let $f_0 \in F(S)$ be given by $f_0(w) = e_0$ for all $w \in W$. We define $\Psi \colon \C W \to E(S)$ by $(\Psi(u)(f))(w) = f(wu)$ for $u,w \in W$. It is an embedding, and using $\Phi$ we think of $e$ as an element of $E(\C W_H)\subseteq E(S)$; explicitly $e(f) = \frac{1}{|W|} \sum_{u\in W}u\cdot  f$ for $f\in F(S)$.  It follows  that $e(f_0)=f_0$.
 
Note that  $F(\C W_H) \subseteq F(R) \subseteq F(S)$ and $F(S) = F(R) \o_R S$.   Set 
$$
E(R,S) = \{ D \in E(S) \, | \, D(F(\C W_H)) \subset F(R)\},
$$
and define $\phi \colon E(R,S) \to E(R)$ by $\phi(D)(f \o r)(w) = D(f)(w) r$, for $f \o r \in F(\C W_H) \o_{W_H} R = F(R)$. 

The basic properties of these objects is given by the following result.

\begin{lemma} \label{lem:varphi}  Keep the above notation. 
\begin{enumerate}
\item  The morphism $\phi$ is an algebra  isomorphism.
\item
    There is an algebra isomorphism 
         $\varphi \colon e E(S) e \cong e_0S e_0$,  which is defined in the proof. This morphism satisfies $(e D e)(f_0) = f_0\varphi(eDe) $ for $D\in E(S)$.
        \item  Moreover $\varphi(e E(R,S) e) = e_0 R e_0$.
        \end{enumerate}
\end{lemma}

 \begin{proof} 
 (1) After fixing a set of left coset representatives  $\{g_j : 1\leq j\leq t\} $  of $W_H$ in $W$, a direct computation shows that $\phi$ is
 indeed  an isomorphism.  
 
 (2) This   is stated  in 
  \cite[Lemma~3.1(ii)]{BE} under slightly different hypotheses and so we will just outline the argument.
  
 Write $ef$ in place of $e(f)$ for $f\in F(S)$. Expanding $(ef)(x)$ for $f\in F(S)$ and $x\in W$ shows that 
 \[(ef)(x)=e_0 \alpha(f)=f_0(x) \alpha(f)\quad\text{ where }\quad\alpha(f)=\frac{|W_H|}{|W|}\sum_i f(g_i)\in S.\]
 Hence $ef=f_0\alpha(f)$.  Now  $f_0=f_0e_0$.   
 Consequently,  if $D\in E(S)$,   then  
 \[(eDe)(f_0) \ = \ e(D(ef_0)) \ = \ e(D(f_0)) \ = \ f_0 \alpha(D(f_0))   \ = \ f_0e_0\alpha(D(f_0))e_0.\]
 Hence if we define 
$  \varphi(eDe)=e_0\bigl(\alpha(D(f_0))\bigr)e_0, $ for $\  D\in E(S),$
then $ (eDe)(f_0) \ = \ f_0 \varphi(eDe)$, as claimed. 

The fact that $\varphi$ is an  isomorphism
of sets follows from Part~(3) for the special case $R=S$, in which case $E(R,S)=E(S)$. Moreover, since $\varphi$ is actually  defined as an equality of subsets of $E(S)$, this also implies that it is an isomorphism of algebras.

(3) Now let $D\in E(R,S)$. By definition, $\alpha(D(f_0))=\frac{|W_H|}{|W|}\sum_i D(f_0)(g_i)$ with $D(f_0) \in F(R)$ from which it follows that $D(f_0)(g_i)\in R$ for each $g_i$. Thus  $\alpha(D(f_0))\in R$ and therefore $\varphi(eDe) \in e_0Re_0$. 

Conversely,  let $a\in e_0Re_0$. If $f\in F(S)$, we can  define $D' (f)\in F(R)$ by $D'(f)(w)=af(w)$.  A simple calculation shows that $D'\in E(R,S)$ with $\varphi(eD'e)=a$, as required. 
\end{proof}

\begin{lemma}\label{lem:BEembedlocal}
        Set $F = F(\dd(\h_{\reg}) \rtimes W_H)$ and $E = E(\dd(\h_{\reg})
        \rtimes W_H)$. For $f \in F, u,w \in W$ and $D \in \dd(\h_{\reg})$,
        the rules  
        $$
        \Psi(D \o 1)(f)(w) = w(D) f(w), \quad \Psi(1 \o u)(f)(w) = f(wu)
        $$
        define an  injective algebra map $\Psi \colon \dd(\h_{\reg}) \rtimes W \to E$. Moreover   
        $\Psi$ induces a commutative diagram 
        $$
        \begin{tikzcd}
                e(\dd(\h_{\reg}) \rtimes W) e \ar[r,hook,"e \Psi e"] & e
                E e \ar[r,"\cong"',"\varphi"] & e_0(\dd(\h_{\reg})
                \rtimes W_H) e_0  \\ 
                \dd(\h_{\reg})^W \ar[rr,hook] \ar[u,"\cong","j"'] &&
                \dd(\h_{\reg})^{W_H} \ar[u,"\cong","j_0"'],
        \end{tikzcd}
        $$ 
        where $j(D)= e D e$ and  $j_0(D)= e_0 D e_0$.  
   \end{lemma}

  \begin{proof}       The fact that $\Psi$ is a well-defined ring homomorphism is a routine
      computation that is given in Appendix~\ref{App-C}.  By   \cite[Theorem~2.5 and Corollary 2.6]{Mo}, the maps $j$ and $j_0$ are isomorphisms. The same results show that  the ring   $e(\dd(\h_{\reg}) \rtimes W) e$ is simple and so     the nonzero  morphism $\Psi$ is injective.      The  isomorphism $\varphi$ is given by  Lemma~\ref{lem:varphi}(2)     and then the fact that the diagram commutes is another  direct computation, which is again given in the appendix.  \end{proof}

For $y \in \h$, we temporarily write  $T_y^W$ for  the Dunkl embedding of
$y \in H_{\kappa}(W)$ in $\dd(\h_{\reg}) \rtimes W$  defined by
\eqref{eq:Dunkl}. 
 We denote by the same letter $\kappa$ the
restriction of $\kappa$ to $W_H$ and set $\Ak(W_H) = e_0 H_{\kappa}(W_H) e_0$, for the corresponding Cherednik algebra $\Hk(W_H)$. Let
$T_y^{W_H}$ denote the Dunkl embedding of $y \in \h \subset H_{\kappa}(W_H)$
in $\dd(\h_{\reg}) \rtimes W_H$; thus \eqref{eq:Dunkl} simplifies to give:
$$
T_y^{W_H} = {\partial_y} + \frac{\langle y,\alpha_H \rangle}{\alpha_H} \sum_{i
        = 0}^{\ell_H - 1} \ell_H \kappa_{H,i} e_{H,i}.
$$

Recall from~\eqref{eq:Hdiscriminant} that  
$\deltah_H = \alpha_H^{-\ell_H} \deltah\in \deltah_H \in \C[\h]^{W_H} \subset A_\kappa(W_H)$. This implies that $\deltah_H $ acts locally ad-nilpotently on $A_\kappa(W_H)$ and we can
therefore form the localisation $
\Ak(W_H)_{\deltah_H} = \Ak(W_H)[\deltah_H^{-1}]
$    at the associated   Ore set $\euls{C}=\{\deltah_H^n\}$.  Write 
$$
E(H,\dd) := E(H_{\kappa}(W_H)_{h_H},\dd(\h_{\reg}) \rtimes W_H), 
$$
which is a subalgebra of $E = E(\dd(\h_{\reg}) \rtimes W_H)$.

\begin{lemma}\label{prop:sphericalinrankone1}
        $\Psi(H_{\kappa}(W)) \subseteq E(H,\dd)$ and $\varphi(e E(H,\dd) e) =
        j_0(\Ak(W_H)_{\deltah_H})$.  
\end{lemma}

\begin{proof} 
       This is similar to  \cite[Theorem~3.2]{BE}. The key to the proof is to show that, for $y \in \h$ and $w \in W$,
one has 
        \begin{equation}\label{eq:imagedunkl}
                \left(\Psi\left(T_y^W\right) f\right)(w) = T_{w(y)}^{W_H} f(w) +
                \sum_{\begin{smallmatrix} H'\in \calA \\ H' 
                                \neq H\end{smallmatrix}} \frac{\langle w(y), \alpha_{H'}
                        \rangle}{\alpha_{H'}} \sum_{i = 0}^{\ell_{H'} - 1}
                \ell_{H'} \kappa_{{H'},i} f( e_{{H'},i} w),
        \end{equation}
        where, with a slight abuse of  notation, we write 
        \[f( e_{{H'},i} w) := \frac{1}{\ell_{H'}} \sum_{s \in
                W_{H'}} \mr{det}_{\h}(s)^i f(sw).\]
        
      The remark after   \cite[Theorem~3.2]{BE} argues that one can deduce their   analogue of 
      \eqref{eq:imagedunkl}  from general results about sheaves of Cherednik algebras, and this allows one to guess the formula in question. Similar arguments apply to \eqref{eq:imagedunkl}, but since the resulting computations are still intricate, we prefer to give a detailed proof. This is given in Appendix~\ref{App-C}.

         Once     \eqref{eq:imagedunkl} has been established, the proof of the lemma is easy.    Indeed, note that every
        denominator  $\alpha_{H'}$ appearing on the right hand side of \eqref{eq:imagedunkl} is a
        factor of $\deltah_{H}$. Therefore, $\Psi\left(T_y^W\right)$ belongs to the
        subalgebra $E(H,\dd)$ of $E(\dd(\h_{\reg}) \rtimes W_H)$. Since
        $H_{\kappa}(W)$ is generated by $\C[\h], W$ and the Dunkl operators
        $T_y^W$, it follows that $\Psi(H_{\kappa}(W)) \subseteq E(H,\dd)$.  
        
        Finally, by Lemma~\ref{lem:varphi}(3), 
        \[\varphi(e E(H,\dd) e)= e_0
        H_{\kappa}(W_H)_{\deltah_{H}} e_0 = j_0(\Ak(W_H)_{\deltah_H})\] inside $e_0
        (\dd(\h_{\reg}) \rtimes W_H) e_0$.  
\end{proof}

We can now give the   inclusion of algebras mentioned at the beginning of the subsection. 

\begin{proposition}\label{prop:sphericalinrankone}
        Let  $H \in \calA$.  Under the embedding $\dd(\h_{\reg})^W
        \hookrightarrow \dd(\h_{\reg})^{W_H}$, the image  of $\Ak(W) $  is
        contained in $\Ak(W_H)_{\deltah_H}$. 
\end{proposition}

\begin{proof}  First note 
        that, by Lemma~\ref{prop:sphericalinrankone1},
        \[
        (e\Psi e)( j(\Ak(W)))  =   (e \Psi e)(e H_{\kappa}(W) e) \subseteq e
        E(H,\dd) e = \varphi^{-1}(j_0(\Ak(W_H)_{\deltah_H})).
        \]
        By   Lemma~\ref{lem:BEembedlocal}, this pulls back to the  desired inclusion 
        $\Ak(W)  \subseteq  \Ak(W_H)_{\deltah_H}.$
\end{proof}

\begin{notation}\label{defn:hcirc} Let $\h^{\circ} = \h \smallsetminus
        \bigcup_{H_\alpha \neq H_\beta \in \calA} H_{\alpha} \cap H_{\beta}$; equivalently,
        $\h^{\circ}$ is the set of points of $\h$ that lie on at most one
        hyperplane. The complement to $\h^{\circ}$ in $\h$ has codimension two.
\end{notation}

\begin{theorem}\label{thm:intersectsphericalinddhreg} In
        $\dd(\h_{\reg})^W$,
        $$
        \Ak(W) = \bigcap_{H \in \calA} \Ak(W_H)_{\deltah_H}.
        $$
\end{theorem}

\begin{proof} By Proposition~\ref{prop:sphericalinrankone},  
        \[\Ak(W) \ \subseteq \bigcap_{H \in
                \calA} \Ak(W_H)_{\deltah_H},\] as subalgebras of $\dd(\h_{\reg})$.  
        By Remark~\ref{rem:orderfiltrationH} and
        Lemma~\ref{lem:quotientetaleh}, the order filtrations on both $\Ak(W)$ and
        $\Ak(W_H)_{\deltah_H}$ are given by restriction of the order filtration on
        $\dd(\h_{\reg})$.  Thus the inclusion of $\Ak(W)  \hookrightarrow\bigcap_{H
                \in \calA} \Ak(W_H)_{\deltah_H}$ is   filtered and so it suffices to
        show that the associated graded map is  an equality.
        
        The associated graded  map
        is
        $$
        \C[\h \times \h^*]^W  \ \hookrightarrow \  \bigcap_{H \in \calA} \C[\h
        \times \h^*]_{\deltah_H}^{W_H}.
        $$
        Let $f$ be an element in the right hand side. Then $f \in  X:=
        \bigcap_{H \in \calA} \C[\h \times \h^*]_{\deltah_H}$ and so $f$
        is regular on $\h^{\circ} \times \h^*$. Since the complement to $\h^{\circ}
        \times \h^*$ has codimension two in $\h \times \h^*$, it follows that 
       $X = \C[\h \times \h^*]$. Moreover,  $f \in \C[\h \times
        \h^*]_{\deltah_H}^{W_H}$ for each $H\in \calA$ and $\deltah_H$ is a $W_H$-invariant. Thus, $f \in
   X^{W_H}=     \C[\h \times \h^*]^{W_H}$. Since the group $W$ is generated by all  the $W_H$,
        it follows that $f \in \C[\h \times \h^*]^W$.
\end{proof}

We note an interesting consequence of Theorem~\ref{thm:intersectsphericalinddhreg}. First, we recall that to the symplectic quotient singularity $(\h \times \h^*)/W$ one can associate Namikawa's Weyl group $\Gamma$, which plays an important role in the Poisson deformation theory of the singularity. In this particular case, it is shown in \cite[Lemma~4.1]{BellSchedThiel} that $\Gamma = \prod_{[H] \in \calA/W} \mf{S}_{\ell_H}$ is a product of symmetric groups acting on the permutation representation with basis $\{ \kappa_{H,i} : [H] \in \calA/W, \ 0 \le i \le \ell_H-1 \}$. We think of $\boldsymbol{\kappa}_{H,i}$ as a variable and let $\C[\boldsymbol{\kappa}]$ be the polynomial ring in the variables $\boldsymbol{\kappa}_{H,i}$. Then $\Gamma$ acts on $\C[\boldsymbol{\kappa}]$. We define a twisted action of $\Gamma$ on $\C[\boldsymbol{\kappa}]$ by setting 
$$
\sigma \star \boldsymbol{\kappa}_{H,i} = \boldsymbol{\kappa}_{H,\sigma(i)} + \frac{\sigma(i)-i}{\ell} +(\delta_{\sigma(i),0} - \delta_{i,0}) 
$$
Consider the generic spherical algebra $A_{\boldsymbol{\kappa}} \subset \dd(\h_{\reg})^{W}[\boldsymbol{\kappa}]$; it is a $\C[\boldsymbol{\kappa}]$-algebra. Make $\Gamma$ act on the algebra $\dd(\h_{\reg})^{W}[\boldsymbol{\kappa}]$ by twisted action $\star$ on the coefficients $\C[\boldsymbol{\kappa}]$ and trivially on $\dd(\h_{\reg})^{W}$. The following should be viewed as the analogue of the Harish-Chandra isomorphism (of the centre of the enveloping algebra of a reductive Lie algebra) for spherical Cherednik algebras.  

\begin{corollary}\label{cor:HCisoRCA}
        The subalgebra $A_{\boldsymbol{\kappa}}$ is stable under $\Gamma$ and the centre of $A_{\boldsymbol{\kappa}}^{\Gamma}$ equals $\C[\boldsymbol{\kappa}]^{(\Gamma,\star)}$. 
\end{corollary}

\begin{proof}
The proof of Theorem~\ref{thm:intersectsphericalinddhreg} goes through verbatim with the complex parameters $\kappa$ replaced by the variables $\boldsymbol{\kappa}$. Fix a hyperplane $H \in \calA$ and think of $A_{\boldsymbol{\kappa}}(W_H)$ as a subalgebra of $\dd(\h_{\reg})^{W}[\boldsymbol{\kappa}]$ as usual. The key point about the twisted action is that it follows from Example~\ref{ex:Dunklcyclci} that 
$$
\Phi_{}(Y)  = \frac{\ell^{\ell}}{z} \prod_{i =
        0}^{\ell-1} \left( z \partial_z  +
\boldsymbol{\kappa}_{H,i} + \frac{i}{\ell} +(\delta_{i,0}
- 1) \right)     
$$
and $\Phi(Z) = \ell z \partial_z$ belong to $A_{\boldsymbol{\kappa}}(W_H) \cap (\dd(\h_{\reg})^{W}[\boldsymbol{\kappa}])^{\Gamma}$. That is, $\sigma \star \Phi(Y) = \Phi(Y)$ for all $\sigma \in \Gamma$. Since $A_{\boldsymbol{\kappa}}(W_H)$ is generated as a $\C[\boldsymbol{\kappa}]$-algebra by $\C[\h]^{W_H}$, $\Phi(Z)$ and $\Phi(Y)$, it follows that $A_{\boldsymbol{\kappa}}(W_H) $
is generated by $A_{\boldsymbol{\kappa}}(W_H) \cap (\dd(\h_{\reg})^{W}[\boldsymbol{\kappa}])^{\Gamma},$ and $ \C[\boldsymbol{\kappa}] $ as a $\C$-subalgebra of 
$ \dd(\h_{\reg})^W[\boldsymbol{\kappa}].$
This implies that both $A_{\boldsymbol{\kappa}}(W_H)$ and $A_{\boldsymbol{\kappa}}(W_H)_{h_H}$ are stable under $\Gamma$. It follows from Theorem~\ref{thm:intersectsphericalinddhreg} that $A_{\boldsymbol{\kappa}}$ is also stable under $\Gamma$.  

Finally, we note that the centre of $(\dd(\h_{\reg})^{W}[\boldsymbol{\kappa}])^{\Gamma}$ equals $\C[\boldsymbol{\kappa}]^{(\Gamma,\star)}$ and the former is the localization of $A_{\boldsymbol{\kappa}}^{\Gamma}$ with respect to the powers of $\delta$. It follows that the centre of $A_{\boldsymbol{\kappa}}^{\Gamma}$ equals $\C[\boldsymbol{\kappa}]^{(\Gamma,\star)}$ too.  
\end{proof}

Returning to complex valued parameters, we see that:  
  
\begin{corollary}\label{cor:Namparameters}
        For each $\sigma \in \Gamma$, $A_{\sigma \star \kappa}(W) = \Ak(W)$ as subalgebras of $\dd(\h_{\reg})^W$. 
\end{corollary}  

\begin{proof}
        This is similar to the first part of the proof of Corollary~\ref{cor:HCisoRCA}. 
\end{proof}

Corollary~\ref{cor:Namparameters} is a generalization of \cite[Proposition~5.4]{BerestChalykhQuasi}, since the group $G$ of that proposition can be realised as a subgroup of $\Gamma$, with the same twisted action on parameters. 
          
\section{Background on polar representations} \label{Sec:polarreps}

In this section we give  the basic definitions and results  for polar
representations, for which our main  reference is \cite{DadokKac}.

 We begin with the relevant notation. 
Fix a  reductive group $G$ with a finite
dimensional representation $V$. Let  $\pi : V \rightarrow V\git G$  be the categorical
quotient; that is, the morphism defined by the inclusion of algebras
$\C[V]^G \to \C[V]$.
An element $v \in V$ is called \textit{semisimple} if the orbit $G \cdot v$
is closed. In this case $G \cdot v$ is affine and the \emph{stabiliser} $G_v$ is a reductive
subgroup of $G$.  A semisimple element $v$ is
\textit{$V$-regular}\label{defn:h-reg} if $\dim G \cdot v$ is maximal
amongst orbits of semisimple elements. The set of semisimple elements, respectively
 $V$-regular
elements in $V$ is denoted  by  $V_{\mathrm{ss}}$, respectively $V_{\st}$.  
Set
\begin{equation}\label{eq:msdefnorbit}
        m=\max \{\dim G\cdot v : v \in V\} \ \; \text{and} \ \;
        s= \max \{\dim G \cdot x : x \in V_{\mathrm{ss}}\}. 
\end{equation}
Thus $s \le m \le \dim V - \dim V \git G$.  
 
Pick a  $V$-regular element $v$ and define
$\h_v = \{ x \in V \ | \ \mf{g} \cdot x \subseteq \mf{g} \cdot v
\}$. Then \cite[Lemma~2.1 and Proposition~2.2]{DadokKac} say that
$\h_v$ consists of semisimple elements and the map
$\pi:  \h_v \rightarrow V\git G$ is finite. In particular,
$\dim \h_v \le \dim V \git G$. If $\dim \h_v = \dim V \git G$ for some such $v$, 
then $V$ is called \textit{polar} \label{defn:polar} and
$\h_v$ is a \textit{Cartan subspace} of $V$. \label{defn:Cartan}
If $G^\circ$ is the connected component of the identity, the
representation $(G,V)$ is polar if and only if $(G^\circ, V)$ is
polar,  see \cite[Remark, p.~510]{DadokKac}.

\emph{We will assume for the rest of the paper
that the group $G$ is connected.}

\smallskip

Assume that $V$ is polar, and fix a Cartan subspace $\h$ of $V$. Let 
\begin{equation}\label{N(h)-defn}
N_G(\h) = \{ g \in G \ | \ g\cdot \h = \h \} \ \text{ and } \
Z_G(\h) = \{ g \in G \ | \ g\cdot x = x, \ \forall  x \in \h \}. 
\end{equation}
Then $W = N_G(\h) / Z_G(\h)$ is a finite group by
\cite[Lemma~2.7]{DadokKac}, called the \textit{Weyl
  group}\label{defn:Weyl} of $V$. By
\cite[Theorem~2.10]{DadokKac}, $W$ is a complex reflection group
and the map $\h \rightarrow V \git G$ factors to give an
isomorphism $\h / W \stackrel{\sim}{\longrightarrow} V \git G$
\cite[Theorem~2.9]{DadokKac}.
 Write $\rr \colon \C[V] \to \C[\h]$ for restriction of functions; thus $\rr$ induces a graded
isomorphism $ \C[V]^G \stackrel{\sim}{\longrightarrow} \C[\h]^W$, which will again be written $\rr$.

As in \cite[page~515]{DadokKac},   the \textit{rank} of
the polar representation $(G,V)$ is defined to be
\begin{equation}
  \label{defn:polarrank}
  {\rank}(G,V)= \dim \h - \dim \h^{\mf{g}}.
\end{equation}
Recall from Notation~\ref{not:regularelementh} that
$x \in \h$ is regular if the stabiliser of $x$ in $W$ is
trivial and that the set regular elements in
$\h$ is denoted $\h_{\mathrm{reg}}$.

 \begin{lemma}\label{lem:stabilizerpolarh}
        If $v \in V$ is $V$-regular and $\h = \h_v$ then $G_v
        \subseteq N_G(\h)$. Hence, if $u \in \h$ is regular then
        $\h= \h_u$ and
        $\g \cdot \h=\g \cdot u = \g \cdot v$ with $G_u = Z_G(\h)$.  
      \end{lemma}

\begin{proof}
 By definition,  $\h = \h_v = \{ x \in V \, | \, \g \cdot x \subseteq \g \cdot v \}$. If
  $g \in G_v$ then $g \cdot (\g \cdot v) = \g \cdot v$ which implies
  that   $\g \cdot (g \cdot x) = g\cdot ( \g \cdot x) \subseteq \g \cdot v$ for all $x\in \h$ and
  hence $g\cdot  x \in \h$. Thus $G_v \subseteq N_G(\h)$.
        
  If $u \in \h$ is regular then \cite[Theorem~2.12]{DadokKac} implies
  that it is $V$-regular. In particular,
  $\dim G \cdot v = \dim G \cdot u$. Since
  $\mf{g} \cdot u \subseteq \mf{g} \cdot v$, it follows that
  $\g \cdot u = \g \cdot v$ and hence $\h_u = \h$. By the
  previous paragraph, $G_u \subseteq N_G(\h)$. The image of $G_u$
  in the quotient $W = N_G(\h) / Z_G(\h)$ is $W_u$. The latter is
  trivial by definition of regularity. Thus, $G_u \subseteq
  Z_G(\h)$. Conversely, $Z_G(\h) \subseteq G_x$ for all
  $x \in \h$.
\end{proof}

  \begin{remark}\label{regular-remark}
    An element $x \in \h$ can be regular, or it can be
    $V$-regular when thought of as an element of $V$. It is
    conjectured that these concepts coincide; see
    \cite[p.~521]{DadokKac}.
  \end{remark}
 
\begin{notation}\label{notation:h-reg}
  Since $W$ is a complex reflection group, the open set
  $\h_{\reg} / W$ is principal; it is the non-vanishing locus of
  the discriminant $\deltah \in \C[\h]^W$ defined
  in~\eqref{eq:discriminant}. We define the \emph{regular locus} 
  $V_{\reg}$ to be the preimage of $\h_{\reg} / W$ under the
  quotient map $\pi : V \rightarrow  V\git G \cong \h / W$.   Equivalently,
   $V_{\reg}=(\delta\not=0)$ where
\begin{equation} \label{deltav}
\deltav= \rr^{-1}(\deltah) \in \C[V]^G
\end{equation}
is  defined to be the \emph{discriminant} of
$(G,V)$.
\end{notation}

Finally, we recall that there exists a Hermitian
inner product $( - , - )$ on $V$,
invariant under the action of a maximal compact subgroup  of $G$,
such that every vector in $\h$ has minimal length in its orbit
and $\h$ is perpendicular to $\mf{g} \cdot v$ for all $v \in \h$. 
See \cite[Lemma~2.1 and Remark~1.4]{DadokKac} for the details.

We can take slices in polar representations. Let $p \in \h$
and recall that $p$ is semisimple. 
 Define the \emph{slice} $S_p$ at $p$   to be the orthogonal complement, with respect to $(-,-)$, 
in $V$ to $\mf{g} \cdot p  $.   Then $G_p$ acts on $S_p$, and $(G_p,S_p)$ is again polar, with Cartan subspace $\h$; see \cite[Lemma~2.1(ii) and
Theorem~2.4]{DadokKac} for the details.

As in \cite[Corollary~2.5]{DadokKac},  we can pick   a
$Z_G(\h)$-invariant  complement $U$ to $\h \oplus \mf{g} \cdot \h$ in
$V$. Thus the decomposition
\begin{equation}
  \label{eq:Ucomplement}
  V = \h \oplus \g \cdot \h \oplus U
\end{equation}
is $Z_G(\h)$-invariant and it is not difficult to see that this decomposition
is $N_G(\h)$-invariant. Furthermore, since Lemma~\ref{lem:stabilizerpolarh} implies that $G_u=Z_G(\h)$ for a regular element $u\in \h$, the proof of 
\cite[Corollary~2.5]{DadokKac} implies that  $U \git Z_G(\h) = \{ \mr{pt}
\}$.

\begin{lemma}
                \label{lem:slices}
                Let $p \in \h$  with slice representation  $(G_p,S_p)$.
   \begin{enumerate}
   \item             One has $S_p= \h \, \oplus \, \g_p\cdot \h \oplus U$. 
     \item
  $p$ is $V$-regular  $\iff$  $\g_p = Z_\g(\h)$ $\iff$
                $\rank(G_p,S_p) = 0$ $\iff$ $S_p = \h \oplus U$.
                \end{enumerate}
        \end{lemma}

\begin{proof}
        (1) The inclusion $(\h +   \g_p\cdot \h) \oplus U \subseteq S_p$ follows from
          \cite[Eq.~1, p.511]{DadokKac}. Pick $x  \in V_{\st}$
        such that $\h= \h_x$ and hence $\g\cdot \h =\g\cdot x$. Then, as in the proof of 
        \cite[Theorem~2.4]{DadokKac},  
        $\g \cdot x= \g \cdot p \oplus \g_p\cdot x$ {and} $\g_p\cdot x = \g_p \cdot \h$.  
        We then get:
      \begin{align*}
             \dim \h +  \dim \g\cdot p \,\, +  & \dim \g_p \cdot \h    \ = \  \dim \h + \dim \g
        \cdot x   \ = \\
        & =\  \dim \h + (\dim V - \dim \h - \dim U)  \ = \  \dim V - \dim U. 
        \end{align*}  
      Since $V= \g \cdot p  \oplus  S_p$ and thus $\dim S_p = \dim V - \dim \g \cdot p$, it follows that 
      $S_p= \h \, \oplus \, \g_p\cdot \h \oplus U$.
       
        \smallskip
       (2)  Define $s$ by \eqref{eq:msdefnorbit} and note that $s = \dim \g - \dim
       Z_\g(\h)$ by Lemma~\ref{lem:stabilizerpolarh}. We have $Z_\g(\h) \subseteq \g_p$ for all $p \in \h$. From $\dim G \cdot p = \dim \g - \dim \g_p$ we deduce:
         \begin{align*}  p\in V_{\st} \iff 
        \dim& G\cdot p= s \iff \dim \g_p = \dim Z_\g(\h)  \\
        & \iff \g_p =
        Z_\g(\h) \iff \g_p \subseteq Z_\g(\h). 
      \end{align*}  
        This gives the first equivalence.
        
        By definition, $\rank(G_p,S_p) = 0$ is equivalent to $\h=
        \h^{\g_p}$. In other words, $\rank(G_p,S_p) = 0\iff \g_p \subseteq Z_\g(\h)$, giving  the
        second equivalence.

        Finally, from Part~(1), $S_p = \h \oplus \g_p\cdot \h \oplus U
        = \h \oplus U \iff\g_p\cdot \h= 0$, which  certainly implies that $\g_p \subseteq Z_\g(\h)$.
       Conversely, if $p\in V_{\st}$, then $\h=\h_p$ and so   $\g_p\cdot \h=0$ follows from \cite[Lemma~2.1(iii)]{DadokKac}. This gives the final equivalence.
         \end{proof}

\smallskip  
We end the section with some comments on the discriminant $\deltav$.
Factorise
\begin{equation}\label{eq:factordiscinV}
  \deltav = \deltav_1^{m_1} \cdots \deltav_k^{m_k}
\end{equation}
in $\C[V]$, where the $\deltav_i$ are pairwise coprime irreducible homogeneous
polynomials. Since we have assumed that $G$ is
connected, the $\deltav_i$ are semi-invariants, say of weight $\theta_i$,  and the highest common factor
of the $m_i$ is one. In general, one can have $m_i > 1$; this happens, for example, in \cite[Example~15.1]{BNS}.

\begin{remark}   \label{discriminant-factors}
  We assume in this remark that the group $G$ is semisimple. Then the space of linear characters
  $\mathbb{X}^*(G)$   is trivial and any semi-invariant polynomial is
  $G$-invariant.  It follows that the decomposition
  $\deltav= \prod_j \deltav_j^{m_j}$ coincides with the
  decomposition of $\deltav$ into a product of irreducible polynomials in the
  polynomial algebra $\C[V]^G$.  Since
  $\rr : \C[V]^G \to \C[\h]^W$ is an isomorphism, the factors
  $\deltav_j$ are the $\rr^{-1}(h_j)$ where $\prod_j h_j^{m_j}$
  is the decomposition of the discriminant
  $\deltah = \prod_{H \in \calA} \alpha_H^{\ell_H}$ as product of
  irreducible elements in $\C[\h]^W$. This
  decomposition depends on the number of $W$-orbits of
  hyperplanes in~$\calA$ (see 
  \cite[Theorem~4.18]{Broue}).

As an explicit example, suppose  that the group $W$ is the  Weyl group of an
irreducible root system $R$.
In this case, there is one orbit of hyperplanes in $\calA$ if $R$
is simply laced  and two orbits  otherwise.  
Then, $\deltav$ is irreducible when $R$ is simply laced.  If $R$ is not simply laced,
choose a set of positive roots $R_+$ and set 
 \[
h_{\mathrm{sh}}= \prod_{\alpha \in R_+,
  \,    \alpha \, \mathrm{short}} \alpha^2,  \quad \text{and} \quad
h_{\mathrm{lg}}= \prod_{\alpha \in R_+,
\,    \alpha \, \mathrm{long}} \alpha^2,
\]
 and write $\deltav_{\mathrm{sh}} = \rr^{-1}(h_{\mathrm{sh}})$ and $
\deltav_{\mathrm{lg}} = \rr^{-1}(h_{\mathrm{lg}})$. 
  Then  $\deltav = \deltav_{\mathrm{sh}} \deltav_{\mathrm{lg}} \in
\C[V]$ where $\deltav_{\mathrm{sh}}$ and $\deltav_{\mathrm{lg}}$ are 
  irreducible.
 \end{remark}


\section{The radial parts map: the one-dimensional case}\label{sec:radialpartsmap1}

In this section, we define the radial parts map $\rad_{\vs}$. We consider in detail the case where $\dim V \git G = 1$ and show that the image of $\rad_{\vs}$ is always a spherical subalgebra of a rational Cherednik algebra in this case. 

\subsection*{The radial parts map}

Fix a polar representation $V$ for a connected, reductive
group $G$. As in \eqref{eq:factordiscinV},
$\deltav_1, \ds, \deltav_k$ denote the pairwise distinct
irreducible factors of $\deltav$ in $\C[V]$. The $\deltav_i$ are
homogeneous semi-invariant functions of weight $\theta_i$ that
are invertible on $V_{\reg}$.  

For each $i$, choose $\vs_i \in \C$ and set
\begin{equation} \label{chi-defn}
        \chi = \sum_{i=1}^k \vs_i d \theta_i .  
        \end{equation}
Recall that a  \textit{weakly equivariant}\label{defn:weak}  left $\dd(V_{\reg})$-module is a left $\dd(V_{\reg})$-module $M$ that is also a rational $G$-module such that the action
  $\dd(V_{\reg}) \otimes M \rightarrow M$ is $G$-equivariant. 
Since each $\delta_i$ is a $G$-semi-invariant, one can define a weakly $G$-equivariant $\dd(V_{\reg})$-module, denoted by $\C[V_{\reg}] \deltav^{\vs}$, as follows. As a $\C[V_{\reg}]$-module it is free of rank one with basis $\deltav^{\vs} := \deltav_1^{\vs_1} \cdots
\deltav_k^{\vs_k}$. The group $G$ acts trivially on the generator
$\deltav^{\vs}$ and
\begin{equation}\label{eq:twistedop}
        \partial \ast \deltav^{\vs} = \sum_i \vs_i
        \frac{\partial(\deltav_i)}{\deltav_i} \deltav^{\vs}, 
\end{equation}
for each derivation $\partial \in \dd(V_{\reg})$. The action of a
general differential operator then follows from the fact that
$\dd(V_{\reg})$ is generated by derivations and $\C[V_{\reg}]$.
We remark that, here and elsewhere, we often use $\ast$ to denote
the action of a differential operator on functions to distinguish
it from multiplication of operators.  One can also regard
$\C[V_{\reg}] \deltav^{\vs}$ as the rank one integrable connection
defined by the logarithmic one-form $d \log \deltav^{\vs}$.

Denote by $D \mapsto \deltav^{-\vs}D\deltav^{\vs}$ the conjugation by
$\deltav^{-\vs}$; that is,~the automorphism of
$\dd(V_{\reg})$ defined by:
\begin{gather*}
        \deltav^{-\vs}f\deltav^{\vs} = f \quad \text{if $f \in \C[V_{\reg}]$},
        \\
        \deltav^{-\vs}v\deltav^{\vs} = v + \sum_{i=1}^k \vs_i
        \frac{v(\deltav_i)}{\deltav_i}   \quad \text{if $v \in
                \mathrm{Der} \bigl(\C[V_{\reg}]\bigr)$.}
\end{gather*}
Then, $D \in \dd(V_{\reg})$ acts on
$f\deltav^{\vs} \in \C[V_{\reg}] \deltav^{\vs}$ by
$f\deltav^{\vs} \mapsto (\deltav^{-\vs}D\deltav^{\vs})(f)
\deltav^{\vs}$.  We also adopt the notation
$(\deltav^{-\vs}D\deltav^{\vs})(f)=
\deltav^{-\vs}D(f\deltav^{\vs})$.
Recall from  Notation~\ref{notation:h-reg}  that $ \deltah = \rr(\delta)$ is the discriminant
on $\h$ and set  $h_i := {\deltav_i}_{\mid \h}$ and
$\deltah^{\vs} := h_1^{\vs_1} \cdots h_k^{\vs_k}$. The analogous definitions
will be given for the rank one $\dd(\h_{\reg})$-module
$\C[\h_{\reg}] \deltah^{\vs}$ and the automorphism
$D \mapsto \deltah^{-\vs}D \deltah^{\vs}$ of $\dd(\h_{\reg})$  given by
conjugation by $\deltah^{-\vs}$.

Recall that  the  morphism $\rr \colon \C[V] \to \C[\h]$, given by restriction
$f \mapsto f_{\mid \h}$,   induces graded isomorphisms
$\C[V]^G \stackrel{\sim}{\longrightarrow} \C[\h]^W$ and
$\C[V_{\reg}]^G \stackrel{\sim}{\longrightarrow} \C[\h_{\reg}]^W$, which are also denoted~$\rr$.

Since the module $\C[V_{\reg}] \deltav^{\vs}$ is weakly $G$-equivariant, the algebra $\dd(V)^G$ preserves the subspace
$(\C[V_{\reg}] \deltav^{\vs})^G = \C[V_{\reg}]^G \deltav^{\vs}$ and
we can identify the latter, via $\rr$, with the rank one free
$\C[\h_{\reg}/W]$-module generated by $\deltah^{\vs}$. Now each
$z \in \C[\h_{\reg}/W] \cong \C[V_{\reg}]^G$ acts ad-nilpotently
on $\dd(V)^G$. Thus, combined with Lemma~\ref{lem:quotientetaleh}(2), this identification defines a map
\begin{equation}\label{eq:radialpartsmap}
        \rad_{\vs} \colon \dd(V)^G \longrightarrow \dd(\h_{\reg}/W)\ =\
        \dd(\h_{\reg})^W,  
\end{equation}
which is \emph{the radial parts map} in this context. Explicitly,
\begin{equation}\label{eq:radialpartsmap2}
        \rad_{\vs} (D)(z) = \bigl(\deltav^{-\vs}D(  \rr^{-1}(z)
        \deltav^{\vs}) \bigr) |_{\h_{\reg}}\quad \text{for } z \in
        \C[\h_{\reg}]^W \text{ and }  
        D\in \dd(V)^G.
\end{equation}

 \begin{remark} \label{eq:g-chi}
      Let $\chi$ be as in~\eqref{chi-defn} and set
      \[       \g_{\chi} = \{ \tau(x) - \chi(x) \, | \, x \in \g
      \} \subset \dd(V). 
     \]
      Then $\g_{\chi}$ annihilates $\C[V_{\reg}]^G
      \deltav^{\vs}$ and hence the two-sided ideal $(\dd(V)
      \g_{\chi})^G$ of $\dd(V)^G$  is contained in the kernel of
      $\rad_{\vs}$.    
\end{remark}

 \subsection*{The one-dimensional case}
  
 In this subsection $K$ will be a reductive group, which need
 \textit{not} be connected, and $S$ will denote a
 finite-dimensional, faithful $K$-module such that
 $\dim S \git K = 1$. Much of this subsection parallels
 \cite{Le3}, although that paper is not directly applicable since
 \cite{Le3} only considers the case where $\vs=0$ and $K$
 connected.
 
 \begin{lemma}\label{lem:rankonepolar}   
        The pair $(K,S)$ is a polar representation with
        $\C[S]^K = \C[u]$ for some homogeneous polynomial $u$. The rank
        of $(K,S)$ is at most one.
 \end{lemma}
 
 \begin{remark}\label{rem:u} This choice of polynomial $u$ will be
        fixed throughout the subsection.
 \end{remark}
 
 \begin{proof}
        Since $\dim S \git K > 0$, there exists some non-zero semisimple element in $S$. Therefore, 
        there exists a non-zero $V$-regular element $v \in S$. Set $\h=\C v$. Then  it follows directly 
        from the definition that $\h$ is a Cartan
        subspace of $V$ and hence that $(G,V)$ is polar. The rest of
        the lemma follows from \cite[Theorem~2.9]{DadokKac}.
 \end{proof}
 
 Let $\ell = \deg u$ and write $\C[\h] = \C[x]$. Then
 $z := u |_{\h}$ is a homogeneous polynomial of degree $\ell$ and
 so we may assume that $z = x^{\ell}$. This implies that the Weyl
 group of $(K,S)$ is $W = \mathbb{Z} / \ell \mathbb{Z}$. The following result is presumably well-known, 
 but we do not know of a suitable  reference. 
 
 \begin{lemma}\label{lem:semiinvUFD}
        Let $L$ be an affine algebraic group acting linearly on a vector space $U$. Then the ring 
        $\C[U]^{[L,L]}$ is a unique factorisation domain. 
 \end{lemma}

 \begin{proof}
        Assume that $\C[U]^{[L,L]}$ is not a UFD. Then we can find  $u\in \C[U]^{[L,L]}$ having two distinct factorisations; say
        $$
        u = u_1^{p_1} \cdots u_r^{p_r} = v_1^{q_1} \cdots v_s^{q_s}.
        $$
         Moreover, we may assume that $\sum_i p_i$ is minimal. This means that, up to scalars, 
          $u_i \neq v_j$   for all $i,j$. Let $f_1$ be an irreducible factor of $v_s$ in $\C[U]$. Then $f_1$ divides 
          some $u_i$, say $f_1|u_r$. By           the proof of \cite[II 3.3, Satz 2]{KraftGeoBook},
           $f_1$ is a $L^{\circ}$-semi-invariant. The orbit $\{ [f_1], \ds, [f_m] \}$ of $[f_1]$ under $L/L^{\circ}$ in 
           $\mathbb{P}(\C[U])$ is finite. In other words, if $l \in L$ then $l \cdot f_i = c f_j$ for some 
           $c \in \C^{\times}$ and some $j$. This implies that $f_1 \cdots f_m$ is an $L$-semi-invariant dividing
            $v_s$. Since $v_s$ is irreducible in $\C[U]^{[L,L]}$, we have $v_s = \alpha f_1 \cdots f_m$ for some
             $\alpha \in \C^{\times}$. By the same argument $u_r = \beta f_1 \cdots f_m$. But this means we 
             can cancel one copy of $v_s = u_r$ in the factorisation of $u$, contradicting the minimality of $\sum_i p_i$. 
 \end{proof}

We say that a polynomial $f \in \C[S]$ is \textit{equivariantly irreducible} if $f$ is a $K$-semi-invariant and no 
proper factor of $f$ is a semi-invariant. Lemma~\ref{lem:semiinvUFD} says that every $K$-semi-invariant 
admits a unique   factorisation into equivariantly irreducible polynomials, up to scalar and permutation of factors. Factorise
 \begin{equation}\label{eq:ufactoru1tourun}
        u = u_1^{p_1} \cdots u_r^{p_r}
 \end{equation} 
into equivariantly irreducible polynomials.

 As in \cite{Le3}, fix a unitary inner product $(-,-)$ on $S$  
  and let $\{s_1, \ds, s_d\}$ be an orthonormal basis of $S^*$
 with dual basis $\{\partial_1, \ds, \partial_d\}$ for $S$. The
 form $( - , -)$ defines a sesquilinear isomorphism
 $\C[S] \stackrel{\sim}{\longrightarrow} \Sym \, S$ given
 explicitly by
 $$
 \sum_{i}a_i s^i \mapsto \sum_i \overline{a}_i \partial^i.
 $$
 If $f \in \C[S]$, write $f_*(D)\in \Sym S$ for the image of $f$. As shown
 in \cite[Proposition~2.21]{IntroPHV}, if $f \in \C[S]$ is a
 $\theta$-semi-invariant then $f_*(D)$ is a
 $\theta^{-1}$-semi-invariant (this part of the proof of \cite[Proposition~2.21]{IntroPHV} does 
 not require that $S$ is a prehomogeneous vector space). This implies that
 $(\Sym \, S)^G = \C[\Delta]$, where $\Delta := u_*(D)$.
 
 In order to incorporate the twist $\vs$ into our picture, we need
 to introduce a multivariable version of the $b$-function for
 $S$. We can consider a version of the module
 $\C[S_{\reg}]u^{\balp}$ defined over a polynomial ring
 $\C[\balp]=\C[\alpha_1, \ds, \alpha_r]$ where
 $$
 \partial \ast u^{\balp} = \sum_{i = 1}^r \alpha_i \frac{\partial
        (u_i)}{u_i} u^{\balp} \qquad\text{for every derivation }\
 \partial \in \Der\C[S].
 $$
 
 \begin{proposition}\label{thm:PHVbfunction}
        There exists a nonzero polynomial $b(\balp) = b(\alpha_1, \ds, \alpha_r)\in \C[\balp]$  
        of total degree $\ell$ such that
        \begin{equation}\label{eq:PHVbfunction}
                \Delta u^{\balp +1} = b(\balp) u^{\balp}.
        \end{equation}
 \end{proposition}

 \begin{proof}    We define  $u^{\balp}$ and each of the $\alpha_i$  to be $K$-invariant of degree zero.
        Then the  $\dd(S)[\balp]$-module $\C [S_{\reg}][\balp]u^{\balp}$ is $\Z$-graded and weakly $K$-equivariant since each $u_i$ is a $K$-semi-invariant.   Since 
        $\bigl(\C[S_{\reg}][\balp][u^{\balp}]\bigr)^K = \C[\balp][u^{\pm 1}]u^{\balp}$,
         the $K$-invariant, degree zero part $\bigl(\C[ S_{\reg}][\balp][u^{\balp}]\bigr)_0^K$ of $\C [S_{\reg}][\balp]u^{\balp}$ equals $\C[\balp]u^{\balp}$. But $\Delta u^{\balp +1}$ is $K$-invariant of degree zero. Therefore, there exists $b(\balp) \in \C[\balp]$ such that \eqref{eq:PHVbfunction} holds. 
        
        Repeating the argument given in the     paragraph preceding \cite[Proposition~2.22]{IntroPHV} shows     
        that $b$ is a  polynomial of total degree at most $\ell = \deg u$. Moreover, if we specialise this
        polynomial to $\alpha_1 = \cdots = \alpha_r = \alpha$, we must recover the
        polynomial given in \cite[Proposition~2.22(1)]{IntroPHV}; again,  this part of the proof of \cite[Proposition~2.22]{IntroPHV} does not use the fact that $S$ is a prehomogeneous space. This shows that the total degree of $b$ is exactly $\ell = \deg u$.
 \end{proof}
 
 \begin{remark}\label{rem:chiclinear} In general it appears difficult to
        compute the polynomial $b(\balp)$  but, based on \cite{Gyojamultib} we expect that
        $$
        b(\balp) \ = \ \prod_{i = 0}^{\ell-1} (a_{1,i} \alpha_1 + \cdots
        + a_{r,i} \alpha_r + c_i),
        $$
        where $a_{i,j} \in \mathbb{N}$ and $c_j \in
        \mathbb{Q}_{>0}$.  
 \end{remark}
 
\begin{notation}\label{notation:chidot}
   Choose $\vs_i \in \C$ for $1\leq i  \leq  r$ and set $\alpha_i = t - 1 + \vs_i$,
        for an indeterminate $t$. We factorise
        $$
        b(\balp) |_{\alpha_i \ = \ t - 1 + \vs_i} \ = \  c \prod_{i
                =0}^{\ell-1} (t + \lambda_i)
        $$   
        for some $\lambda_i \in \C$ and $c \in
        \C^{\times}$. Write $\mathbf{d} := \sum_{i = 1}^r \vs_i \deg u_i$. 
        
        Finally, recall that the \emph{Euler element} is
        $\eu_S = \sum_{i = 1}^n s_i \partial_{s_i}\in \dd(S)$ which is
        independent of the choice of basis $\{s_i\}$ of $S^*$.
 \end{notation}

 \begin{corollary}\label{cor:Imradrank1} Using the
        definitions from Notation~\ref{notation:chidot}, we have
        $$
        \rad_{\vs}(u) = z, \quad \rad_{\vs}(\eu_S) = \ell z
        \partial_z + \mathbf{d}, \quad  \text{and}\quad\rad_{\vs}(\Delta) = \frac{c}{z}
        \prod_{i = 0}^{\ell-1} (z \partial_z + \lambda_i).
        $$
 \end{corollary}
 
 \begin{proof}
        The equality $\rad_{\vs}(u) = z$ is immediate since
        $\rad_{\vs}$ is simply restriction on functions. An element of
        $\C[S_{\reg}]^K u^{\vs}$ is a linear combination of
        $u^m u^{\vs}$ for various $m \in \mathbb{Z}$. Since
        $\ell=\deg u$,
        $$
        \eu_S (u^{m} u^{\vs}) \ = \ (\ell m + \mathbf{d}) u^{m} u^{\vs}
        $$
        and hence $\rad_{\vs}(\eu_S) = \ell z \partial_z +
        \mathbf{d}$. If we set $\eu_{\vs} := \eu_S - \mathbf{d}$, then
        $\eu_{\vs}(u^{m} u^{\vs}) = \ell m u^{m} u^{\vs}$ and
        hence $\rad_\vs(\eu_\vs ) = \ell z \partial_z$.
        
        Finally, when restricting $\alpha_i$ to $m - 1 + \vs_i$ for
        each $i$, Proposition~\ref{thm:PHVbfunction} shows that
        $$\begin{aligned} 
                \Delta u^{m} u^{\vs} \ = \  b(\alpha)  u^{m-1} u^{\vs} \ = \
                &  \frac{1}{u} \left[
                c\prod_{i=0}^\ell(m+\lambda_i)\right] (u^{m} u^{\vs}
                )   \\ 
                &\quad \ = \ \frac{1}{u} \left[ c \prod_{i=0}^\ell(
                \ell^{-1}\eu_{\vs} + \lambda_i)\right] (u^{m}
                u^{\vs}).
        \end{aligned}
        $$
        Since $\ell^{-1}\rad_{\vs} \eu_{\vs}=z\partial_z$ this
        implies that
        $$
        (\rad_{\vs} \Delta)(z^{m}) \ = \ \frac{1}{z} \left[c
        \prod_{i=0}^\ell(z\partial_z+\lambda_i)\right] (z^{m}).
        $$
        Hence
        $\rad_{\vs}(\Delta) = \frac{c}{z} \prod_{i = 0}^{\ell-1}
        (z \partial_z + \lambda_i),$ as required.
 \end{proof}

 As shown in Example~\ref{ex:Dunklcyclci}, the image of $\Ak(W)$
 under the Dunkl embedding $\Phi$ is the algebra  
 generated by $z$, $z \partial_z$ and 
 $$
 \Phi(Y) \ = \  \frac{\ell^{\ell}}{z} \prod_{i = 0}^{\ell-1}
 \left(z \partial_z + \kappa_{i} + \frac{i}{\ell} +(\delta_{i,0} -
 1) \right).  
 $$
 Therefore, Corollary~\ref{cor:Imradrank1} implies that
 
 \begin{corollary}\label{cor:kappa-defn}  Set
        $\kappa=(\kappa_0,\dots,\kappa_{\ell-1})$  where 
        \begin{equation}\label{eq:kappa}
                \kappa_{i} + \frac{i}{\ell} +(\delta_{i,0} - 1) =
                \lambda_i.
                     \end{equation}
      Then  $\Ak(W)$ is  the subalgebra of  $\dd(\h_{\reg})^W$
        generated by 
        \[z=\rad_{\vs}(u),\  \  \rad_{\vs}(\eu_S) \ \text{ and }\
         \nabla= \rad_{\vs}(\Delta).  \qed \] 
 \end{corollary}

  \begin{remark}  \label{kappa-choice}
        The precise definition of the idempotents $e_{H,i}$ is not
        consistent over the literature and the definition of the
        $\kappa_i$ changes accordingly. 

        For example,  in \cite{Le3} the idempotents $e_i$ are related to the
        $e_{H,i}$    by $e_{\ell -i} = e_{H,i}$ for
        $0 \le i \le \ell-1$ (with the
        convention that  $e_\ell = e_0$).
 
        Suppose that $\zeta = 0$. The parameters $\kappa_i$
        from~\eqref{eq:kappa} then differ from the parameters $k_i$ given in
        \cite[(2.5)]{Le3}: for an appropriate numbering of the roots
        of the $b$-function, we have $\kappa_0=k_0=0$ and
        $\kappa_i= k_{\ell -i}$ for $1 \le i \le \ell -1$.

 \end{remark}

 \begin{lemma}\label{prop:localizingRU}
        Let $R$  be a subalgebra of $ \dd(\h_{\reg})^W$, containing $\Ak=\Ak(W)$, and such that the adjoint action of $\C[\nabla]$ on $R$ is locally nilpotent. Then:
        \begin{enumerate}
        \item
         $\Ak[z^{-1}] = R[z^{-1}]=\dd(\h_{\reg})^W$ and 
         \item $\Ak[\nabla^{-1}] = R[\nabla^{-1}]$.
         \end{enumerate}
 \end{lemma}

\begin{proof}
 We use the generators of $\Ak$ defined by Corollary~\ref{cor:kappa-defn}. Since $\eu\in \Ak\subseteq R$,  both $\Ak$ and $R$ are  $\Z$-graded
 subalgebras of $\dd(\h_{\reg})^W$.

(1)  The   action of $z$ on
 $\dd(\h_{\reg})^W$ is locally ad-nilpotent and so  this is also the
 case for $R$. Hence the localisation $R[z^{-1}]$ exists with
   $\Ak[z^{-1}] \subseteq  R[z^{-1}]\subseteq \dd(\h_{\reg})^W$.  The result now follows from the fact that, by 
  Lemma~\ref{lem:quotientetaleh}(1), $\Ak[z^{-1}] = \dd(\h_{\reg})^W$.

(2)  Write $M$ for the polynomial representation $\C[z^{\pm
  1}]$; it is a faithful graded $R$-module. If $p(t) = c
 \prod_{i = 0}^{\ell-1} (t + \lambda_i)$, then
 Corollary~\ref{cor:Imradrank1} shows that $\nabla(z^m) = p(m)
 z^{m-1}$. Thus if $m_1 < \cdots < m_k$ are the integer roots of
 $p$, then $\nabla(z^m) = 0$ if $m = m_i$ for some $i$ and otherwise
 $\nabla(z^m) = \lambda_m z^{m-1}$ for a non-zero scalar $\lambda_m$. If $p$ has no integer roots, this   ensures that $M$ is a cyclic $\C[\nabla^{\pm 1}]$-module generated by the element $1$ and hence is isomorphic to $\C[\nabla^{\pm 1}]$ as a module over that ring. 
 
 On the other hand, if $p$ has integer roots then 
 $\ann_{M} \nabla = \bigoplus_{i=1}^k\C z^{m_i}$.  This
 means that the $\nabla$-torsion submodule $\text{tor}_{\nabla} M = \C \{ z^{j + m_1} \, : \, j \ge 0
 \}$ is actually an $R$-submodule of $M$ simply because the action of
 $\adj(\nabla)$ is locally nilpotent on the ring $R$.  In particular, $M\not=\text{tor}_{\nabla} M$ and $M[\nabla^{-1}]\not=0$.  It follows that, as a $\C[\nabla]$-module, $M/\text{tor}_{\nabla} M = \C[\nabla] z^{m_1-1}$ and hence that $\C[\nabla^{\pm 1}] \cong
 M[\nabla^{-1}]$ (as both a $\C[\nabla]$ and a $\C[\nabla^{\pm 1}]$-module) via $\nabla^j \mapsto \nabla^j z^{m_1 - 1}$. 
 
   Since  $\adj(\nabla)$ acts locally nilpotently on $R$, we can form the localisation $R[\nabla^{-1}]$ at $\euls{C}=\{\nabla^{n}\}$ and clearly the $R$-module structure of $M[\nabla^{-1}]$ extends to an $R[\nabla^{-1}]$-module structure. In other words, we have constructed an $R[\nabla^{-1}]$-structure on $M[\nabla^{-1}]\cong \C[\nabla^{\pm 1}]$. Let $Z=\mathrm{Spec}(\C[\nabla^{\pm 1}])$ denote the torus.
 
 \begin{sublemma}\label{sublemma:localising} Keep the hypotheses of the lemma. Then:
 \begin{enumerate}
 \item  The adjoint action of $\C[\nabla^{\pm 1}]$ on $R[\nabla^{-1}]$ is locally nilpotent;
   \item     the action of $R$ on $M[\nabla^{-1}]$ factors through $\dd(Z)$. 
   \end{enumerate}
 \end{sublemma}
  
 \begin{proof} 
 (1)  The identity $[\theta, r\nabla^{-m}]=[\theta, r]\nabla^{-m}$ for $\theta\in \C[\nabla]$, $ r\in R$ and $m\geq 0$ and  a routine induction ensures that the  action of $\C[\nabla]$ on $R[\nabla^{-1}]$ is also  locally ad-nilpotent. 
 Set $\Lambda_{-1}=0$ and, by induction, write  $\Lambda_n=\bigl\{r\in R[\nabla^{-1}] : [\C[\nabla], r]\in \Lambda_{n-1}\bigr\};$  thus, by local ad-nilpotence,  $R[\nabla^{-1}] =\bigcup \Lambda_n$.
 
 We claim, by induction, that each $\Lambda_n$ is a $\C[\nabla^{\pm 1}]$-module. To see this,
  let $\theta, \phi \in \C[\nabla]$, $ r\in \Lambda_n$ and $m\geq 0$. Then, by induction on $n$, 
 \[ [\phi, \theta\nabla^{-m}r] \ = \  \theta\nabla^{-m}[\phi, r] \ \in\  \theta\nabla^{-m}\Lambda_{n-1}
  \ \subseteq \  \Lambda_{n-1},\]and hence $ \theta\nabla^{-m}r\in \Lambda_n$, as claimed.
 But this in turn implies that, modulo $\Lambda_{n-1}$, 
  \[ [\theta \nabla^{-m},r] \ = \  \nabla^{-m}\bigl(\theta r \nabla^m-\nabla^m r \theta\bigr)\nabla^{-m}
\   \equiv \  \nabla^{-m}\bigl(r \theta  \nabla^m-r \nabla^m \theta\bigr)\nabla^{-m}  \ \equiv \ 0.\]
This in turn implies that $\Lambda_n=\bigl\{r\in R[\nabla^{-1}] : [\C[\nabla^{\pm 1}], r]\in \Lambda_{n-1}\bigr\}$ which, since 
 $R[\nabla^{-1}] =\bigcup \Lambda_n$, proves Part (1).
 
 (2)  Recall that, by definition,  $\dd(Z)$ is  the set of $\C$-linear endomorphisms of $\C[Z]=\C[\nabla^{\pm 1}] =
 M[\nabla^{-1}]$ on which $\C[\nabla^{\pm 1}]$ acts ad-nilpotently. Since $\C[\nabla^{\pm1}]$ is an $R[\nabla^{-1}]$-module,  Part~(1)  implies that  $R[\nabla^{-1}]\hookrightarrow \dd(Z)$, as required.   \end{proof}
 
 We return to the proof of the lemma. Since $\Ak$ contains both $\nabla$ and $\eu$, we see that, under the action of $\Ak$ on $M[\nabla^{-1}]$, the ring $\Ak[\nabla^{-1}]$ surjects onto $\dd(Z)$. Thus, by the sublemma, $R[\nabla^{-1}]$ also surjects onto $\dd(Z)$. If this surjection $\phi$   has a nonzero  kernel, say $J\not=0$, then $I=J\cap R \not=0$. But we know that $\mr{GKdim} \, R = 2$ since $\Ak\subseteq R \subseteq \dd(\h_{\reg})$.  Hence $\mr{GKdim} \, R/I \leq 1$. But then $R/I$ is a P.I. ring \cite{SW}, hence so is its localisation $(R/I)[\nabla^{-1}] = R[\nabla^{-1}]/J$. This cannot map onto the simple ring $\dd(Z)$.  Thus $\phi$ is an isomorphism, as required.
 \end{proof}

 \begin{proposition}\label{thm:RadnilnablaU}  Assume that  $\dim S\git K=1$ and keep the  notation developed above.
        Let $R $ be a subalgebra of $ \dd(\h_{\reg})^W$, containing $\Ak(W)$, such that the adjoint action of $\C[\nabla]$ on $R$ is locally nilpotent. Then $R = \Ak(W)$. 
 \end{proposition}
 
 \begin{proof}
       Lemma~\ref{prop:localizingRU} shows that, for any $D \in R$, there exists $k \gg 0$ such that       $z^k D$ and $\nabla^k D$ belong to $\Ak$. As we now explain, the  result will now follow by using
        the argument from \cite[Lemma~3.8 and Theorem~3.9]{Le3}.  
        First, by comparing  Example~\ref{ex:Dunklcyclci}  with
        \cite[Proposition~2.8]{Le3}, $ \Ak$ identifies
        with   
        the ring $U$ of \cite{Le3}, where our $\nabla$ corresponds to
        the element  $\delta$ of \cite{Le3}, up to a scalar. Thus, the proof of
        \cite[Lemma~3.8]{Le3}  
        can be used to prove that, for any $r\in
        R$,  one has $\GKdim \bigl((\Ak+r\Ak)/\Ak\bigr)
        \leq \GKdim \Ak -2= 0$. The proof of \cite[Theorem~3.9]{Le3} can now  be
        used unchanged to prove that $R=\Ak$. 
 \end{proof}

   \begin{corollary}\label{thm:Levsurjectiverank1} Assume that  $\dim S\git K=1$ and keep the  notation developed above. If $\kappa$ is defined by Equation~\ref{eq:kappa}, then $\Im(\rad_{\vs} ) =
        \Ak(W)$.  
   \end{corollary}
   
   \begin{proof}  As $\Sym S$ acts locally ad-nilpotently on $\dd(S)$, certainly 
   $\C[\Delta]\subseteq (\Sym S)^K$ acts locally ad-nilpotently on $\dd(S)^K$, Thus 
   $\C[\nabla]=\rad_{\vs} (\C[\Delta])$  acts locally ad-nilpotently on $R=\rad_{\vs}(\dd(S)^K)\subseteq \dd(\h_{\reg})^W$. Since $\Ak(W) \subseteq R$ by Corollary~\ref{cor:kappa-defn}, the result follows   from Proposition~\ref{thm:RadnilnablaU}. 
   \end{proof}


\section{The radial parts map:
  existence}\label{sec:radialpartsmap}

As in the rest of the article,   $V$ will denote a polar representation of a
connected reductive group $G$.  %
Recall  that we defined the radial parts  map $ \rad_{\vs} \colon \dd(V)^G
\to \dd(\h_{\reg})^W$ in  \eqref{eq:radialpartsmap}
and \eqref{eq:radialpartsmap2}. 
The goal of this section is to prove the following
      theorem, which will be proved as Theorem~\ref{thm:surjectivereducerank1}. 
      
 \begin{theorem}\label{thm:radpartsmappolar}  Let $V$ be a polar representation for the connected, 
reductive group  $G$.  There exists a parameter $\kappa =
  \kappa(\vs)$ such that the image of $\rad_{\vs}$ is contained in the spherical algebra $\Ak(W)$.
\end{theorem}

The basic idea of the proof follows that in
\cite[Section~4]{LevWashington} for the case $V=\g$, which is in
turn based on the proof in \cite[Section~3]{Sc2}. Here is an
outline. Recall that $\calA$ denotes the set of
reflecting hyperplanes in $\h$. For a rational Cherednik algebra
$\Hk(W)$, the parameter $\kappa$ is uniquely defined by its
values on the rank one parabolic subgroups of $W$. Such a subgroup is the
stabiliser of a generic point $p \in H$, for some  $H \in \calA$. One can take a slice $S_p$ to the $G$-orbit
through $p$. Then $(G_p,S_p)$ is a rank (at most) one polar
representation with Weyl group $W_H$. In the rank one case one
can explicitly compute the image of the radial parts map, as is done in
Corollary~\ref{cor:kappa-defn}, to show that it does indeed lie in the appropriate spherical algebra. This tells us the value of the corresponding parameter 
$\kappa_H$. By considering all hyperplanes $H$ one can then
explicitly compute $\kappa$ in terms of $\vs$ and show that $\Im(\rad_{\vs})\subseteq \Ak$.

We note that in all our examples, $\kappa$ depends affine
linearly on $\vs$, although we do not have a general proof of
this fact. See Remark~\ref{rem:chiclinear} for a possible
explanation.

Before beginning on the proof of Theorem~\ref{thm:radpartsmappolar}, we describe what the radial parts 
map $\rad_{\vs}$ does to the regular locus $V_{\reg}$. For this, it is convenient to rewrite $\rad_{\vs} $ as a composition of
functions. First we have the map $\gamma_{\vs}: \dd(V_{\reg})\to \dd(V_{\reg})$
given by conjugation with the formal element $\deltav^{-\vs}$; thus
$\gamma_{\vs}(D) = \deltav^{-\vs} D \deltav^{\vs}$ for $D\in
\dd(V_{\reg})$. Equivalently, using \eqref{eq:twistedop}, $\gamma_{\vs}$ is the
morphism that is the identity on $\C[V_{\reg}]$ and maps a derivation
$\partial\in \dd(V_{\reg})$ to
$\partial + \sum \vs_i \frac{\partial (\deltav_i)}{\deltav_i}$. Thus
$\gamma_{\vs}$ does indeed map $\dd(V_{\reg})$ to itself and, as
$\deltav^{\vs} \in \C[V_{\reg}]\deltav^{\vs}$ is a $G$-semi-invariant,
$\gamma_{\vs}$ also induces a filtered automorphism of $\dd(V_{\reg})^G$.  Now let
$\eta: \dd(V_{\reg})^G \to \dd(V_{\reg}\modmod G)$ denote the restriction
of operators and write   $\rrr $ for  the identification $\dd(V_{\reg}\git G) \iso \dd(\h_{\reg}\git W)$
 induced from $\rr:\C[V\git G]\iso \C[\h\git W]$. Then \eqref{eq:twistedop} and~\eqref{eq:radialpartsmap}
ensure that 
\begin{equation}\label{eq:radvzfactor1}
\rad_{\vs} = \rrr\circ \eta \circ \gamma_{\vs}.
\end{equation}

A useful consequence of this discussion is that we  can also compose $\rad_{\vs}$ as
\begin{equation}\label{eq:radvzfactor}
        \rad_{\vs} = \rad_0 \circ \, \gamma_{\vs} \colon \dd(V_{\reg})^G \to \dd(\h_{\reg})^W.
\end{equation} 
 
 Recall from Remark~\ref{rem:orderfiltrationH}  that $\Ak(W)$ is given the order filtration induced from that on 
 $\dd(\h_{\reg})$. We similarly use the order filtration on  $\dd(V_{\reg})^G$ induced from that on $\dd(V_{\reg})$.
Finally, given filtered rings $A=\bigcup_{n\geq0} A_n$ and $B=\bigcup_{n\geq0} B_n$, then a homomorphism $\alpha:A\to B$ is called \emph{filtered surjective} if $\alpha$ is a filtered morphism such that $B_n=\alpha(A_n)$ for all $n$. For related concepts see Notation~\ref{defn:order1}.  

\begin{lemma}\label{lem:genericregfiltredsurj}
        The localization of $\rad_{\vs}$ to $\dd(V_{\reg})^G$  gives  a filtered  surjective map $\rad_{\vs} \colon \dd(V_{\reg})^G \twoheadrightarrow \dd(\h_{\reg})^W$. 
\end{lemma}

\begin{proof}
        Since $\rad_{\vs} = \rad_0 \circ \ \gamma_{\vs}$,   where $\gamma_{\vs}$ is   a filtration preserving automorphism  of $\dd(V_{\reg})^G$, it suffices to prove the lemma when $\vs = 0$.  Now $\rad_0=\rrr\circ\eta$ by \eqref{eq:radvzfactor1} and since the restriction of  differential operators is a filtered morphism, $\rad_0$ is clearly  a filtered map.

It remains to prove that $\rad_0$ is filtered surjective.         Let $\bar{v} \in V_{\reg}\git G$ with semisimple lift $v \in V_{\reg}$. By \cite[Theorem~2.3]{DadokKac},  every Cartan subspace of $V$ is conjugate to $\h$ and so we may assume that $v \in \h$. Then \cite[Theorem~4.9]{Schwarz} says (using the notation from \cite[(3.24)]{Schwarz}) that $\rad_0$ is filtered  surjective in an affine open neighbourhood of $\bar{v}$ in $V\git G$ if the radial parts map for a slice at $v$ is  filtered surjective. A slice of the $G$-action at $v$ is given by $(Z,S)$ where 
$S = \h \oplus U$ as in Lemma~\ref{lem:slices}(2),  while $Z=Z_G(\h)$ is the centraliser of $\h$ acting trivially on $\h$ and $U \git Z = \{ \mr{pt} \}$. Identifying $\dd(\mr{pt}) = \C$, the radial parts map for the slice is 
        $$
        \dd(\h) \o \dd(U)^Z \to \dd(\h) \o \C = \dd(\h),
        $$
        which is vacuously  filtered  surjective. 
\end{proof}

\subsection*{Slice representations in $V$}\label{subsec:slices}

In this subsection we examine how to use slice representations to
reduce our problem to the one-dimensional case and, ultimately, 
to prove Theorem~\ref{thm:radpartsmappolar}.   

\begin{notation}\label{slice-notation1}
Let $\calA$ denote the set of reflecting hyperplanes in $\h$ and,
following Notation~\ref{defn:hcirc}, set
$\h^{\circ} = \h \smallsetminus \bigcup_{H_\alpha \neq H_\beta \in \calA}
H_{\alpha} \cap H_{\beta}$.  Fix $H \in \calA$ and
$p \in H^\circ=H\cap \h^\circ$; thus $p$ lies on just the one hyperplane in
$\calA$. The stabiliser of $p$ in $W$ equals the stabiliser
$W_H = \{ w \in W \, | \, w \cdot x = x, \, \forall \, x \in H\}$.  Let $K=G_p$ be
the stabiliser of $p$ in $G$ and note that, as $p$ is semisimple, $K$ is
reductive. As described before \eqref{eq:Ucomplement}, we may chose a
slice $S=S_p$. Thus $S$ is 
$K$-stable complement   to $\g \cdot p$ in $V$  with 
$S\supseteq \mf{h}$; see Lemma~\ref{lem:slices}.  
\emph{This notation will be fixed throughout the subsection.}
\end{notation}

 In general, $K$ need not be connected; see
\cite[Example~15.3]{BNS} for an explicit example. Even so, we have: 

\begin{lemma}\label{lem:slicebproperties}
             {\rm (1)} $(K,S)$ is polar with Cartan subalgebra $\h$.
              
            {\rm (2)} The Weyl group of $(K,S)$ equals $W_H$.
          
      {\rm (3)}     The restriction map $\rr_p \colon \C[S]^{K}
                  \to \C[\h]^{W_H}$ is an isomorphism. 
      
      {\rm (4)}  {$S^{K} = \h^{W_H} =H$ and $K= Z_G(H)$.}
                
       {\rm (5)} $(K,S)$ has rank at most one.
        \end{lemma}

\begin{proof}
  (1) This  is  \cite[Theorem~2.4]{DadokKac}.      
        
        (2)     By definition, the Weyl group of $(K,S)$ equals
        $N_{K}(\h) / Z_{K}(\h)$. This embeds in $W$, say with
        image $W'$. 
        Let $w \in W_H$ and choose a lift $g$ of $w$ in
        $N_G(\h)$. Then $g$ fixes every point in $H$ and hence 
         $g \in K$.  
        Therefore,  $g \in K \cap N_G(\h) = N_{K}(\h)$ and hence $w \in
        W'$. Conversely, if $w \in W'$ then $w$ fixes a generic
        point of the  
        hyperplane $H$,  so $w$ fixes every point in $H$ and thus $w \in W_H$. 
        
        (3) See \cite[Theorem~2.9]{DadokKac}.
        
        (4) The complex reflection group $W_H$ has rank one, so
        $\dim \h^{W_H} = \dim \h - 1$. This implies that the
        space of elements of degree~$1$ in $\C[\h]^{W_H}$
        satisfies $\dim \C[\h]^{W_H}_1 = \dim \h - 1$. But, by
        Part~(3), the restriction map
        $\C[S]^{K}_1 \to \C[\h]^{W_H}_1$ is an
        isomorphism and so $(S^*)^{K} = (\h^*)^{W_H}$. It
        follows that $S^{K} = \h^{W_H} =H$.
          Therefore        $K \subseteq Z_G(H)$. Conversely, since $p \in H$ we have
          $Z_G(H) \subseteq G_p=K$ and so  $K= Z_G(H)$.
        
        (5)   Finally, if $\mf{k}=\Lie(K)$ then $\rank (S,K)=\dim
        \h - \dim \h^{\mf{k}}$ by  \eqref{defn:polarrank}. Since $\h \cap S^{K} \subseteq 
        \h^{\mf{k}}$, it follows from Part~(4) that $\dim    \h^{\mf{k}} \ge \dim \h - 1$, as desired.    
\end{proof}

\begin{remark}\label{rem:rankzeropolar}
        In theory, it is possible for $(K,S)$ to have rank zero
        and this happens precisely if the point $p$ is
        $V$-regular but not regular, see Lemma~\ref{lem:slices}.  (As noted in
        Remark~\ref{regular-remark}, it is conjectured that this
        cannot happen.) In this case the image of $K$ in $\GL(S)$
        is a finite group; indeed, $S = \h$ and the image of $K$
        in $\GL(\h)$ equals $W_H$.  
\end{remark}

As before, let $\mf{k}=\Lie(K)$ and set $ \dim V=n$ and $\dims =\dim S$, for $S=S_p$. 
Given $X\in \g$ and $u\in V$,  write $\tau(X)_u = X\cdot u$.
Pick   a basis  $s_1, \ds, s_{\dims}$   for $S$, extended to a basis 
$\{s_1,\dots, s_n\}$ of $V$.  Under our identification $V\subset
\Sym V\subset \dd(V)$, these $s_i$  become 
derivations on $V$, which we write as
$\partial_{s_i}$ to avoid confusion. Let $X_1, \ds, X_{n-{\dims}}$ be a
basis for the $K$-stable complement $L$ to  
$\mf{k}$ in $\g$. The fact that $V=S\oplus \g\cdot p$ implies that 
$s_1, \ds, s_{\dims}, \tau(X_1)_p, \ds, \tau(X_{n-{\dims}})_p$ is also a
basis of $V$.    
Then, as the $\partial_{s_i}$ form a $\C[V]$-basis of
$\Der(\C[V])$, we can write  
\begin{equation}\label{defn:U}
        \partial_{s_1} \wedge \cdots \wedge \partial_{s_{\dims}} \wedge
        \tau(X_1) \wedge \cdots \wedge \tau(X_{n-{\dims}})  
        = t\cdot  \partial_{s_1} \wedge \cdots \wedge \partial_{s_n},
\end{equation}  
for some $t\in \C[V]$.  Finally, set $U = (t \neq 0)\subset V$. 

\begin{lemma}\label{lem:invariant function}
               {\rm (1)} The function $t$ is  non-zero and $K$-invariant. 
           
               {\rm (2)}  $U$  is a $K$-stable  affine open subset of $V$.
               
               {\rm (3)}  $U  \, = \, \left\{u\in V \, :\, V=S \oplus
                    \mr{Span} (\tau(X_1)_u, \ds, \tau(X_{n-{\dims}})_u)
                  \right\}$. 
       \end{lemma}

\begin{proof}
        (1) The left hand side of Equation~\ref{defn:U}
        evaluated at $p$ is non-zero since $s_1, \ds, s_{\dims},
        \tau(X_1)_p, \ds, \tau(X_{n-{\dims}})_p$ is a basis of
        $V$. Therefore $t(p) \neq 0$ and hence $t\not=0$. 
        If $g \in K$, the fact that $\tau$ is $G$-equivariant implies that 
        \begin{align*}
                g \cdot \partial_{s_1} \wedge \cdots \wedge \partial_{s_{\dims}}  & 
                = \mr{det}_{S}(g)\,\, \partial_{s_1} \wedge \cdots \wedge \partial_{s_{\dims}} \\
                g \cdot \tau(X_1) \wedge \cdots \wedge \tau(X_{n-{\dims}}) & = 
                \mr{det}_{L}(g)\, \tau(X_1) \wedge \cdots \wedge \tau(X_{n-{\dims}}) \quad \text{and}  \\
                g \cdot \partial_{s_1} \wedge \cdots \wedge \partial_{s_n} & =
                \mr{det}_V(g)\,\, \partial_{s_1} \wedge \cdots \wedge \partial_{s_n}.
        \end{align*}
        Since $S \oplus L \cong V$ as $K$-modules, $\det_{S}(g)
        \det_L(g) = \det_V(g)$. Therefore, comparing the three displayed equations to  \eqref{defn:U},  $g(t)=t$. 
        
        Part (2) is immediate  from (1) and Part (3) follows
        from the definition of $t$. 
\end{proof} 

\begin{lemma}\label{lem:decompoUbUS}
        Let $I$ denote the kernel of the map $\C[V] \to \C[S]$ of
        restriction to $S$. Then $\C[V] = \C[S] \oplus I$ as
        $K$-equivariant $\dd(S)$-modules. 

\end{lemma}

\begin{proof}  
        Using the $K$-module decomposition $V = S \oplus \g \cdot
        p$, we regard $\dd(S)$ as a subalgebra of $\dd(V)$. Let
        $x_1, \ds, x_{n-{\dims}} \in V^*$ be coordinates whose set of
        common zeros is the subspace $S$. Then
        $I=\sum_{j=1}^{n-{\dims}} \C[V]x_j$ and $\partial(x_i) = 0$ for
        every derivation $\partial \in \dd(S)$.   Therefore, the $\C[S]$--module decomposition 
        $\C[V] = \C[S] \oplus I$ extends to a $\dd(V)$-module decomposition.  The
        $K$-equivariance is immediate.  
\end{proof}

By Lemma~\ref{lem:slicebproperties}(3), the restriction map
$\rr_p \colon S \to \h$   induces an isomorphism $\rr_p \colon
\C[S]^{K} \stackrel{\sim}{\longrightarrow} \C[\h]^{W_H}$. As
such, the discriminant $\deltah \in \C[\h]^W$ can be considered as a 
$K$-invariant function on $S$ (denoted $\deltav_S$) whose
non-vanishing locus $S_{\reg}$ equals $S \cap V_{\reg}$. The map  
$\rr_p$ localises to $\rr_p \colon \C[S_{\reg}]^{K}
\stackrel{\sim}{\longrightarrow} \C[\h_{\reg}]^{W_H}$. Recall
from \eqref{eq:factordiscinV}  
that the irreducible factors of $\deltav$ are $\deltav_1, \ds,
\deltav_k$. Each $\deltav_i$ restricts to a non-zero
$K$-semi-invariant $\deltav_{S,i}$ on $S$ and so the radial parts
map $\rad_{\vs}$  
from 
\eqref{eq:radialpartsmap} induces a radial parts map
\begin{equation}\label{eq:radbvs}
        \rad_{\vs,p} \colon \dd(S)^{K} \longrightarrow
        \dd(\h_{\reg})^{W_H},  \quad \rad_{\vs,p} (D)(z) =
        \bigl(\deltav^{-\vs}_S D(  \rr^{-1}_p(z) \deltav^{\vs}_S)
        \bigr) |_{\h_{\reg}}, 
\end{equation}
for $\deltav^{\vs}_S = \deltav_{S,1}^{\vs_1} \cdots
\deltav_{S,k}^{\vs_k}$.

Let $\euls{S}$ denote the set of all functions in $\C[V]$ not
vanishing on $\h_{\reg}^{\dagger} = U \cap \h_{\reg}$ and
$\euls{C}$ the set of all functions in $\C[\h]$ not vanishing on
$\h_{\reg}^{\dagger}$. The following technical Lemma~\ref{lem:PQRdecomp} is key to the
main result.  Before stating it, recall that the function $t$ is defined in~\eqref{defn:U} and that $\g_{\chi}$ is defined in Remark~\ref{eq:g-chi}.   Also, as the next remark shows, we are free to use  $\h_{\reg}^{\dagger}$ in place of $\h_{\reg}$ when defining $\rad_{\vs}$.

\begin{remark}\label{rem:alternativedefrad}
  Choose any $W$-stable affine open subset $C$ of $\h_{\reg}$ and set $U = \pi^{-1}(C/W) \subset V_{\reg}$. As $V$ is polar, restriction defines an isomorphism $\rr \colon \C[U]^G \stackrel{\sim}{\longrightarrow} \C[C]^W$. One can then define
        $$
        \rad'_{\vs} (D)(z) = \bigl(\deltav^{-\vs}D( \rr^{-1}(z)         \deltav^{\vs}) \bigr) |_{C}\qquad \text{for}\quad z \in         \C[C]^W \text{ and } D\in \dd(V)^G.
      $$
       Under the inclusion $\dd(\h_{\reg})^W \subseteq \dd(C)^{W}$   the image of this  map $\rad'_{\vs}$ is the same as         that of \eqref{eq:radialpartsmap2}, and so we can relabel $\rad'_{\vs}=\rad_{\vs} $ without ambiguity. In other words,  it does not         matter which ``test functions'' are used to define the         radial parts map.     \end{remark}

Recall from Lemma~\ref{lem:invariant function}(2) that $U=(t\not=0)$ is $K$-stable and that $t$ vanishes nowhere on $\h^{\dagger}_{\reg}$. Thus,     the morphism $\rad_{\vs,p}$ restricts to give   a map   
$  \dd(U)^{K} \to
        \dd(\h^{\dagger}_{\reg})^{W_H}, $ which we also write as $\rad_{\vs,p}$. It is  this second version of $\rad_{\vs,p}$ that is used in the next lemma.

\begin{lemma}\label{lem:PQRdecomp}
        Let $D \in \dd(V)^G$. There exists $\ell > 0$ and $P \in \dd(S)^{K}$ such that 
        $$
        \rad_{\vs}(D) = \rad_{\vs,p}\left((t|_S)^{-\ell} P\right)
        $$
        as elements of $\dd(\h)_{\euls{C}}$.
 \end{lemma}

\begin{proof}  
        As in Lemma~\ref{lem:invariant function} we identify
        $\C[S]\subseteq \C[V]\subseteq \C[U]$. We note that
        $\deltav, \deltav_S$ and $t$ belong to $\euls{S}$. Recall
        from Notation~\ref{defn:order} that $\dd_{\idot}(X)$
        denotes the order filtration on $\dd(X)$. It follows from
        Lemma~\ref{lem:invariant function} that, locally at any
        point $u\in U$,  we have  
        \begin{equation}\label{eq:local-der}
                \mr{Der}\, \C[U] = \left(\C[U] \o_{\C[S]}
                  \mr{Der} \, \C[S] \right) + \C[U]
                \tau(\mf{g}). 
        \end{equation} 
        Therefore, by \cite[Corollary~A1]{BD}, \eqref{eq:local-der}
        holds globally on $U$.  We have $\C[U] = \C[U] \o_{\C[S]}
        \C[S]\subseteq \C[U] \o_{\C[S]} \dd_1(S)$, and
        $\C[U] \g_{\chi} + \C[U] = \C[U] \tau(\g) + \C[U]$ inside
        $\dd(U)$ by the definition \eqref{eq:g-chi} of $\g_{\chi}$.  We therefore deduce that  
        $$
        \dd_1(U) = \left(\C[U] \o_{\C[S]} \dd_1(S)\right) + \C[U]\g_{\chi}. 
        $$
        This implies that $\dd(U) = \left(\C[U] \o_{\C[S]}
          \dd(S)\right) + \dd(U) \g_{\chi}$ and hence that 
        $$
        \dd(U)^{K} = \left(\C[U] \o_{\C[S]} \dd(S)\right)^{K} + (\dd(U) \g_{\chi})^{K}. 
        $$
        Since $U$ is the principal open set $(t \neq 0)$, there
        exists  
        $P_0 \in (\C[V] \o_{\C[S]} \dd(S))^{K}$ and
        $R \in (\dd(V) \g_{\chi})^{K}$ such that
        $t^{\ell} D = P_0 + R$ for some $\ell\geq 0$. Moreover, by
        Lemma~\ref{lem:decompoUbUS} we may further decompose
        $P_0 = P + Q$, for  $P \in \dd(S)^{K}$ and
        $Q \in (I \o_{\C[S]} \dd(S))^{K}$.

Let $z \in \C[\h_{\reg}]^W$ with $\rr^{-1}(z) \in
\C[V_{\reg}]^G$. Then \eqref{eq:radialpartsmap2} combined with Remark~\ref{rem:alternativedefrad}
says that  
        $$
        \rad_{\vs}(D)(z) = \bigl(\deltav^{-\vs}
D(\rr^{-1}(z)\deltav^{\vs})\bigr) |_{\h_{\reg}^{\dagger}} =
\bigl(\deltav^{-\vs} t^{-\ell}(P + Q +
R)(\rr^{-1}(z)\deltav^{\vs})\bigr) |_{\h_{\reg}^{\dagger}}.
        $$
        Now $Q(\rr^{-1}(z)\delta^{\vs})\in I_{\euls{S}}
\deltav^{\vs}$. Since $I |_{\h} = 0$, it follows that
$\deltav^{-\vs}t^{-\ell}Q(\rr^{-1}(z)\deltav^{\vs})
|_{\h_{\reg}^{\dagger}} = 0$. On the other hand,
$R(\rr^{-1}(z)\deltav^{\vs})|_{\h_{\reg}} = 0$ since
$\g_{\chi}(\rr^{-1}(z)\deltav^{\vs})|_{\h_{\reg}} = 0$. Hence
        $$
        \rad_{\vs}(D)(z) = \bigl(\deltav^{-\vs} t^{-\ell}
P(\rr^{-1}(z)\deltav^{\vs})\bigr) |_{\h_{\reg}^{\dagger}}.
        $$
 
 Let $I=\sum \C[V]x_i$, as in the proof of Lemma~\ref{lem:decompoUbUS}   Then, by that lemma,   we can write $\rr^{-1}(z) =
\rr_p^{-1}(z) + \sum_{i = 1}^{n-m} x_i z_i$ for some $z_i \in
\C[V_{\reg}]$. Since $P\in \dd(S)$ acts trivially on the $x_i$
and $x_i(\h_{\reg}^{\dagger})=0$ for each $i$,
 \begin{equation}\label{eq:PDQ}
 \begin{aligned}
  \bigl(\deltav^{-\vs} t^{-\ell}  &
P(\rr^{-1}(z)\deltav^{\vs})\bigr)|_{\h_{\reg}^{\dagger}} =   \\  & =
\bigl(\deltav^{-\vs} t^{-\ell} P(\rr_p^{-1}(z)\deltav^{\vs})\bigr)
|_{\h_{\reg}^{\dagger}} + \sum_{ i = 1}^{n-m} x_i \deltav^{-\vs}
t^{-\ell} P(z_i \deltav^{\vs}) |_{\h_{\reg}^{\dagger}} \\  \noalign{\vskip 4pt} & =
\bigl(\deltav^{-\vs} t^{-\ell}P(\rr_p^{-1}(z)\deltav^{\vs})\bigr)
|_{\h_{\reg}^{\dagger}}  \ =\ \left(t^{-\ell} \right)|_{\h_{\reg}^{\dagger}}
\bigl(\deltav^{-\vs} P(\rr_p^{-1}(z)\deltav^{\vs})\bigr)
|_{\h_{\reg}^{\dagger}}.
\end{aligned}
\end{equation}
Recall that $P \in \dd(S)^K$, but $\deltav,t
\in \C[V]$.

\begin{sublemma}\label{PDQ-sublemma}
 $\bigl(\deltav^{-\vs}
t^{-\ell}P(\rr_p^{-1}(z)\deltav^{\vs})\bigr) |_{\h_{\reg}^{\dagger}}
= \bigl(\deltav^{-\vs}_S
(t|_S)^{-\ell}P(\rr_p^{-1}(z)\deltav^{\vs}_S)\bigr)
|_{\h_{\reg}^{\dagger}}$
\end{sublemma}

\begin{proof} Since
       $\left(t^{-\ell} \right)|_{\h_{\reg}^{\dagger}} =
\left((t|_S)^{-\ell} \right)|_{\h_{\reg}^{\dagger}}$, by \eqref{eq:PDQ} it remains to show that 
$$
\bigl(\deltav^{-\vs}P(\rr_p^{-1}(z)\deltav^{\vs})\bigr) |_{\h_{\reg}^{\dagger}} =
\bigl(\deltav^{-\vs}_S P(\rr_p^{-1}(z)\deltav^{\vs}_S)\bigr)
|_{\h_{\reg}^{\dagger}},
$$
for $P \in \dd(S)^K$. Since $I_{\euls{S}}
|_{\h_{\reg}^{\dagger}} = 0$, we must show that
        $$
        \deltav^{-\vs} P \deltav^{\vs} \ \equiv \ \deltav^{-\vs}_S P
\deltav^{\vs}_S \ \  \mathrm{mod} \ I_{\euls{S}} \o_{\C[S]} \dd(S),
        $$
        for any $P \in \dd(S)$. Formal conjugation by both
$\deltav^{\vs}$ and $\deltav_S^{\vs}$ are algebra automorphisms of
$\dd(V)_{\euls{S}}$. Therefore, it suffices to prove this identity 
when $P = \partial\in \Der(\C[S])$. Taking each factor $\deltav_i$
of $\deltav$ in turn, this reduces to the statement that
$\partial(\deltav_i)\deltav_i^{-1} \equiv \partial(\deltav_{S,i})
\deltav_{S,i}^{-1} \ \mathrm{mod} \ I_{\euls{S}}$.  Since 
$\deltav_i = \deltav_{S,i} + f_i$ for some  $f_i \in I$, this follows from
the fact that, by Lemma~\ref{lem:decompoUbUS},  $\partial(f_i) \in I$.
\end{proof}

We return to the proof of the lemma.         The sublemma implies that $\rad_{\vs}(D)(z) =
\rad_{\vs,p}((t|_S)^{-\ell}P)(z)$, as required.
\end{proof}

For the rest of the subsection we assume  that 
$(K,S)$ has rank equal to one. Let $S_H$ denote the $K$-stable complement
to $S^{K}$ in $S$. By Lemma~\ref{lem:slicebproperties}(4), $S^{K}
= \h^{W_H} = H$. If we set $\h_H = \h \cap S_H$, then $(K,S_H)$ is a
rank one polar representation with one-dimensional Cartan
subalgebra $\h_H$. In particular, $\dim S_H \git K = 1$. 

We can write $\C[S_H]^{K} = \C[u]$  for some  polynomial $u$, but we need to be precise about the choice of  $u$.  Set $\alpha :=\alpha_H \in \h^*$.  
Since  $\dim S_H\cap \h=1$, clearly $\C[S_H\cap \h]=\C[\alpha]$ and hence $\alpha^{\ell_H}\in \C[S_H\cap \h]^{W_H}= \C[\h_H]^{W_H}$, where $\ell_H=|W_H|$. From the Chevalley isomorphism 
$\phi: \C[S_H]^K \iso \C[\h_H]^{W_H}$ we  can, and will, take $u$ to be the preimage of $\alpha^{\ell_H} $ in $\C[S_H]^K$. As in   \eqref{eq:ufactoru1tourun}, fix a factorisation 
$$
u = u_1^{p_1} \cdots u_r^{p_r}
$$ 
into equivariantly irreducible polynomials in $\C[S_H]$. Recall from \eqref{eq:Hdiscriminant} that we associate to the hyperplane $H$ the factor $\deltah_H$ of the discriminant $\deltah \in \C[\h]^W$.

\begin{lemma}\label{lem:easyresdeleteun} Assume that 
        $(K,S)$ has rank one and, as in \eqref{eq:factordiscinV}, let the $\deltav_i$ be the 
        irreducible factors of $\deltav$ in $\C[V]$.
        \begin{enumerate}
                \item $\deltav|_S = g u$, where $g\in\C[S]^K$ is a
                $K$-invariant function with $g |_{\h} = \deltah_H$.
                \item There exist unique $n_j
                \ge 0$ such that $\deltav_i |_S = g_i \prod_{j
                        = 1}^r u_j^{n_{i,j}}$, where $g_i$ is a $K$-invariant function not divisible by $u$ and $g_i |_{\h}$ divides $\deltah_H$.
        \end{enumerate}
\end{lemma}

\begin{proof}  (1)
        Restriction is an isomorphism $\phi \colon \C[S]^K \iso \C[\h]^{W_H}$ of graded polynomial rings. By \eqref{eq:discriminant} and \eqref{eq:Hdiscriminant},   $\deltah = \alpha^{\ell_H} \deltah_H$ in $\C[\h]^{W_H}$. Since $\phi(u) = \alpha^{\ell_H}$ and $\phi(\delta |_S) = \delta |_{\h} = \alpha^{\ell_H} h_H$, we deduce that $\deltav|_S = g u$, where $g = \phi^{-1}(h_H)$.

        (2) Recall from Lemma~\ref{lem:semiinvUFD} that $B := \C[S]^{[K,K]}$ is a UFD. We begin by noting that the highest common factor of any $u_j$ and $g$ (in $B$) is one. Indeed, $u_j |_{\h}$ divides $u |_{\h} = \alpha^{\ell_H}$, and if $f$ is a factor of $g$ then $f |_{\h}$ divides $h_H$. But the highest common factor of $\alpha$ and $h_H$ in $\C[\h]$ is one, as is evident from \eqref{eq:discriminant} and \eqref{eq:Hdiscriminant}. 
        
        Let $g_i$ denote the highest common factor of $\delta_i |_S$ and $g$ in $B$ and $\prod_{j = 1}^r u_j^{n_{i,j}}$ the highest common factor of $\delta_i |_S$ and $u$, again in $B$. Since $\delta |_S = g u$, the previous paragraph implies that $\deltav_i |_S = g_i \prod_{j= 1}^r u_j^{n_{i,j}}$. The uniqueness of the $n_{i,j}$ follows from the fact that $B$ is a UFD. 
        
        Finally, we must show that $g_i$ is $K$-invariant. Recall that $S = S_H \oplus H$, with $S^K = H$, as $K$-modules. If $g_i |_{H} = 0$ then necessarily $\alpha$ divides $g_i |_{\h}$. But the latter divides $h_H$ and we  already saw that the highest common factor of $\alpha$ and $h_H$ is one. Thus, $g_i |_{H} \neq 0$. The restriction map $\C[S] \to \C[H]$ is $K$-equivariant and $g_i$ is a $K$-semi-invariant. Since $K$ acts trivially on $H$, $g_i |_{H} \neq 0$ implies that $g_i$ is $K$-invariant. \end{proof}

  Recall that we are assuming that we are in the rank one case and we now 
wish to reduce to the one-dimensional situation. Define
\begin{equation}\label{eq:bspichi}
  \bpsi = (\psi_1, \ds, \psi_r)
\quad \textrm{ by } \quad \psi_j = \sum_{i=1}^k n_{i,j} \vs_i
\end{equation}
where the $n_{i,j}$ are constructed by Lemma~\ref{lem:easyresdeleteun}. 
Then we can define a radial parts map $\rad_\phi:\dd(S)^K\to \dd(\h_{\reg})^{W_H}$ by 
$$\rad_{\phi}(D)(z) = \bigl(u^{-\psi}
\sigma(D) ( \rr^{-1}_p(z) u^{\psi}) \bigr) |_{\h_{\reg}}\quad \text{ for} \ z\in \C[\h_{\reg}]^{W_H}.$$

If $A$ is a
noetherian algebra and $a \in A$ is such that $\{ a^n \}_{n \in \mathbb{N}}$ forms an Ore set,
then  we write $A_a$ for  the localization of $A$ at this Ore set.

\begin{lemma}\label{lem:gfradpsi} Assume  that 
$(K,S)$ has rank one.
 Define $g$ and $\bpsi$ by Lemma~\ref{lem:easyresdeleteun} and
 \eqref{eq:bspichi}, respectively and let $\sigma \colon \dd(S)^K_{g} \to \dd(S)^K_{g}$ be
        the automorphism given by conjugating with $g^{-\vs}$, analogous to  
  the definition of conjugation by $\deltav^{-\vs}$.
  
   Then there is a commutative diagram
        $$
        \begin{tikzcd} \dd(S)^K_{g}\ar[rr,"\sigma"]
\ar[dr,"\rad_{\vs,p}"'] & & \dd(S)^K_{g} \ar[dl,"\rad_{\psi}"] \\
& \dd(\h_{\reg})^{W_H}. &
        \end{tikzcd}
        $$
         
\end{lemma}

\begin{proof} Formally, $\delta^{\vs}_S = g^{\vs} u^{\psi}$,
where $g^{\vs} = g_1^{\vs} \cdots g_k^{\vs_k}$ and $u^{\psi} =
u_1^{\psi_1} \cdots u_r^{\psi_r}$. By \eqref{eq:radbvs},
        \begin{align*} \rad_{\vs,p} (D)(z) & =
\bigl(\delta^{-\vs}_S D( \rr^{-1}_p(z) \delta^{\vs}_S) \bigr)
|_{\h_{\reg}} \\ & = \bigl(u^{-\psi} (g^{-\vs} D g^{\vs}) (
\rr^{-1}_p(z) u^{\psi}) \bigr) |_{\h_{\reg}} \\ & = \bigl(u^{-\psi}
\sigma(D) ( \rr^{-1}_p(z) u^{\psi}) \bigr) |_{\h_{\reg}} =
\rad_{\psi}(\sigma(D)),
        \end{align*}
        as required.
\end{proof}

\subsection*{General polar representations}
We now return to a general polar representation $(G,V)$. Before proving Theorem~\ref{thm:surjectivereducerank1}, which is the main
result of this section, we begin with a remark on slice representations. This will  ensure that the $\kappa_{H,i}$ 
to be defined in~\eqref{kappa-defn}  are indeed  well defined. We recall that $V_{\st}$ denotes the set of $V$-regular elements as defined in Section~\ref{Sec:polarreps}.
 
\begin{remark}
  \label{rem:parameters} Keep the notation from Notation~\ref{slice-notation1}  and note
  that $H^{\circ} = H \cap \h^{\circ}$ is a non empty open
  subset of~$H$. 
   
   \smallskip
  (1) 
  Suppose that $H \cap V_{\st} = \emptyset$. Let
  $p \in H^{\circ}$; since $p$ is not $V$-regular,    Lemmata~\ref{lem:slices} and~\ref{lem:slicebproperties} imply that $\rank(G_p,S_p) =1$. Moreover,
  from the same lemmata we deduce that $G_p= Z_G(H)$ and
  $S_p= \h \oplus \g_p\cdot \h\oplus U= \h \oplus Z_\g(H) \cdot \h \oplus U$ are independent
  of $p \in H^0$. Thus   $\C[S_p]^{G_p}= \C[H] \otimes \C[u]$ and
  $\C[\h]^{W_H} = \C[H] \otimes \C[u_{\mid \h}]$ with $u$ independent of
  $p$. Hence, by definition of the $b$-function associated to
  $\Delta=u_*(D)$ in~\eqref{thm:PHVbfunction}, the roots $\lambda_i$
  introduced in Notation~\ref{notation:chidot} do not depend
  on~$p \in H^{\circ}$.
  \smallskip
 
  (2) Suppose instead  that $H \cap V_{\st} \neq
  \emptyset$ (as noted in Remark~\ref{regular-remark},  it is conjectured that this  case never
  occurs). Then, $H \cap V_{\st}$ is open in $H$ since the
  elements of $\h$ are semisimple, and so there exists
  $p_0 \in H^{\circ} \cap V_{\st}$. Then, Lemmata~\ref{lem:slices}(2)
  and~\ref{lem:slicebproperties}(4) applied to $K=G_{p_0}$ imply  that
  $\g_{p_0}= Z_\g(H) = Z_\g(\h)$.

 Now let $p \in H^{\circ}$ be arbitrary. Then $\g_p=\text{Lie}(G_p)=Z_{\g}(H)$ 
  by  Lemma~\ref{lem:slicebproperties}(4).     Since  we   proved that $Z_\g(H) = Z_\g(\h)$ it follows that 
$ \g_p = Z_\g(\h) $  and hence $p\in H \cap V_{\st}$ by Lemma~\ref{lem:slices}(2), again.  
Therefore $\rank(G_p,S_p)=0$ for all $p\in H^\circ$. 
In this case $H$ will contribute nothing to $\kappa$ (see \eqref{kappa-defn}), which is certainly
 independent of $p\in H^\circ$. 
 \end{remark}

We now come to the main result of this section, announced in Theorem~\ref{thm:radpartsmappolar}.

\begin{theorem}\label{thm:surjectivereducerank1}
 Let $\zeta \in \C^k$. Define the parameter
  $\kappa=\kappa(\zeta)$ by
\begin{equation} \label{kappa-defn}
  \kappa_{H,i}= 
  \begin{cases}
    \lambda_i - \dfrac{i}{\ell_H} +(1 - \delta_{i,0}) \ & \text{for
      $i=0,\dots,\ell_H-1$, if $H \cap V_{\st} = \emptyset$,}
    \\
    0 \ & \text{for all $i$, if $H \cap V_{\st} \neq
      \emptyset$,} 
    \end{cases}
\end{equation}
where the $\lambda_i$ are associated, via Notation~\ref{notation:chidot},
to the slice representation $(G_p,S_p)$ for $p \in H^{\circ}$. Then:
\[
\rad_{\vs} (\dd(V)^G)\subseteq \Ak(W).
\]  
\end{theorem}

\begin{proof} 
  We begin by considering the case when $\mathrm{rank}(G,V) =
  0$. Since $G$ is connected, this implies that
  $V = \h \oplus U$, where $G$ acts trivially on $\h$ and
  $U\git G = \{ \mr{pt} \}$. Note that $\dd(U)^G$ acts on
  $\C[U]^G = \C$ and hence admits a quotient $\dd(U)^G \to
  \dd(\C)=\C$. Here the discriminant is equal to one, so  there is no twist 
   and $\rad \colon \dd(V)^G \to \dd(\h)$ is  simply the  morphism
  $\dd(\h) \o \dd(U)^G \to \dd(\h)$ which is the identity on the
  first factor and the map $\dd(U)^G \to \C$ on the second. Since
  $\Ak(W) = \dd(\h)$ for $W = \{ 1 \}$, the theorem holds in this
  case.
        
Therefore, we may assume that  $\text{rank}\,(G,V)\geq 1$ and hence that $\mathcal{A} \neq \emptyset$. Pick a hyperplane $H\in \mathcal{A}$, an
element $p\in H^\circ=H\cap \h^\circ$ and a corresponding slice
representation $(G_p,S_p)$.   By
Lemma~\ref{lem:slicebproperties}, $(G_p,S_p)$ is a polar
representation with Weyl group $W_H$ and rank at most one.

If $H\cap V_{\st}\not=\emptyset$, then  $p \in V_{\st}$ by Remark~\ref{rem:parameters}(2). Moreover, 
by Remarks~\ref{rem:rankzeropolar} and~\ref{rem:parameters}~(2),  $G_p = W_H$,
$S_p = \h$ and $\rank(G_p,S_p) = 0$. In this case we set
$\kappa_{H,i} = 0$ for all $i$. Then
$\Ak(W_H) = \dd(\h)^{W_H} = \dd(S_p)^{G_p}$, so again the theorem holds on this slice.
        
We are left with the case when  $H \cap V_{\st} = \emptyset$ and hence
$\rank(G_p,S_p)=1$ by Remark~\ref{rem:parameters}~(1). Set
$(K,S) := (G_p,S_p)$ with $S = S_H \oplus \h^{W_H}$.
{Recall that if we set $\h_H = \h \cap S_H$, then $(K,S_H)$ is a
 rank one polar representation with one-dimensional Cartan
 subalgebra $\h_H$, hence $(\h_H)_{\reg} = \h_H \smallsetminus\{ 0 \}$.}
As in Lemma~\ref{lem:easyresdeleteun}, we have $\deltav|_S=g u$  where $g |_{\h} = \deltah_H$
and factorisations $
        \deltav_i |_S = g_i \prod_j u_j^{n_{i,j}} $.
 Define  $\bpsi$ by
\eqref{eq:bspichi}. Now choose  the parameters $\kappa_{H,i}$  as in the first case of~\eqref{kappa-defn}, 
 set $\kappa_H=\{\kappa_{H,i}\}$ and define a spherical algebra $A_{\kappa_H}(W_H, \h_H)  \subseteq 
\dd((\h_H)_{\reg})^{W_H}$  in terms of these parameters. Then, 
Corollary~\ref{thm:Levsurjectiverank1} shows that $\rad_{\psi}(\dd(S_H)^K)=A_{\kappa_H}(W_H, \h_H)$.
Recall that, by definition, $S_H\oplus S^K=S$ and hence,  by Lemma~\ref{lem:slicebproperties},  $\h=\h_H\oplus \h^{W_H}$.  Thus we can extend  $\rad_{\psi}$  to a morphism   
        \begin{align*}
     \rad_{\psi}\otimes 1:    \dd(S)^K\ =& \ \dd(S_H)^K \o \dd(\h^{W_H}) \ \longrightarrow \ 
A_{\kappa_H}(W_H, \h_H)  \o \dd(\h^{W_H})\  \  \subseteq\    \\ &
\subseteq \  \dd((\h_H)_{\reg})^{W_H} \o \dd(\h^{W_H})\  \subseteq\   \dd(\h_{\reg})^{W_H}.
   \end{align*}     
        Now $A_{\kappa_H}(W_H,\h_H)\otimes \dd(\h)^{W_H}$ is   the spherical algebra 
        constructed inside  $\dd(\h_{\reg})^{W_H}$ using the parameters $\kappa_H$, and we call this algebra $A_{\kappa_H}(W_H,\h)$.
        We deduce that the image of $\dd(S)^K$ under
$\rad_{\psi}$ equals $A_{\kappa_H}(W_H,\h) \subseteq
\dd(\h_{\reg})^{W_H}$.  Therefore, Lemma~\ref{lem:gfradpsi} implies
that  
        $$
 \rad_{\vs,p}(\dd(S)^K) \ \subseteq \        \rad_{\vs,p}\left(\dd(S)^K_g \right) \ = \  \rad_{\psi}
\left(\dd(S)^K_g\right)  \ \subseteq \   A_{\kappa_H}(W_H,\h)_{\delta_H}.
        $$
        Let $h_p = (t |_{\h}) \deltah_H$. We deduce from
Lemma~\ref{lem:PQRdecomp} that $\rad_{\vs}(\dd(V)^G) \subseteq
A_{\kappa_H}(W_H)_{h_p}$.
        
        Since the parameter $\kappa$ is defined on one $W$-orbit
of hyperplanes at a time (see Definition~\ref{defn:Cherednik}),
by repeating the above argument for one hyperplane from each such
orbit independently we define the global parameter $\kappa$.  By Remark~\ref{rem:parameters} this is independent of the particular choice of points $p\in H^\circ$.

        In order to complete the proof, we need the following subsidiary result.
        
        \begin{sublemma}\label{lemma:sliceintersect} Let $\eS =
\C[\h] \smallsetminus \{ 0 \}$. As subalgebras of
$\dd(\h)_{\eS}$,
$$ 
\Ak(W_H)_{\deltah} \, \cap \bigcap_{p \in
H^{\circ}} \Ak(W_H)_{h_p}   =\Ak(W_H)_{\deltah_H}.
 $$ 
        \end{sublemma}
        
        \begin{proof} Let $R = \C[\h]^{W_H}$. By
Lemma~\ref{pdQ}, the algebra $\Ak(W_H)$ is a free
$R$-submodule of $\dd(\h)_{\eS}$. Since
 $$
 \Ak(W_H)_{h_p} = \Ak(W_H) \o_{R} R_{h_p}, \quad\text{and}\quad
\Ak(W_H)_{\deltah_H}  \ = \  \Ak(W_H) \o_{R} R_{\deltah_H},
                $$
it therefore suffices to show that
\begin{equation}\label{eq:localdelteHinterection}
R_{\deltah} \,  \cap  \bigcap_{p \in H^{\circ}} R_{h_p} \  = \  R_{\deltah_H}.
\end{equation}

Now $R_{h_H}$    lies inside
the left hand side of  \eqref{eq:localdelteHinterection} 
since $\deltah_H$ divides both $\deltah$ and $h_p$. Conversely, take
$f$ belonging to the left hand side, considered as a
$W_H$-invariant rational function on $\h$. Then $f$  is regular on
$\h_{\reg}$ and at every point of $H^{\circ}$. The set $\h_{\reg}
\cup H^{\circ}$ is open in $\h$ and its complement in $\h_{\reg}
\cup H$ has codimension two. Therefore, since $R$ is an  integrally closed domain, $f$ is regular on the
whole of $\h_{\reg} \cup H$. The complement of $\h_{\reg} \cup H$
in $\h$ is contained inside the zero set of $\deltah_H$ (the
latter being the union of all hyperplanes except $H$). Thus, the
left hand side of \eqref{eq:localdelteHinterection} is contained
inside  $R_{h_H}$, as required.
 \end{proof}
        
  We return to the proof of the theorem.       For each $p \in H^{\circ}$, we have shown that
$\rad_{\vs}(\dd(V)^G) \subseteq \Ak(W_H)_{h_p}$. Moreover, $\rad_{\vs}(\dd(V)^G)
\subseteq \dd(\h_{\reg})^{W}$ and, by Lemma~\ref{lem:quotientetaleh}(1), $\dd(\h_{\reg})^W \subseteq
\dd(\h_{\reg})^{W_H} = \Ak(W_H)_{\deltah}$.  Thus, by
Sublemma~\ref{lemma:sliceintersect},
\[
  \rad_{\vs}(\dd(V)^G) \  \subseteq \  \Ak(W_H)_{\deltah} \, \cap
  \bigcap_{p \in H^{\circ}} \Ak(W_H)_{h_p} \  = \
\Ak(W_H)_{\deltah_H}.
\]
Therefore, by Theorem~\ref{thm:intersectsphericalinddhreg},
\[
  \rad_{\vs}(\dd(V)^G) \ \subseteq \  \bigcap_{H\in
            \mathcal{A}} \Ak(W_H)_{\deltah_H}\  \subseteq\ \Ak(W),
\]
as required.
\end{proof}


\section{Conjugation of radial parts}\label{sec:conjugation}

 A natural question raised by the last section is whether the images $\Im(\rad_{\vs})$ are isomorphic as $\vs$ varies. 
 In this short section, we show that, \emph{for fixed $\chi$},  this is indeed the case; see Corollary~\ref{lem:invconjugationvs}. As part of the proof, we also show that the radial parts map  $\rad_{\vs}$ does induce the surjection \eqref{eq:localizatioradiso} upon restriction to the regular locus; see Proposition~\ref{lem:localizeradialpartsunstable}.  We also briefly discuss stable representations, since these will be needed later in the paper.

\subsection*{The generic slice} Let $\h$ be a Cartan subspace of
a polar presentation $(G,V)$ and adopt the notation of
Section~\ref{Sec:polarreps}.  Set
$Z= Z_G(\h)$, $N= N_G(\h)$ and  fix an $N$-stable complement
$U$ to $\h \oplus \g \cdot \h$ in $V$ as
in~\eqref{eq:Ucomplement}.  As shown after that equation,  $U \git Z = \{ \mr{pt} \}$ and so  the
closed $Z$-orbits in $\h_{\reg} \times U$ are precisely the
singleton sets $\{ (x,0)\}$ for $x \in \h_{\reg}$. Recall that a
subset $X \subseteq V$ is said to be \emph{$G$-saturated} if
$X = \pi^{-1}(\pi(X))$, where $\pi:V\to V\git G$ is the categorical quotient. In particular,   $V_{\reg}$ is $G$-saturated.

 \begin{proposition} \label{lem:unstablepolaropen}
        The  map  \[G \times_{N_G(\h)} (\h_{\reg} \times U) \longrightarrow
        V_{\reg},\]
        given  $[g,(x,u)] \to  g\cdot (x+u)$,  is an  isomorphism.
 \end{proposition}
 
 \begin{proof}  We  have a morphism $\sigma':G \times_{N}(\h \times U) \to V$ given by
        $[g,(x,u)]\mapsto g\cdot (x+u)$ and we write $\sigma$ for the restriction of $\sigma'$ to 
        $G \times_{N}(\h_{\reg} \times U)$.

     We first show that the image $V'=\Im(\sigma)$ is a $G$-saturated open subset of $V$. Since  $V'$ is
   also the image of $G \times_Z (\h_{\reg} \times U)$ under the
   obvious map (which we also call $\sigma$), we work with the
   latter. For each $x \in \h_{\reg}$, the map $\sigma$ is \'etale at
   $[1,(x,0)]$. Moreover, since $Z \cdot (x,0)$ is closed in
   $\h_{\reg} \times U$, the orbit $G \cdot [1,(x,0)]$ is closed in
   $G \times_Z (\h_{\reg} \times U)$. Since
   $Z = G_x$ by Lemma~\ref{lem:stabilizerpolarh},  $\sigma$ maps  $G \cdot [1,(x,0)]$ bijectively onto
   $G \cdot x \subset V$. Therefore, by \cite[Lemme
   fondamental, II.2, p.94]{Luna}, there exists a $G$-saturated affine open
   neighbourhood of $G \cdot [1,(x,0)]$ whose image in $V$ is a
   $G$-saturated affine open neighbourhood of $G \cdot x$. 
       
   Now  let $v \in
\pi^{-1}\pi(V')$  and let $G\cdot x$ be  the unique closed orbit
in $\overline{G\cdot v};$ thus $\pi(v)= \pi(x).$
By the last paragraph, there exists  an affine open
neighbourhood $\Omega$ of $G\cdot[1,(x,0)]$ such that $\sigma(\Omega)$
is a saturated neighbourhood of $G\cdot x.$
Set $x'= [1,(x,0)].$ Then $\pi(\sigma(x')) = \pi(x) = \pi(v)$
whence $v \in \pi^{-1}\pi(\sigma(x'))$.  Since
$\pi^{-1}\pi(\sigma(x'))\subseteq \pi^{-1}\pi(\sigma(\Omega)) = \sigma(\Omega),$ we get that 
$v\in \sigma(\Omega) \subseteq Z.$ Therefore $ V' =
\pi^{-1}\pi(V')$  and $V'$ is indeed $G$-saturated.

   As noted in \cite[Remark~(2), page~98]{Luna}, even though the map $G \times_Z (\h_{\reg} \times U) \to V$ need not be  bijective onto its image, it  follows from  \cite[Lemme
   fondamental, II.2, p.94]{Luna} that for each $x \in \h_{\reg}$ the induced map $G \times_Z (\{ x \} \times U) \to \pi^{-1}(\pi(x))$ is a bijection. In particular, if $x + u = g\cdot (x + u')$ in $V$ for some $u,u' \in U$ then $g \in G_x = Z$.

   Now we return to $G \times_{N}(\h \oplus U)$, and again set   $V' = \Im(\sigma)$. Then 
   \cite[Lemme, p.87]{Luna} implies that $V' = \pi^{-1}(A)$ for some open subset $A$
   of $\h / W$. Necessarily, $\h_{\reg} / W \subseteq A$, so
   $V_{\reg} \subseteq V'$. If this inclusion is proper then the
   fact that $V'$ is $G$-saturated means that there exists a
   $G$-orbit $\mc{O}'$ in $V' \smallsetminus V_{\reg}$ that is
   closed in $V$. Then $\sigma^{-1}(\mc{O}')$ is a (finite) union of
   closed orbits in $G \times_{N}(\h_{\reg} \times U)$. Let
   $\mc{O}$ be one of these closed orbits. Then
   $\mc{O} = G \times_N \mc{O}_0$ for some closed $N$-orbit 
   $\mc{O}_0\subseteq \h_{\reg} \times U$.  Since $U\git Z = \{ \mr{pt} \}$,
    the closed $Z$-orbits in $\h_{\reg} \times U$ are precisely the
   sets $\{ (x,0) \}$ for $x \in \h_{\reg}$.  Thus  the closed $N$-orbits in $\h_{\reg} \times U$
   are just  the ones of the form $N \cdot (x,0)$ for
   $x \in \h_{\reg}$. Hence
   $\mc{O}' \subseteq V_{\reg}$ and so   $\sigma$ 
   surjects onto $V_{\reg}$.

   Since $V_{\reg}$ is smooth, in order  to prove  that $\sigma$ is an isomorphism, 
    it suffices to show that  it is injective. Let
   $[g,(x,u)] \in G \times_{N}(\h_{\reg} \times U)$ and suppose  that
   $g\cdot (x + u) = g_1\cdot (x_1 + u_1)$ with
   $[g_1,(x_1,u_1)] \in G \times_{N}(\h_{\reg} \times U)$. Without loss of generality, $g = 1$. 
   The unique closed orbit in $\overline{G \cdot(x+u)}$ is   $G \cdot x$. Therefore, the identity 
   $x + u = g_1\cdot (x_1 + u_1)$ implies that   $G \cdot x = G \cdot x_1$. It follows that
   $W\cdot x = G\cdot x \cap \h  = G\cdot x_1 \cap  \h =   W \cdot x_1$  and
   hence  $x_1 =k \cdot x$ for some $k \in N$. Thus, 
   $$
   x + u = (g_1k)\cdot x + g_1\cdot u_1 = (g_1 k)\cdot (x + u_2),
   $$
   where $u_2 = k^{-1} \cdot u_1$.    By  Lemma~\ref{lem:stabilizerpolarh} this forces $z:=g_1 k \in G_x = Z$.  
    It follows that $ x + u = x + z \cdot u_2$ and hence $u = z\cdot u_2 = (z k^{-1})\cdot  u_1$. Putting 
    all this together gives
        \begin{align*}
                [g_1,(x_1,u_1)] & \ =\  [zk^{-1},(x_1,u_1)]
                 \ =\ [1,(zk^{-1}\cdot x_1, zk^{-1} \cdot u_1) ]\\
                &\ \ =\ [1,(z \cdot x,u)] \ =\ [1,(x,u)],
        \end{align*}
       as required.
 \end{proof}

\begin{lemma}\label{lem:thetaZtrivunstable}
       Assume that $\delta_i$ is a  $\theta$-semi-invariant irreducible factor of $\delta$ for some 
       character $\theta\in \mathbb{X}^*(G)$.  Then:
        \begin{enumerate}
                \item                  $\theta(Z) = 1$;
                \item under the identification $V_{\reg} \cong G \times_N(\h_{\reg} \times U)$, 
                we have $\delta_i |_{V_{\reg}} = F_{\theta} \otimes h_i \o 1$ for some $W$-semi-invariant $h_i\in \C[\h_{\reg}]$. 
        \end{enumerate}
\end{lemma}

\begin{proof}
        (1) By definition, $\delta$ is nowhere vanishing on $V_{\reg}$ and so $\delta_i$ does not vanish on 
        $V_{\reg}$.  Fix   $x \in \h_{\reg}$. Since $\h_{\reg} \subset V_{\reg}$, we have $\delta_i(x) \neq 0$.   
        Then, for all $z \in Z$, 
        $$
        \delta_i(x)= \delta_i(z\cdot x) = \theta(z) \delta_i(x) .
        $$
        Thus $(1-\theta(z))\delta_i(x)=0$. Since $\delta_i(x) \neq 0$ we deduce that $\theta(z) = 1$. 
        
        (2) Consider the left regular action of $G$ on itself. As $\theta$
        is a linear character of $G$,  the algebraic Peter-Weyl
        Theorem \cite[Theorem 4.2.7]{GoodmanWallach} implies that, up
        to scalar, there is a unique $\theta$-semi-invariant function
        $F_{\theta} \in \mc{O}(G)$. If the group $N$ acts by (inverse)
        multiplication   on $G$ from the right, then $F_{\theta}$ is a
        $\theta^{-1}|_N$-semi-invariant. Then each
        $\theta$-semi-invariant function $\alpha$ on
        $V_{\reg} = G \times_N (\h_{\reg} \times U)$ is of the form
        $F_{\theta} \o f_\alpha$ for some $N$-semi-invariant function $f_\alpha\in \C[\h_{\reg} \times U]$, with character $\theta|_N$. Note that, for $\alpha=\delta_i$, we have 
        $f_\alpha = \delta_i |_{\h_{\reg} \times U}$. By  (1), the character $\theta|_N$ factors through 
        $W$ and so $\delta_i |_{\h_{\reg} \times U}$ is $Z$-invariant. Hence, it is the pullback to 
        $\h_{\reg} \times U$ of an invertible $W$-semi-invariant $h_i\in \C[\h_{\reg}]$.  
\end{proof}

\subsection*{Localisation of radial maps}

In this subsection, we consider the localization of $\rad_{\vs}$ to
$V_{\reg}$.  
We write $\mf{z}= \text{Lie}(Z)=\text{Lie}(N)$ and let $\chi$ be the character associated to $\vs$, as in \eqref{chi-defn}. Recall from
Proposition~\ref{lem:unstablepolaropen} that
$V_{\reg} \cong G \times_{N}(\h_{\reg} \times U)$.

By Lemma~\ref{lem:thetaZtrivunstable}(1), $\chi |_{\mf{z}} = 0$. 
Therefore, by \cite[Lemma~9.1.2]{BG} that there is an
identification
\begin{equation}\label{eq:qhrprincipalspace}
        \psi_{\chi} \colon (\dd(\h_{\reg} \times U) / \dd(\h_{\reg} \times U) \tau_U(\mf{z}))^N
        \stackrel{\sim}{\longrightarrow} (\dd(V_{\reg}) / \dd(V_{\reg})
        \g_{\chi})^G, 
\end{equation}
where $\psi_{\chi}$ sends $D \in \dd(\h_{\reg} \times U)^N$ to the image
of $1 \o D \in \dd(G \times \h_{\reg} \times U)^{G \times N}$ in the right
hand side. 
The algebra $(\dd(\h_{\reg} \times U) / \dd(\h_{\reg} \times U) \tau_U(\mf{z}))^N$ acts on $\C[\h_{\reg} \times U]^N \cong \C[\h_{\reg}]^W$. This action defines a \textit{split} surjective algebra morphism 
$$
\nu \, \colon \, (\dd(\h_{\reg} \times U) / \dd(\h_{\reg} \times
U) \tau_U(\mf{z}))^N \lto \dd(\h_{\reg})^W, 
$$
with right inverse $\imath$ induced from the natural injection of $\dd(\h_{\reg})$ into $\dd(\h_{\reg}\times U)$.  This hold simply because  the action of   $D \in \dd(\h_{\reg} \times U)^Z$ on $\C[\h_{\reg} \times U]^Z \cong \C[\h_{\reg}]$ factors though $\dd(\h_{\reg})$.

We now come to the main result of this section, which also proves the final statement of 
 Theorem~\ref{thm:intro-existence}.

\begin{proposition}
\label{lem:localizeradialpartsunstable}
        The localised map
        $$
        \rad_{\vs} \colon (\dd(V_{\reg}) / \dd(V_{\reg})
        \mf{g}_{\chi})^G \lto \dd(\h_{\reg})^W
        $$
        is a split surjection with $\rad_{\vs} \circ \, \psi_\chi \circ \imath$
        equal to conjugation by $h^{-\vs}$ (in the sense made
        to be made precise in Sublemma~\ref{sublem:conj}).
\end{proposition}

\begin{proof}
        Since $V_{\reg} = G \times_N (\h_{\reg} \times U)$ by Proposition~\ref{lem:unstablepolaropen} and
        $\mf{z}=\text{Lie}(N)$ acts trivially on $\h_{\reg}$, one has
        \begin{equation*} \label{eq:localizeradialparts}
                (\dd(V_{\reg}) / \dd(V_{\reg}) \mf{g}_{\chi})^G =
                \left(\frac{\dd(G \times \h_{\reg} \times U)}{\dd(G \times
                        \h_{\reg} \times U) \mf{g}_{\chi} + \dd(G \times \h_{\reg} \times U)
                        \tau_0(\mf{z})} \right)^{G \times N},
        \end{equation*} 
        where $\tau_0 \colon \mf{z} \to \dd(G \times U)$ is the
        derivative of the diagonal $Z$ action on $G \times U$ and
        the action on $G$ is by right multiplication. The
        isomorphism \eqref{eq:qhrprincipalspace} says that every
        class in
        $(\dd(V_{\reg}) / \dd(V_{\reg}) \mf{g}_{\chi})^G$ can be
        represented as $1 \o D$ for a unique element
        $D \in (\dd(\h_{\reg} \times U) / \dd(\h_{\reg} \times U)
        \tau_U(\mf{z}))^N$. Such an operator will only act on the
        second factor of a function of the form $F \o f$, where
        $f \in \C[\h_{\reg} \times U]$. For clarity, write
        $\beta \colon \C[\h_{\reg}] \iso \C[\h_{\reg} \times
        U]^Z$. For $f \in \C[\h_{\reg}]$, we have
        $\beta(f)(x,u) = f(x)$ and hence $\beta(f) = f \o 1$. If
        $f \in \C[\h_{\reg}]$ is a $W$-semi-invariant,
        $D $ is as above          and $r \in \C$,  then
        $$
        \nu(\beta(f)^{-r} D \beta(f)^r) = f^{-r} \nu(D) f^r.
        $$
        In other words, $\nu((f \o 1)^{-r} D (f \o 1)^r) = f^{-r} \nu(D) f^r$.

        \begin{sublemma}\label{sublem:conj}  If $D$ is defined as above, then 
                $\rad_{\vs}(1 \o D) =h^{-\vs} \circ \nu(D) \circ h^{\vs}$.
        \end{sublemma}

        \begin{proof}
               Proposition~\ref{lem:unstablepolaropen} induces an isomorphism
                $\alpha \colon \C[\h_{\reg}]^W \iso \C[G \times_N
                (\h_{\reg} \times U)]^G$ and we identify the target with
                $ \C[G \times \h_{\reg} \times U]^{G \times N}$, where $N$ acts
                diagonally. If $z \in \C[\h_{\reg}]^W$ then
                $\alpha(z)(g,x,u) = z(x)$; that is,
                $\alpha(z) = 1 \otimes z \otimes 1 \in \C[G \times \h_{\reg} \times U]^{G
                        \times N} $.  Now we compute
                \[
                \rad_{\vs}(1 \otimes D)(z) = (\deltav^{-\vs}(1 \otimes
                D)(1 \otimes z \o 1) \deltav^{\vs}) |_{\h_{\reg}} .
                \]
             Recall from Section~\ref{sec:radialpartsmap1} that the irreducible factors of $\deltav$ are $\{\deltav_1,\dots, \deltav_k\}$, where $\deltav_j$ is a semi-invariant of weight $\theta_j$.   
             Since $\deltav_i$ is $Z$-invariant,  Lemma~\ref{lem:thetaZtrivunstable}(2) implies that 
                $\deltav_i |_{V_{\reg}} = F_{\theta_i} \o h_i \o 1$, where
                $h_i = \deltav_i |_{\h_{\reg}}$. Therefore,
                \[\deltav^{\vs} \ = \ \deltav_1^{\vs_1}\cdots\deltav_k^{\vs_k}  \ = \  F_{\theta_1}^{\vs_1} \cdots
                F_{\theta_k}^{\vs_k} \otimes h_1^{\vs_1} \cdots
                h_k^{\vs_k} \o 1\] and so
                $$ 
                \begin{aligned}   
                        (1\otimes D)\left( (1\otimes z \o 1) \deltav^{\vs}
                        \right) & \ = \ F_{\theta_1}^{\vs_1}
                        \cdots F_{\theta_k}^{\vs_k} \otimes \left(D(z
                        h_1^{\vs_1} \cdots h_k^{\vs_k} \o 1 )\right)  \\ 
                        & \ = \ F_{\theta_1}^{\vs_1} \cdots
                        F_{\theta_k}^{\vs_k} \otimes \left(D (z h^{\vs}
                         \o 1)\right) \\ 
                        & \ = \  \left( 1\otimes (h^{-\vs} \nu(D)
                        h^{\vs})(z) \o 1 \right) F_{\theta_1}^{\vs_1} \cdots
                        F_{\theta_k}^{\vs_k} \otimes h_1^{\vs_1} \cdots
                        h_k^{\vs_k} \o 1 \\ 
                        & \ = \ \left(1\otimes (h^{-\vs} \nu (D) 
                        h^{\vs})(z)\otimes  1 \right)
                        \deltav^{\vs}.  \end{aligned}
                $$  
                Hence
                $\left(\deltav^{-\vs} (1\otimes D)( 1\otimes z \o 1)
                \deltav^{\vs}\right) |_{\h_{\reg}} = 1\otimes
                (h^{-{\vs}} \nu(D) h^{\vs}) (z) \o 1 $, as required.
        \end{proof}
        
        We return to the proof of the  proposition. By the   sublemma, if $D \in \dd(\h_{\reg})^W$ then
        \begin{equation} \label{eq:conjugation}
        \rad_{\vs}(1 \o \imath(D)) = h^{-\vs} \circ \nu(\imath(D)) \circ
        h^{\vs} = h^{-\vs} \circ D \circ      h^{\vs}
        \end{equation}
        or, equivalently,
        $\rad_{\vs} \circ \, \psi_{\chi} \circ \imath$ equals
        conjugation by $h^{-\vs}$. Since this is an
        automorphism of $\dd(\h_{\reg})^W$, it follows that
        $\rad_{\vs}$ is a split surjection.
\end{proof}

We continue with the notation of
Proposition~\ref{lem:localizeradialpartsunstable}. Let $\vs, \vs' \in \C^k$
with associated characters $\chi,\chi'$ as in~\eqref{chi-defn}.

\begin{corollary}\label{lem:invconjugationvs}
        Assume that $\chi = \chi'$ and set $\xi = \vs'-\vs$.
        \begin{enumerate}
                \item
                $\rad_{\vs'}(D) = h^{-\xi} (\rad_{\vs}(D)
                )h^{\xi}$ for all $D \in \dd(V)^G$.
                \item Therefore, 
                $\mr{Im}(\rad_{\vs'}) \ = \  h^{-\xi} (\mr{Im}(\rad_{\xi}))
                h^{\xi}  \ \cong \  \Im(\rad_{\vs})$.
                \item  Moreover,  $\ker(\rad_{\vs'}) = \ker(\rad_{\vs})$.
        \end{enumerate}
\end{corollary}

\begin{proof}  Parts~(2) and (3) both follow from Part (1). 

In order to prove Part~(1), 
  let $D \in \dd(V)^G$; as seen in  Proposition~\ref{lem:localizeradialpartsunstable}, the class
  of $D$ in $(\dd(V_{\reg}) / \dd(V_{\reg}) \mf{g}_{\chi})^G$ can
  be represented by $1 \otimes D_0$ where
  $D_0 = \imath(D)\in \dd(\h_{\reg} \times U)^N/ \dd(\h_{\reg} \times U)
        \tau_U(\mf{z})^N$.  Note that the choice of
  $1\otimes D_0$ depends on $\chi$ but not $\vs$. Using Sublemma~\ref{sublem:conj} twice we get:
\begin{equation*}
\rad_{\vs'}(1\otimes D_0 )  =
(h^{-\vs'} h^{\vs})   h^{-\vs} \nu(D_0) h^{\vs}   (h^{-\vs} h^{\vs'})
= h^{-\xi}   \rad_{\vs}(1\otimes D_0)   h^{\xi}.
\end{equation*}
Thus, using ~\eqref{eq:conjugation},  $\rad_{\vs'}(D) = h^{-\xi}
(\rad_{\vs}(D))h^{\xi}$. 
\end{proof}

One significance of Corollary~\ref{lem:invconjugationvs} is the following result. 

\begin{corollary}
\label{rem:invconjugation} If $V$ is a   polar representation of a semisimple group  $G$, then
  the algebras $\mr{Im}(\rad_{\vs})$ are all isomorphic
to $\mr{Im}(\rad_{0})$, and they all have the same kernel.  
\end{corollary}

\begin{proof}  If $G$ is semisimple, then   $\mathbb{X}^*(G)=\{1\}$   and hence
 the character $\chi$ associated to any $\vs \in \C^k$ is~$0$.  Now apply Corollary~\ref{lem:invconjugationvs}.
 \end{proof}
 
   With regard to  Corollary~\ref{lem:invconjugationvs},
     it  is important to note that for  general $\chi' \neq \chi$  one can have   $
\mr{Im}(\rad_{\vs'}) \not\cong \theta^{-1}(\mr{Im}(\rad_{\vs}))
\theta  $.
This is because
the image of $1\otimes D$   in the proof of  Proposition~\ref{lem:localizeradialpartsunstable}, depends
 upon $\chi$.  For example, as \cite[Corollary~14.15]{BNS}  shows for a particular representation, one can easily obtain
  both simple and non-simple rings as $\Im(\rad_{\vs})$ by varying $\vs$.

 \subsection*{Stable Representations}\label{defn:stable}
We end the section by briefly discussing stable representations. 

Recall from \eqref{eq:msdefnorbit} that $m$ is the maximum dimension over all orbits in $V$ and $s$ the maximum dimension over all semisimple orbits. As there, the set of semisimple elements  in $V$ is denoted  by  $V_{\mathrm{ss}}$.

Define the set of \emph{stable} elements in $V$ to be
$V_{\mathrm{st}} = \{v \in V_{\mathrm{ss}} : \dim G \cdot v=
m\}$. It is known (and not difficult to see) that
$V_{\mathrm{st}}$ is an open subset of $V$.
The  representation $(G,V)$ is said to be \textit{stable} if
there is a closed orbit whose dimension is maximal amongst all
orbits. Equivalently, the representation
is stable if $V_{\mathrm{st}} \neq \emptyset$ and one then has  
$m=s$ and $V_{\st} = V_{\mathrm{st}}$. Observe also that $\pi^{-1}(\pi(x))= G \cdot
x$ if $x \in V_{\mathrm{st}}$ since $\pi^{-1}(\pi(x))$ contains a unique closed orbit.  
 
By  \cite[Corollary~2.5]{DadokKac},   $V$  is stable if only
if $ V= \h \oplus \g \cdot \h$.   Equivalently, $V$  is polar if and only if the space $U$ defined at the beginning of the section is  zero. Thus for polar representations  we obtain the following simpler versions  of   Propositions~\ref{lem:unstablepolaropen}  and~\ref{lem:localizeradialpartsunstable}.

 \begin{corollary} \label{cor:unstablepolaropen}   Assume that $(G,V)$ is a stable polar representation. 
        Then there is an isomorphism   \[G \times_{N} (\h_{\reg})\  \isom \ 
        V_{\reg},\]
        given by $[g,(x,u)] \mapsto  g\cdot (x+u)$.\qed
 \end{corollary}

\begin{corollary}
\label{cor:localizeradialpartsunstable}
      If $V$ is a stable representation then  the localised map
        $$
        \rad_{\vs} \colon (\dd(V_{\reg}) / \dd(V_{\reg})
        \mf{g}_{\chi})^G \lto \dd(\h_{\reg})^W
        $$
 is an isomorphism.  \end{corollary}

\begin{proof} When $V$ is stable, $U=0$ and so \eqref{eq:qhrprincipalspace} simplifies to give an isomorphism 
\[   \psi_{\chi} \colon (\dd(\h_{\reg})^W
        \stackrel{\sim}{\longrightarrow} (\dd(V_{\reg}) / \dd(V_{\reg})
        \g_{\chi})^G.\] 
Thus the result follows from         Proposition~\ref{lem:localizeradialpartsunstable}. 
\end{proof}


\section{The radial parts map: surjectivity}\label{sec:surjectivity}
 
In Theorem~\ref{thm:radpartsmappolar} we proved the existence of
radial parts maps $\rad_{\vs}$. In this section, we prove that
$\rad_{\vs}$ is surjective in  many  cases of interest. Specifically, by combining
Theorems~\ref{cor:simpleLSsurjective}
and~\ref{thm:symmLSsurjective} with  Corollary~\ref{more-cases}, we
obtain the following result.

\begin{theorem}\label{thm:radsurjects} Let $V$ be a polar
representation for the connected, reductive group $G$, with a
radial parts map $\rad_{\vs}: \dd(V)^G \to \Ak(W)$, for some
spherical  algebra $\Ak(W)$.  Then $\rad_{\vs}$ is
surjective in each of  the following cases:
        \begin{enumerate}
                \item when $\Ak(W)$ is a simple algebra;
                \item when $V$ is a symmetric space, in the sense
of Section~\ref{Sec:examples};
                \item when the associated complex reflection
group $W$ is a Weyl group with no summands of type $\mathsf{E}$ or $\mathsf{F}$.
        \end{enumerate}
\end{theorem}

The conclusions of the theorem also hold when  $(V,G)$ has rank one and 
 when $V$ is a visible stable locally free representation; see Corollaries~\ref{thm:Levsurjectiverank1}
and~\ref{cor:vislocallyfreepolargr}, respectively. This of course raises the question  of whether 
  $\rad_{\vs}$ is surjective for every polar representation. The authors know of no counterexample to this question  and  do not even have an opinion as to whether  such an example should exist.

\subsection*{The surjectivity of $\rad_{\vs}$ when $\Ak(W)$ is
simple}

Much of the proof of this case involves the relationship between
the Euler gradation and the order filtration of $\dd(V)^G$ and so
we begin with the relevant notation.

\begin{notation}\label{notation:gradings} As in
Notation~\ref{notation:chidot}, let $\eu_V\in \dd(V)$ denote the
Euler operator on $V$, so that $[\eu_V,x_i] = x_i$ for coordinate
vectors $x_i \in V^*\subset \C[V]$, and let $\eu_\h\in \dd(\h)$
be the analogous operator on $\h$. Conjugation by these Euler
operators induces $\BZ$-graded structures on the rings of
differential operators, and hence on $\Ak=\Ak(W)$.  See, for example,
\cite[Section~2.4, p.186]{BellamySRAlecturenotes}.  Explicitly,
$$
\dd(V)=\bigoplus_{k\in \BZ} \dd(V)_k \qquad \text{where}\quad 
        \dd(V)^G_k = \{ D \in \dd(V)^G \, :\, [\eu_V,D] = - k D\}.
$$

In order to distinguish the Euler degree from other degree functions, we will write $\degeu d=k$ if $d\in \dd(V)_k$        
The analogous definition applies to $\dd(\h)$ and related
rings, and we still use the notation $\degeu$ for the corresponding degree function. In particular, as the discriminant $\deltah$ is
homogeneous, $\dd(\h_{\reg})^W = \bigoplus_{n\in \BZ}
\dd(\h_{\reg})^W_k $ where $ \dd(\h_{\reg})^W_k = \{ D \in
\dd(\h_{\reg})^W \, : \, [\eu_{\mf{h}}, D] = - k D\}. $ When
restricted to $\Ak$, the Euler gradation can also be defined
by setting $\degeu y = 1, \degeu x = -1$ and $\degeu w = 0$ for $x \in
\h^*$, $y \in \h$ and $w \in W$.
\end{notation}

Recall the definitions of the order filtration and $\dd_\ell(X)$  from Notation~\ref{defn:order}.
 The algebra $\Ak$ inherits the order
filtration from $\dd(\h_{\reg})^W$, which we will write as
$\Ak = \bigcup \ord_{\leq n} \Ak$. 
As explained in
Remark~\ref{rem:orderfiltrationH}, this is the same filtration
coming from the realisation $\Ak = e \Hk(W) e$ of $\Ak$ as
a subalgebra of $\Hk(W)$, where $\Hk(W)$ is filtered by placing
$\C[\h]$ and $W$ in degree zero and $ \h\smallsetminus\{0\}$ in
degree one.

The following lemma is part of the folklore but we include a
proof as we do not know an appropriate reference.

\begin{lemma}\label{lem:filteredhom} The morphism $\rad_{\vs}:
        \dd(V)^G \to \Ak$ is a filtered homomorphism under the two
        order filtrations.
\end{lemma}

\begin{proof}
        Recall from \eqref{eq:radvzfactor} that we can factorize $\rad_{\vs} = \rad_0 \circ \gamma_{\vs}$, where 
        $\gamma_{\vs}$ is conjugation by $\deltav^{-\vs}$. By construction, the morphisms $\gamma_{\vs}$ and $\rr$
        are filtered isomorphisms while restricting a differential
        operator to a subalgebra certainly does not increase the order of
        an operator. Thus, $\rad_{\vs}$ is also a filtered morphism.
\end{proof}

Using these observations it is easy to determine the image
$\rad_{\vs}(\eu_V)$.  Recall that  $\rr \colon \C[V]
\to \C[\h]$ denotes the restriction of functions from $V$ to $\h$, with  the induced Chevalley isomorphism
 again written $\rr: \C[V]^G \stackrel{\sim}{\longrightarrow}\C[\h]^W$.
 
\begin{lemma}\label{lem:RadchiZgraded}  
               {\rm (1)}  There exists $d\in \mathbb{C}$ such that
$\rad_{\vs}(\eu_V) = \eu_{\h} + d$.

            {\rm (2)}     The morphism $\rad_{\vs} \colon \dd(V)^G
\to \dd(\h_{\reg})^W$ preserves the  Euler grading.
  \end{lemma}

\begin{proof} (1) This is similar to the proof of
Corollary~\ref{cor:Imradrank1}.  Recall from \eqref{eq:twistedop}
that $\deltav^{\vs} = \deltav_1^{\vs_1}\cdots \deltav_k^{\vs_k}$.
Hence $\eu_V \cdot \deltav^{\vs} = \Bigl( \sum_{i} \vs_i d_i
\Bigr) \deltav^{\vs} = d \deltav^{\vs}, $ where $d_i =\degeu
\delta_i$ and $d=\sum \vs_i d_i$.  Next, let $a\in
\C[\h_{\reg}]^W$ be homogeneous, say  of degree $p$. Since
$\rr^{-1}(a)$ is still homogeneous of degree $p$, it follows that
$\eu_V\left(\rr^{-1}(a) \deltav^{\vs}\right) =(p+d) \rr^{-1}(a)
\deltav^{\vs} .$ Therefore, by the definition of $\rad_{\vs}$ in
\eqref{eq:radialpartsmap2},
$$
\rad_{\vs}(\eu_V) (a)=
        \left(\deltav^{-\vs}(\eu_V(\rr^{-1}(a)\deltav^{\vs}))
        \right)|_{ \h_{\reg}} = (d+p)a.
        $$ 
Thus $\rad_{\vs}(\eu_V) = \eu_{\h} + d$.           
        
        (2) Since the two  Euler  gradings are defined by
conjugation by the respective Euler operators, this follows from
Part~(1).
\end{proof}

We need to be precise about the filtrations and associated graded
objects used in this paper and so we will use the following
conventions.

\begin{notation}\label{defn:order1} Let
$\mathfrak{A}=\bigcup_{n\geq 0} \Lambda_n$ and
$\mathfrak{B}=\bigcup_{n\geq 0} \Gamma_n$ be filtered $\C$-algebras
with a filtered morphism $\phi: \mathfrak{A}\to \mathfrak{B}$;
thus $\phi(\Lambda_n)\subseteq \Gamma_n$ for all $n$.  Write the
associated graded ring $\gr_\Lambda \mathfrak{A}= \bigoplus
\overline{\Lambda}_n $, where
$\overline{\Lambda}_n=\Lambda_n/\Lambda_{n-1}$ and define
\emph{the principal symbol maps} $\sigma_n: \Lambda_n\to
\overline{\Lambda}_n\subseteq \gr_\Lambda \mathfrak{A}$ by
$\sigma_n(x) = [x+\Lambda_{n-1}]$ for $x\in \Lambda_n$.  In this
paper any given ring $\mathfrak{A}$ will only have one filtration
on it and so we can write $\gr_\Lambda \mathfrak{A} = \gr
\mathfrak{A}$ without ambiguity.  The associated graded map
$\gr_{\phi} : \gr \mathfrak{A}\to \gr \mathfrak{B}$ is defined by
$\gr_{\phi}(\sigma_n(x)) = \sigma_n(\phi(x))\in
\overline{\Gamma}_n$ for any $x\in \Lambda_n$.  The restriction
of $\gr_{\phi}$ to the $\overline{\Lambda}_n$ will still be
denoted by $\gr_{\phi}$.  The commutativity $\sigma_n\circ\phi =
\gr_\phi \circ \, \sigma_n$ will be used without comment.

        Finally, the morphism $\phi$ is called \emph{filtered
surjective}, respectively \emph{filtered isomorphism} if each map
$\phi: \Lambda_n\to \Gamma_n$ is surjective, respectively an
isomorphism. Equivalently, the associated graded map $\gr_\phi$
is surjective, respectively, an isomorphism.
\end{notation}

\begin{lemma}\label{lem:filteredgradeddegreeinvar}  Assume that
$0\not=f \in \Ak(W)$  is  homogeneous, with  Euler degree
$\degeu f =k$, and   $\ord f=\ell$ for some $k\ge \ell \ge 0$.  Then $k=\ell$ and $f \in (\Sym
\h)^W_k$.
\end{lemma}

\begin{proof} It is easiest to work in $\Hk(W)$, with its natural
  extension of the order filtration given  by Remark~\ref{rem:orderfiltrationH}. Let $y_1, \ds,
y_n$ be a basis of $\h$. We can decompose $f$ uniquely as
        $$
        f = \sum_{\alpha,w,i} f_{\alpha,w,i}(x) y^{\alpha} w,
        $$
        where $y^\alpha = y_1^{\alpha_1}\cdots y_n^{\alpha_n}$ for $\alpha=(\alpha_1,\dots, \alpha_n) \in \mathbb{N}^n$, while $w \in W$, $i \in
\mathbb{Z}$ and $|\alpha| \le \ell$. Moreover, $f_{\alpha,w,i}(x)
\in \C[\h]$ is  chosen to be homogeneous of  degree $i$ (and hence Euler degree $-i$).
        
        Clearly $\ord(y^{\alpha})= \degeu y^{\alpha} = |\alpha|$
while $\degeu \left(f_{\alpha,w,i}(x)\right) <0$ for $i \not= 0$.  
Thus, as $f$ is homogeneous  with $\degeu f = k$ the only terms
$f_{\alpha,w,i}(x) y^{\alpha} w$ appearing in $f$ will have
$|\alpha|-i=k$. Since   $\ord(f)=\ell\leq k$, the
only terms that can appear in $f$ will therefore have
$|\alpha|=k$ and $i=0$.  Since $f\not=0$, this forces $\ell=k$
and then, as $f_{\alpha,w,0}(x) \in \C$, we obtain $f \in (\Sym
\h)_k \o \C W$. Finally since $f \in e \Hk(W) e \subset \Hk(W)$,
this means $f \in (\Sym \h)^W_k e = (\Sym \h)^W_k$.
\end{proof}

\begin{notation}\label{rrp-defn}   Let $v
\in \h_{\reg}$. Let $U$ be the orthogonal complement, with respect to $( - , - )$,  to $\h \oplus \mf{g} \cdot v $ in $V$. Then $V = \h \oplus \mf{g} \cdot v \oplus U$ is a decomposition of $V$ as an $N_G(\h)$-module. Recall that the summand $\mf{g} \cdot v = \g \cdot \h$ is
independent of the choice of $v \in \h_{\reg}$ by
Lemma~\ref{lem:stabilizerpolarh}. 
The projection $V \twoheadrightarrow \h$ along $\mf{g} \cdot v \oplus U$ defines an algebra
surjection $\rrp \colon \Sym \, V \twoheadrightarrow \Sym \,
\h$. By \cite[Proposition~3.2]{BLT}, $\rrp$ restricts to a graded
algebra isomorphism $\rrp:(\Sym \, V )^G
\stackrel{\sim}{\longrightarrow} (\Sym \, \h)^W$.
\end{notation}

\begin{theorem}\label{prop:Radchirestiso} The radial parts map
$\rad_{\vs} \colon \dd(V)^G \to \Ak(W)$ restricts to give the
graded isomorphisms $\rr \colon \C[V]^G
\stackrel{\sim}{\longrightarrow} \C[\h]^W$ and $\rrpp \colon
(\Sym \, V)^G \stackrel{\sim}{\longrightarrow} (\Sym \, \h)^W$. 
\end{theorem}

\begin{remark}\label{rem:Radchirestiso} We do not claim that $\rrpp = \rrp$ in
Theorem~\ref{prop:Radchirestiso}, though we have no reason to
doubt this equality.
\end{remark}

\begin{proof} It is clear from its definition that $\rad_{\vs}$
is just the restriction isomorphism $\rr \colon \C[V]^G
\stackrel{\sim}{\longrightarrow} \C[\h]^W$ when applied to
$\C[V]^G$, so is certainly a graded isomorphism.
        
We now turn to $(\Sym V)^G$.  Here, we first prove the
injectivity of $\rrpp$ for $\rad_0$; that is, when $\vs=0$.  This
step follows from \cite[Corollary~5.10]{Schwarz}, but we prefer
to give an alternate proof.  All the  earlier  results of this paper 
 apply to the case $\vs=0$; thus $\Im(\rad_0)\subseteq
A_{\kappa(0)}(W)$ for some $\kappa(0)$.
        
        Let $0\not=D\in (\Sym V)^G_n$. As noted in  equations \eqref{eq:radvzfactor1} and \eqref{eq:radvzfactor},
   $\rad_0 = \rrr\circ\eta$  where $\rrr $ is the identification $\dd(V_{\reg}\git G) =\dd(\h_{\reg}\git W)$
       and $\eta:\dd(V_{\reg})^G\to \dd(V_{\reg}\git G)$ denotes restriction of functions.
        Therefore, in order to prove that
$\rad_0(D)\not=0$, it suffices to prove that $\eta(D)\not=0$ or,
equivalently, to find $0\not=f\in \C[V]^G$ with $D\ast f\not=0$.
Now $\C[V]$ is an equivariant left $\dd(V)$-module and so the
morphism $\Sym V \otimes \C[V] \to \C[V]$ given by $(E, f)\mapsto
E\ast f$ is $G$-equivariant. There certainly exists $f\in \C[V]$
such that $D\ast f \in \C\smallsetminus\{0\}$ and, by taking an
isotypic component of $\C[V]$, we may further assume that $f$
lies in an irreducible $G$-representation. As $\C$ and $D$ lie in
the trivial $G$-isotypic component, necessarily $f$ does too. In
other words we have found $f\in \C[V]^G$ such that $D\ast
f\not=0$. Thus $\eta(D)\not=0$ and hence $\rad_0(D)\not=0$.
        
        By Lemma~\ref{lem:RadchiZgraded}, $\rad_0(D)$ is
homogeneous of graded degree $ \gr_{\eu_{\h}}(\rad_{\vs}(D)) = n$
and, by Lemma~\ref{lem:filteredhom}, $\ord_\h (\rad_0(D))\leq n$.
Thus, by Lemma~\ref{lem:filteredgradeddegreeinvar}, $\rad_0(D) \in
(\Sym \h)^W_n$, and hence $\ord_{\h}(D)=n$. Since $\rr$ is a
filtered isomorphism, this implies that the order $\ord_{V\modmod
G}(\eta(D))$ of $\eta(D)$ in $\dd(V\modmod G)$ also equals $n$.

        We now return to the case of general $\vs$ and again
consider $0\not=D\in (\Sym V)^G_n$. Thus $n=\ord_{V\modmod
G}(\eta(D))$ by the last paragraph. We can find elements $f_i\in
\C[V]^G$ such that the $n$-fold commutator $z=[\cdots
[\eta(D),f_1],\cdots], f_n]\not=0$. (When $n=0$, interpret this as
saying that $z=\eta(D)\in \C[V]^G$.)  In either case, by the
definition of order, $z\in \C[V]^G$.
        
        But $n$ is also equal to the order $\ord_VD$ back in
$\dd(V)$.  Thus, by the definition of order, $z'=[\cdots [D,
f_1],\cdots], f_n] \in \C[V]^G$, where the commutation now takes
place in $\dd(V)$.  Moreover, $z'\not=0$ as $\eta(z')=z$.  (In
fact $z=z'$, but this is not important.)  Now apply the ring
homomorphism $\rad_{\vs}$. By definition, $\rad_{\vs}$ restricts
to give the isomorphism $\rr \colon \C[V]^G
\stackrel{\sim}{\longrightarrow} \C[\h]^W$ when applied to
$\C[V]^G$. Thus
        \[ [\cdots [\rad_{\vs}(D),\rad_{\vs}(f_1)],\cdots
],\rad_{\vs}(f_n)] = \rad_{\vs}(z) \not=0.
        \] Therefore, $\rad_{\vs}(D)\not=0$ and so $\rad_{\vs}$
does indeed restrict to give an injection $(\Sym
V)^G\hookrightarrow \Ak(W)$.  Finally, as in the case $\vs=0$,
$\rad_{\vs}(D)$ is homogeneous of graded degree $
\gr_{\eu_{\h}}(\rad_{\vs}(D)) = n$ by
Lemma~\ref{lem:RadchiZgraded}, while $\ord_\h (\rad_0(D))\leq n$
by Lemma~\ref{lem:filteredhom}. Thus, by
Lemma~\ref{lem:filteredgradeddegreeinvar} $\rad_{\vs}(D) \in
(\Sym \h)^W_n$, with $\ord_{\h}(D)=n$.
        
        We have therefore proved that the restriction of
$\rad_{\vs}$ to $(\Sym V)^G_n$ gives an injection $\rrpp: (\Sym
V)^G_n \hookrightarrow (\Sym \h)^W_n$ for each $n$.  It remains
to prove that $\rrpp$ is bijective. However, as observed in
Notation~\ref{rrp-defn}, each graded piece $(\Sym V)^G_n$ is
finite dimensional, with $(\Sym V)^G_n\cong (\Sym \h)^W_n$. Thus
in order to prove that $\rrpp$ is bijective, it suffices to prove
that it is injective. Since this was the conclusion of the last
paragraph, $\rrpp$ is indeed a graded isomorphism.
\end{proof}

Theorem~\ref{prop:Radchirestiso} immediately implies one of the main aims of this
section, by proving Theorem~\ref{thm:radsurjects}(1).

\begin{theorem}\label{cor:simpleLSsurjective} If $\Ak(W)$ is a
simple ring, then the radial parts map 
$$ 
{\rad}_{\vs} \colon \left( \dd(V) / \dd(V) \g_{\chi} \right)^G \to \Ak(W)
$$
induced from $\rad_{\vs}$ is surjective. 
\end{theorem}

\begin{proof} Let $R=\Im(\rad_{\vs})$. Thus, by
Theorem~\ref{thm:surjectivereducerank1}, $R$ is a subalgebra of
$\Ak=\Ak(W)$. On the other hand, by Theorem~\ref{prop:Radchirestiso},
$R$ contains $\C[\h]^W$ and $(\Sym \h)^W$.  However, by
\cite[Theorem~4.6]{BEG} $\Ak =\C\langle \C[\h]^W,\,\,(\Sym
\h)^W\rangle$. Hence $R=\Ak$.
\end{proof}

There are numerous   cases of spherical algebras
where it is known that $\Ak$ is generated by
$\C[\h]^W$ and $(\Sym \h)^W$ and so, for these examples,  Theorem~\ref{prop:Radchirestiso} implies that $\rad_\vs$ is
also surjective. A particular case is given by the
following result, which is Part~(3) of Theorem~\ref{thm:radsurjects}.

\begin{corollary}\label{more-cases} Suppose that $W$ is a Weyl
group with no factors of type $\mathsf{E}$ or $\mathsf{F}$.  Then $\Ak(W)$ is
generated by $\C[\h]^W$ and $(\Sym \h)^W$ and so any
corresponding radial parts map $\rad_\vs$ is surjective.
\end{corollary}

\begin{proof} The first assertion follows from \cite{Wallach}, as
noted in the proof of \cite[Proposition~4.9]{EG}. Now apply
Theorem~\ref{prop:Radchirestiso}.
\end{proof}


\subsection*{The surjectivity of $\rad_{\vs}$ for symmetric
spaces}

Much of the proof for this case involves filtered and graded
techniques, so let $\mathfrak{A}=\bigcup \Lambda_n(\mathfrak{A})$
and $\mathfrak{B}=\bigcup \Gamma_{n}\mathfrak{B}$ be filtered
algebras with a filtered map $\phi: \mathfrak{A}\to
\mathfrak{B}$, and keep the conventions from
Notation~\ref{defn:order1}. We further assume that
$\mathfrak{A}_0=\mathfrak{B}_0=B$ is a commutative domain, with a
multiplicatively closed subset $\euls{C}\subset B$ which acts
ad-nilpotently on both $\mathfrak{A}$ and $\mathfrak{B}$. Then
$\euls{C}$ acts as an Ore set in each of $\mathfrak{A}$,
$\mathfrak{B}$, $\gr \mathfrak{A}$ and $\gr \mathfrak{B}$. We can
then filter the localisation $\mathfrak{A}_{\euls{C}}$ by
$\mathfrak{A}_{\euls{C}} = \bigcup \Lambda_{n,\euls{C}} $ for
$\Lambda_{n,\euls{C}}=B_{\euls{C}}\otimes_B\Lambda_n$, with the
analogous definitions apply in the other three cases.  One needs
to be wise to the fact that it is quite possible for there to
exist $x\in \Lambda_n\smallsetminus\Lambda_{n-1}$ for which $x\in
\Lambda_{n-1,\euls{C}}$; this happens precisely when there exists
$c\in \euls{C}$ such that $cx\in \Lambda_{n-1}$.

The next result is standard, but we give the proof since it will
be important later.

\begin{lemma}\label{lem:graded2} Keep the above notation. Then
        \begin{enumerate}
                \item The canonical map $\alpha: \gr
(\mathfrak{A}_{\euls{C}} )\to (\gr \mathfrak{A})_{\euls{C}}$ is
an isomorphism of graded rings.
                \item Let $\phi_{\euls{C}}:
\mathfrak{A}_{\euls{C}} \to \mathfrak{B}_{\euls{C}}$ be the
induced map. Then there is a commutative diagram of graded rings
                \begin{equation}\label{eq:graded2}
                        \begin{tikzcd} \gr
(\mathfrak{A}_{\euls{C}}) \arrow[d, "\alpha"] \arrow[r,
"\gr_{\phi_{\euls{C}}}"] & \gr (\mathfrak{B}_{\euls{C}})
\arrow[d, "\alpha"] \\ (\gr \mathfrak{A})_{\euls{C}} \arrow[r,
"(\gr_{\phi})_{_{\euls{C}}}"] & (\gr \mathfrak{B})_{\euls{C}}
                        \end{tikzcd}
                \end{equation}
        \end{enumerate}
\end{lemma}

\begin{proof} (1) Since $B_{\euls{C}}$ is flat over $B$, for each
$n$ we have an exact sequence
        \[ 0 \ \to \ B_{\euls{C}}\otimes_B \Lambda_{n-1}\ \to \
B_{\euls{C}}\otimes_B \Lambda_n \ \to \
B_{\euls{C}}\otimes_B\overline{\Lambda}_n \ \to \ 0.\]  By definition,
$\overline{\Lambda}_{n,\euls{C}} :=
\Lambda_{n,\euls{C}}/\Lambda_{n-1,\euls{C}} =
B_{\euls{C}}\otimes_B \Lambda_n / B_{\euls{C}}\otimes_B
\Lambda_{n-1}$.  Therefore, for each $n$, we have
well-defined isomorphisms of $B_{\euls{C}}$-modules
        $$\alpha_n:\overline{\Lambda}_{n,\euls{C}}=
        \Lambda_{n,\euls{C}}/\Lambda_{n-1,\euls{C}}       
        \ \buildrel{\sim}\over{\longrightarrow} \ B_{\euls{C}}
        \otimes \overline{\Lambda}_{n}.$$

        Finally, the graded morphism $\alpha: \gr
(\mathfrak{A}_{\euls{C}}) \to (\gr \mathfrak{A})_{\euls{C}}$ is
given by $\alpha=\bigoplus_{n} \alpha_n$, and so $\alpha$ is  
an isomorphism of $B$-modules and also of graded rings.

        (2) If $b\otimes x\in B_{\euls{C}}\otimes_B \Lambda_n$,
then $\phi_{\euls{C}}(b\otimes x)\in
B_{\euls{C}}\otimes_B\Gamma_n$. Thus $\phi_{\euls{C}}$ is a
filtered morphism and and so induces the graded morphism
$\gr_{\phi_{\euls{C}}} : \gr ( \mathfrak{A}_{\euls{C}}) \to \gr (
\mathfrak{B}_{\euls{C}})$.  The rest of the proof is a simple
diagram chase that is left to the reader.
\end{proof}

We now return to the rings 
$\mathfrak{A}=\dd(V)^G$ and $\mathfrak{B}=\Ak(W)$ with the
morphism $\rad_{\vs}$. We always take
$\euls{C}=\{\deltav^n\}_n \subset B=\C[V]^G\cong \C[\h]^W$ and will
usually write $ \mathfrak{A}_{\euls{C}} = \mathfrak{A}_{\reg}$,
etc.  Set $J=\ker(\rad_{\vs})$. Since
$\dd(\h_{\reg})$ is a domain, the kernel of $\rad_{\vs}:
\dd(V_{\reg})^G \to \dd(\h_{\reg})^W$  is equal to
$J\otimes_{\C[V]^G}\C[V_{\reg}]^G$ and so we can write this as
$J_{\reg}$ without ambiguity.

We will always use the order filtration $\ord=\ord_{\dd(V)}$ on $
\dd(V)$ with the induced filtration on both the subring
$\dd(V)^G$ and its factor ring $\dd(V)^G/J$, with the same
notation.  Similarly, we use the order filtration, again written
$\ord$, on $\dd(\h_{\reg})$ and its subalgebra $\Ak(W)$.
Finally, we use the filtration $\ord$ induced from $\dd(V)^G/J$
on the localisation $(\dd(V)^G/J)_{\euls{C}} =
\dd(V_{\reg})^G/J_{\reg}$.  We remark that there is a second
filtration on $\dd(V_{\reg})^G/J_{\reg}$ induced from the order
filtration on $\dd(V_{\reg})$ but, as the order of an element
$d\in \dd(V)$ does not change upon passage to $\dd(V_{
\reg})$ it follows easily that this equals our chosen
filtration.

As observed in Notation~\ref{defn:order1}, we have only one
choice of filtration on any given ring $\mathfrak{A}$ and so we
can write the associated graded ring as $\gr \mathfrak{A}$
without ambiguity.  Thus, in the above notation we obtain the
induced map
$$
\gr_{\rad_{\vs}} : \gr(\dd(V)^G/J) \to \gr(\Im(\rad_{\vs}))
\subseteq \gr \Ak(W) = (\C[\h]\otimes\Sym \h)^W.
$$
By Lemma~\ref{lem:graded2} this commutes with localisation at
$\euls{C}$. Lemma~\ref{lem:graded2}(2) implies the equality $ \gr(J_{\reg}) = (\gr J)_{\reg} $ and 
so we can write this as $\gr J_{\reg}$ without ambiguity.

\begin{notation}\label{grrad-notation} Set $C=(\C[V]\otimes \Sym
V)^G$, graded by the $(\C[V]\otimes (\Sym V)_{n})^G$ and filtered
by $\Lambda_{n,C}=\sum_{j\leq n}(\C[V]\otimes (\Sym V)_{j})^G$.
We use the analogous notation for $C_{\reg}$ and write
$\Gamma_\h=\bigcup \Gamma_{n,\Iring}$ for the analogous filtration on
$\Iring =(\C[\h]\otimes \Sym \h)^W$.  As in Notation~\ref{rrp-defn},
let $\rr\otimes\rrp : C\to \Iring$ be the map given by restriction on
the first factor and projection on the second factor.  Set $I
=\ker(\rr\otimes\rrp)$. 
\end{notation}

 Since $C=\gr \dd(V)^G$ and $\Iring=\gr \Ak(W)$, there exists a
second morphism $\gr_{\rad_{\vs}} \colon C\to \Iring$ induced from
$\rad_{\vs}$ and, in theory, this may differ from the map
$\rr\otimes\rrp$. However, the maps are closely related as the
next lemma describes. Let $Y = \mr{Im} \, \gr_{\rad_{\vs}}$, a subalgebra of $\gr\Ak(W) = \Iring$.

\begin{lemma}\label{lem:images} 
        Assume that $\gr J_{\reg} \subseteq
        I_{\reg} $.  
        \begin{enumerate}
                \item $ \gr J \subseteq I=L$, where $L= \ker\bigl(\gr_{\rad_{\vs}}: C\to \Iring\bigr)$.
                \item $\C[\h]^W\o(\Sym \h)^W\subseteq Y \subseteq
\Iring$ and hence $\Iring$ is a finitely generated module over the
noetherian domain $Y$.
                
                \item Moreover $\Iring_{\euls{C}}= Y_{\euls{C}}$ for
$\euls{C}=\bigl\{\deltah^n\bigr\}_n$ and so $\Iring$ and $Y$ have the same field of
fractions. In particular, they are equivalent orders.
        \end{enumerate}
\end{lemma}

\begin{remark}\label{rem:gradeofJ} 
        Assume that $\gr J_{\reg} = I_{\reg}$. By Lemma~\ref{lem:images}(1), $\gr_{\rad_{\vs}}$ induces a map $C/I\to \Iring$ with image $Y$. For simplicity, we also denote this map by $\gr_{\rad_{\vs}}$.
\end{remark}

\begin{proof}
        (1)  As the ring $\Iring$ is $\euls{C}$-torsionfree, $I=C\cap I_{\reg}$ and hence $\gr J \subseteq C\cap \gr J_{\reg} = C\cap I_{\reg} =
I.$ For the same reason, $L=L_{\reg}\cap C$.  Moreover, $\gr
J\subseteq L$ by the construction of $\gr_{\rad_{\vs}}$ and so
\[L \ = \ L_{\reg}\cap C \ \supseteq \ \gr J_{\reg} \cap
C \ = \ I_{\reg}\cap C \ = \ I,\] as required.

         (2) Let $D\in (\Sym \h)^W_n\subseteq \Ak(W)$.  By
Theorem~\ref{prop:Radchirestiso}, $\rad_{\vs}$ restricts to give
a filtered isomorphism $\rrpp \colon (\Sym V)^G
\buildrel{\sim}\over{\longrightarrow} (\Sym \h)^W$ and so $D\in
\rad_{\vs}((\Sym V)^G_n)$. Thus, $D=\sigma_n(D) \in
\Im(\gr_{\rad_{\vs}})$. By Theorem~\ref{prop:Radchirestiso},
again, $\rad_{\vs}$ restricts to give the filtered isomorphism
$\rr: \C[V]^G \buildrel{\sim}\over{\longrightarrow} \C[\h]^W$ and
so the same argument ensures that $\C[\h]^W\subset
\Im(\gr_{\rad_{\vs}})$. Therefore, $\C[\h]^W\otimes (\Sym
\h)^W\subseteq Y\subseteq \Iring$. Since $\C[\h]\otimes \Sym \h$ is a
finitely generated $\C[\h]^W\otimes (\Sym
\h)^W$-module, Part~(2) follows immediately.

        (3) Since $f(\mathfrak{A}_{\euls{C}}) =
f(\mathfrak{A})_{\euls{C}}$ for any ring homomorphism
$f:\mathfrak{A} \to \mathfrak{B}$ as above, it follows that $Y_{\euls{C}}$
equals the image of $\gr_{(\rad_{\vs})_{\euls{C}}}:
C_{\reg}/ \gr J_{\reg}\to \Iring_{\reg}$. It follows from Lemma~\ref{lem:genericregfiltredsurj} that $\gr_{(\rad_{\vs})_{\euls{C}}}$ is
surjective.  Thus, $Y_{\euls{C}} = \Iring_{\euls{C}}$, and so $Y$ and
$\Iring$ certainly have the same field of fractions. Combined with
Part~(2) this implies that they are equivalent orders.
\end{proof}

Recall that the moment map $\mu \colon V \times V^* \to \mf{g}^*$
is given by $\mu(x,v)(y) = \langle y \cdot x , v \rangle$.  Let
$I(\mu^{-1}(0))$ be the ideal in $\C[V] \o \Sym V$ of functions
vanishing on $\mu^{-1}(0)$. 

\begin{lemma}\label{lem:gradeofJ2} If $V$ is stable then $\gr J_{\reg} =
                I_{\reg}$.
\end{lemma}

\begin{proof}
        
        By Lemma~\ref{lem:genericregfiltredsurj}, $\rad_{\vs} \colon \dd(V_{\reg})^G \to \dd(\h_{\reg})^W$ is filtered surjective. Since $J_{\reg} = \ker(\rad_{\vs} )_{\reg}$, this implies that ${(\rad_{\vs})}_{\euls{C}} :
        \dd(V_{\reg})^G/J_{\reg}\to \dd(\h_{\reg})^W$ is a filtered
        isomorphism and so the map $\gr_{{(\rad_{\vs})}_\euls{C}}: C_{\reg}/\gr
        J_{\reg} \to \Iring_{\reg}$ is a graded isomorphism and $\gr
        J_{\reg}$ is a prime ideal of height $q= \Kdim C - 2 \dim \h$.
                
        Clearly $(\dd(V_{\reg})\g_{\chi})^G \subseteq J_{\reg}$.
        By the definition of $\g_{\chi}$, the spaces of symbols
        $\{\sigma_1(y) : y\in \g_{\chi}\}$ and $\{\sigma_1(x) : x\in
        \tau(\g)\}$ are equal. Thus the set of common zeros of the
        symbols $\{\sigma_1(x) : x \in \g_{\chi}\}$ is precisely
        $\mu^{-1}(0) \subseteq T^* V_{\reg}$. Therefore,
        $\euls{K}:=I(\mu^{-1}(0))^G_{\reg}\subseteq \gr J_{\reg}$.
        
The (reduced) subscheme of $(V \oplus V^*)\git G$ defined by the ideal $I$ is denoted $C_0 \git G$ in \cite{BLT}. Then \cite[Proposition~3.5]{BLT} says that $C_0 \git G$ is a closed subscheme of $\mu^{-1}(0)\git G$. Therefore, $I(\mu^{-1}(0))^G \subseteq I$ and,  in particular, $\euls{K} \subseteq I_{\reg}$.

        Using the fact that $V$ is stable, the proof of \cite[Lemma 4.2]{BLT} shows that
        $\mu^{-1}(0)|_{T^* V_{\reg}}$ is irreducible; thus the radical
        $\sqrt{\euls{K}}$ is a prime ideal. Moreover, as $C_{\reg} =
        (\C[V_{\reg}] \o \Sym V)^G $, \cite[Proposition 3.5]{BLT} implies
        that
        \[ \Kdim C_{\reg}/\euls{K} = 2 \dim \h =\dim \Iring = \Kdim C/
        I_{\reg} .
        \] As $\sqrt{\euls{K}}\subseteq I_{\reg}$, it follows
        that $\sqrt{\euls{K}}=I_{\reg}$, which is therefore a prime ideal
        of height~$q$.
        But $\euls{K}\subseteq
        \gr J_{\reg}$ which also has height $q$.  Thus $\gr
        J_{\reg} =\sqrt{\euls{K}} = I_{\reg}$.
\end{proof}

We are now ready to prove the main result of this section, which is also Part~(3) of Theorem~\ref{thm:radsurjects}.

\begin{theorem}\label{thm:assgrsurjradchi} Assume that $V$ is stable and the map
$\rr\otimes \rrp$ is surjective.  Then the maps
        \[\gr_{\rad_{\vs}}: (\C[V]\o \Sym V)^G \to (\C[\h]\o\Sym
\h)^W\] and
        \[ \rad_{\vs} \colon \dd(V)^G \to \Ak(W)
        \] are both surjective.
\end{theorem}

\begin{proof} It suffices to prove the surjectivity of
$\gr_{\rad_{\vs}}$.

By the hypothesis of the theorem, the map $C/I \to \Iring$
induced from $\rr\o\rrp$ is surjective. It is standard, see
\cite[Proposition~6.4.1]{CohMac}, that
$\Iring=(\C[\h]\otimes \Sym \h)^W$ is a normal domain. Thus $C/I$ is a
normal domain (equivalently, a maximal order). By Lemma~\ref{lem:gradeofJ2}, the hypotheses and hence the conclusions of Lemma~\ref{lem:images}   hold.

 By Lemma~\ref{lem:images}(2), $Y= \gr_{\rad_{\vs}}(C/I) \
 \supseteq \C[\h]^W\otimes (\Sym \h)^W$. Since $\Iring$ is a finite module over $Y$ one has
$\Kdim Y= \Kdim \Iring = \Kdim C/I$. Then, as $C/I$ is a domain,
$\Kdim Y = \Kdim C/I$ implies that the surjection
$\gr_{\rad_{\vs}} : C/I \to Y$ is an isomorphism and so
$Y\cong C/I$ is a maximal order.

On the other hand, $Y\subseteq \Iring$ and
Lemma~\ref{lem:images}(3) says that $Y$ is an equivalent order to
$\Iring$. As $Y$ is a maximal order this forces $Y=\Iring$, whence
$\gr_{\rad_{\vs}}$ is surjective.
\end{proof}

The significance of Theorem~\ref{thm:assgrsurjradchi} is that it
has the following important consequence.  The definition of symmetric spaces, together with their basic properties,  will be given  in 
Section~\ref{Sec:examples}.

\begin{theorem} \label{thm:symmLSsurjective} Suppose that $V$ is
a symmetric space for the connected reductive group $G$.  Then the radial parts map $
\rad_{\vs} : \dd(V)^G\to \Ak(W)$ is surjective.
\end{theorem}

\begin{proof} By Lemma~\ref{lem:symmetricstablepolar}, $V$ is
indeed a stable polar representation of $G$.  Moreover, by
\cite[Theorem]{Te}, the map $ \left(\C[ V]\otimes \Sym V\right)^G
\to \C[\h \times \h^*]^W$ is surjective.  Thus the result follows
from Theorem~\ref{thm:assgrsurjradchi}.
\end{proof}
 
It has recently been shown in \cite{BLT} that the $\C[\mu^{-1}(0)]^G \to \C[\h \times \h^*]^W$ is an isomorphism for a large 
class of polar representations and, as    a consequence, we get the following result, which proves Theorem~\ref{thm:radsurjects}(4).   The relevant definitions are as follows. 
The representation $V$ of $G$ is called \emph{locally free} if the stabiliser subgroup of a general point in $V$ is finite and  is said to be \textit{visible} \label{defn:visible}
  if the nilcone $\mc{N}(V)=\pi^{-1}(0)$ consists of 
finitely many $G$-orbits.   Finally,  let 
$       \rad_{\vs}' \colon \left( \dd(V) / \dd(V) \g_{\chi} \right)^G \to \Ak(W)$  denote  the map induced from $\rad_{\vs}$.

\begin{corollary} 
         \label{cor:vislocallyfreepolargr}
        {\rm (1)}       If the map $(\rr \otimes \rrp) \colon \C[\mu^{-1}(0)]^G \to \C[\h \times \h^*]^W$ is an isomorphism then 
                $\rad_{\vs}'$           is a filtered isomorphism.
                
         {\rm (2)} Suppose that $V$ is a visible, stable    polar representation of $G$ that is also locally free.
 Then   $\rad_{\vs}'$ is a filtered isomorphism   
\end{corollary}

\begin{proof} (1)    As noted in the proof of Lemma~\ref{lem:gradeofJ2},      $I(\mu^{-1}(0))$ is generated by
        the symbols $\{\sigma_i(x) : x\in \g_{\chi}\}$ and so $I(\mu^{-1}(0))\subseteq \gr_{\ord}J$. Combining    Lemma~\ref{lem:gradeofJ2} with the      fact that $\rr\otimes \rrp$ is an isomorphism      therefore shows that $I(\mu^{-1}(0)) = \gr_{\ord}J = I$.  Thus Theorem~\ref{thm:assgrsurjradchi}  implies that   $\gr_{\rad_\vs}$ is surjective with kernel $\gr_{\ord}J$. It also follows that  $J=(\dd(V)\g_\chi)^G$ and hence  $\rad_{\vs}'$ is a filtered isomorphism.

        (2)    This follows from Part~(1) combined with \cite[Theorem~1.2]{BLT}.
\end{proof}

The locally free hypothesis in Corollary~\ref{cor:vislocallyfreepolargr}(2) is obviously very strong---for example, it is    not even 
satisfied in the basic case when    $V=\g$ for a simple algebraic group $G$. A number of examples satisfying the hypotheses of the corollary  are given in \cite{BLT}, with one of    their motivating examples being 
the representation $V=S^3\C^3$ over $G=SL(3)$; see \cite[Example~5.6]{BLT} and also 
\cite[Example~15.3]{BNS}.
 
    
\section{Reductive symmetric pairs}\label{Sec:examples}

  In the next two sections we discuss one of the main examples of polar representations, that of symmetric spaces. In this section we concentrate on the image of $\rad_{\vs}$ and  we determine, at least for $\vs=0$, exactly when $\Ak=\Im(\rad_{\vs})$ is a simple ring. This proves the first half of Theorem~\ref{intro-mainthm} from the introduction. The proof of that result will be completed in Section~\ref{Sec:kernel}.

 \begin{definition}\label{defn:symmetric}
We begin with the basic definitions. 
Let $\Gtilde$ be a connected, complex reductive algebraic group
with Lie algebra $\gtilde$.  Fix a non-degenerate,
$\Gtilde$-invariant symmetric bilinear form $\varkappa$ on
$\gtilde$ that reduces to the Killing form on
$[\gtilde,\gtilde]$.  Let $\vt$ be an involutive automorphism
of $\gtilde$ preserving $\varkappa$ and set $\g = \ker(\vt -I)$, $\p
= \ker(\vt +I)$. Then, $\gtilde= \g \oplus \p$ and the pair
$(\gtilde,\vt) $ (or $(\gtilde,\g)$) is called a {\it symmetric
  pair} with \emph{symmetric space} $V := \p$.
Let $G$ be the
connected reductive subgroup of $\Gtilde$ such that
$\g=\Lie(G)$. Both  $G$ and its Lie algebra $\g$ act on $\p$ via
the adjoint action: $x\cdot v= \adj(x)(v)= [x,v]$ for all
$x \in \g$ and $v \in \p$. 

Finally, let
$\h \subset \p$ be a \emph{Cartan subspace}; thus, $\h$ is
maximal among the subspaces of $\p$ consisting of semisimple
elements in $\gtilde$ which pairwise commute,
see~\cite[Corollary~37.5.4]{TY}.  The \emph{Weyl group} is then defined to be $W= N_G(\h)/Z_G(\h)$. 
\end{definition}
 
Any symmetric pair $(\gtilde,\vartheta)$  is a direct product $(\gtilde,\vartheta)\cong (\gtilde_1,\vartheta_1)\times\cdots \times (\gtilde_m,\vartheta_m)$
of  
irreducible  symmetric pairs, which then decomposes $(G,\p)$ into a corresponding direct product of sub-symmetric spaces, although when $(\gtilde,\vt)$ is irreducible, $\g$  may be reducible. For a classification of irreducible
symmetric spaces, see the tables in \cite[Chapter~X]{He1} or, in
a form more convenient for this paper, the  tables in
Appendix~\ref{app-tables}.
 Symmetric spaces $(G,\p)$  are a natural
generalisation of the adjoint representation of $G$ on $\g$.
Specifically, in   the \emph{diagonal case} where $\Gtilde = G
\times G$ with $\vt(x,y) = (y,x)$, one has  $(\gtilde,\g) = (\g
\oplus \g, \g)$ with the natural adjoint action of $G$ on $\p =
\g$.

As we next show,  symmetric spaces are examples of  polar representations.
 We recall from the discussion at the beginning of Section~\ref{Sec:polarreps} that $\p_{\st}$ denotes
  the set of $\p$-regular elements of $\p$.  
   
\begin{lemma}\label{lem:symmetricstablepolar} The  symmetric  spaces 
$(G, \p)$ are  stable polar representations. Moreover,
in the notation of Section~\ref{Sec:polarreps}, one has  $\p_{\reg} =
\p_{\st}$.  
\end{lemma}

\begin{proof} The representations $(G,\p)$ are particular cases
of polar representations, see \cite{DadokKac} and
\cite[(8.6)]{PopovVinberg}. 
 In  our notation, \cite[Remark 6, p.771]{KR}  says that
\[
\{ v \in \p : \dim G\cdot v \ \text{maximal}\}  \cap \p_{\mathrm{ss}} =
\p_{\reg},
\]
and hence that $\p_{\reg} = 
\p_{\st}$. As $\p_{\reg}\not=\emptyset$, this  also implies that $\p$ is stable.
\end{proof}
 
 Our first aim is to understand exactly when the image $\Ak$ of the radial parts map is simple. Before giving the details we notice a couple of simplifications to the problem.

 \begin{remark}\label{rem:reducible-symm}
In constructing and understanding radial parts maps, one can always
  restrict to the case of irreducible symmetric pairs.  Since this is   true  more generally, in order to justify this assertion   let $(G_i, V_i) $ be polar representations   with 
  associated data $\{h_i, W_i, \vs_i, \kappa_i\}$, where $\kappa_i$ is the parameter introduced
in~\eqref{kappa-defn} for $\h_i$.  Then $V=V_1\oplus V_2$ is   a polar representation for the group $G=G_1\times G_2$,
 with associated data $\h=\h_1\oplus \h_2$ and so forth. Note in particular  that,  from scalars $\vs=(\vs_1, \vs_2)\in \C^{k_1+k_2}$, the parameter $\kappa$ defined in \eqref{kappa-defn} for $(G,V)$ is just 
 $\kappa_1 + \kappa_2$.  It follows that $\Ak = A_{\kappa_1}\otimes_\C A_{\kappa_2}$ and so, by 
Theorem~\ref{thm:symmLSsurjective}, the map $\rad_{\vs}:\dd(V)^G \to \Ak$ is  the map
\[\rad_{\vs_1}\otimes \rad_{\vs_2} : \dd(V_1)^{G_1}\otimes \dd(V_2)^{G_2}  \ \longrightarrow \ 
A_{\kappa_1}\otimes_\C A_{\kappa_2}.\] 
\end{remark} 

From  Corollary~\ref{cor:quotientetaleh} we then deduce the following.

  \begin{lemma}\label{reducible-spaces}  Suppose that $(G,V)= (G_1,V_1) \oplus (G_2,V_2)$ is a
direct sum of polar representations and keep the above notation. 
  Then $\Ak$ is simple if and only if  both  $A_{\kappa_1}$ and $A_{\kappa_2}$ are simple. \qed
  \end{lemma}

 \begin{remark}\label{rem:rtrivial-symm}
 Let $(\gtilde,\vt)$ be a symmetric pair with symmetric space $\p$ and Cartan subspace $\h\subseteq \p$. Recall from \eqref{defn:polarrank} that   $\rank \p=\dim_\C \h-\dim \h^\g$.  If   $\rank\p=0$, we define $(\gtilde,\vt)$   to be a \emph{trivial symmetric pair}, with \emph{trivial symmetric space} 
 $(G,\p)$.   Write $\gtilde = \gtilde'\oplus \mf{z}$, where $\gtilde'=[\gtilde,\gtilde]$ and $\mf{z}$ is the centre of $\gtilde$.  The reason for  calling the rank zero case trivial is that then $\h=\h^\g\subseteq \mf{z}$, from which  it follows  that $\p=\h\subseteq \mf{z}$. In particular,  $G$ acts trivially on $\p$ and so  $\dd(\p)=\dd(\p)^G=\dd(\h)^W$. Hence any radial parts map for a trivial symmetric space  is the identity.
 
 When $\vt=1$, $\g=\gtilde$ and so $\p=0$, which is certainly trivial. We
  do not insist that $\vt\not=1$, since the case $\vt=1$  can appear among sub-symmetric pairs of symmetric pairs with $\vt \neq 1$. 
 \end{remark}

 Let $(\gtilde,\vartheta)$ be a symmetric pair and recall that our first goal is to understand when $\Ak=\Im(\rad_{\vs})$ is simple. When  $\g$ is  semisimple,   it suffices to study the question when
  $\vs=0$.    
  
    \begin{proposition}   \label{Harish-Chandra-map}
Assume that $(\gtilde,\vartheta)$ is a symmetric pair such that the
centre of $\g$ acts trivially
 on $\p$. Then, $\Im(\rad_\vs) \cong
\Im (\rad_0)$ and $\ker(\rad_{\vs})=\ker(\rad_0)$ for all choices of
$\vs$. 
\end{proposition}

\begin{proof}
 Write $G=G_1Z$ where $G_1$ is semisimple and $Z$ is the
 connected component of the centre of $G$. By hypothesis one has
$\C[\p]^G = \C[\p]^{G_1}$ and $\dd(\p)^G = \dd(\p)^{G_1}$. Since
$G_1$  is semisimple, the result follows from
Corollary~\ref{rem:invconjugation}. 
\end{proof}

If the centre of $\g$ does not act trivially on $\p$ then understanding for which $\vs$ the algebra $\Im(\rad_{\vs})$ is simple is a more complicated problem  and we do not have a general answer. Indeed,  even  when
  $(\gtilde,\g)=(\mf{sl}(2), \mf{so}(2))$, all infinite dimensional primitive factors of $U(\mf{sl}(2))$ appear as $\Im(\rad_{\vs})$ for some choice of $\vs$; see \cite[Corollary~14.15]{BNS}. 
    For simplicity, we will therefore stick to the case $\varsigma=0$, which is also 
 the case of most interest in the literature  (in the diagonal case this will suffice to solve the problem in general, but  there are additional subtleties,
  for which the  reader is referred to the discussion after Corollary~\ref{Harish-Chandra-map21}).
  
  Thus, by
 definition, we are interested in the image of 
\begin{equation} \label{radzero}
\rad_0(D)(z) = \rr\left(D(\rr^{-1}(z)) \right)\quad \text{for all $D \in
  \dd(\p)^G$ and $z \in \C[\h]^W$.}
\end{equation}
   We will write
$\rad=\rad_0$ and $\kappa=\kappa(0)$ throughout the remainder of this section.

In order to describe   $\kappa$  we 
need   some standard notation. Given a  
symmetric pair $(\gtilde,\vartheta)$, let  $\Sigma =\Sigma(\gtilde,\h) \subset \h^*$
denote  the (restricted) root system
defined by $\h$; see \cite[38.2.3]{TY}. (If $(\gtilde,\vt)$ is trivial then, by definition, $R=\emptyset=\Sigma.$) 
Let 
$R= R(\p,\h)= \Sigma_{\mathrm{red}} = \{\alpha \in \Sigma : \alpha/2 \notin
\Sigma\}$ be the \emph{reduced root system} associated to $\Sigma$. 
Choose a set of positive roots $\Sigma^+ \subset \Sigma$ and set
$R^+= \Sigma^+ \cap \Sigma_{\mathrm{red}}$. Then 
$
\calA= \{H_\alpha = \ker \alpha: \alpha \in R^+ \}
$
and $W$  is generated by the real reflections $s_\alpha= s_{H_\alpha}$; in
particular $\ell_{H_\alpha} =2$ for all $\alpha$.

For $\alpha \in \Sigma$ set
\[  \mathfrak{q}^{\alpha} = \{v \in  \mathfrak{q} : \adj(a)^2(v) = \alpha(a)^2 v \
 \text{for all  $a \in \h$}\},\]
  when $\mathfrak{q}=\p$ or $\mathfrak{q}=\g$, but 
\[  \gtilde^{\alpha} = \{v \in \gtilde : \adj(a)(v) = \alpha(a) v \
  \text{for all  $a \in \h$}\}.
\] 
Then,  $\gtilde^{-\alpha} = \vt( \gtilde^{\alpha})$ and
$\gtilde^{\alpha} \oplus \gtilde^{-\alpha} = \g^\alpha \oplus
\p^\alpha$ for all $\alpha \in \Sigma^+$. The adjoint action of
$\h$ yields the decompositions  (orthogonal with respect to $\varkappa$):
\[
  \g = Z_\g(\h) \oplus \bigoplus_{\mu \in R}
  (\g^{\mu} + \g^{2\mu}) \quad\text{and}\quad  \p = \h \oplus \bigoplus_{\mu \in R}
   (\p^{\mu} + \p^{2\mu}) .
\]
Note that  $\dim
\p^\alpha = \dim \g^\alpha= \dim \gtilde^\alpha$ and we define
the \emph{multiplicity of $\alpha \in \Sigma$} by:
\begin{equation}
  \label{eq:root-multiplicity}
  m_\alpha= \dim \gtilde^\alpha + \dim \gtilde^{2\alpha}
\end{equation}
with the convention that $\gtilde^{2\alpha} = \{0\}$ if $2\alpha
\notin \Sigma$. Set 
\begin{equation}
  \label{eq:root-multiplicity2}
  k_\alpha= \frac{1}2m_\alpha= \frac{1}{2}\bigl(\dim \gtilde^\alpha + \dim \gtilde^{2\alpha}\bigr)
\end{equation}

\begin{proposition}
  \label{kappa-symmetric}
  For a  
  symmetric pair $(\gtilde,\vartheta)$, define   \begin{equation}
    \label{eq:kappa-alpha}
    \kappa =
  \bigl(\kappa_{H_\alpha,i} : \alpha \in
    \Sigma^+_{\mathrm{red}}, \, i=0,1 \bigr),\quad\text{where}\ \ 
 \kappa_{H_\alpha, 0} = 0 \ \text{and}\   \kappa_{H_\alpha, 1} =
    k_\alpha . 
  \end{equation}
  Then  $\Im(\rad) = A_{\kappa}(W)$. 
\end{proposition}
\begin{remark} 
  If $(\gtilde,\vt)$ is trivial then
$\Sigma=\emptyset$ and we set $\kappa = 0$. Of course, by Remark~\ref{rem:rtrivial-symm}, $\Im(\rad) = A_{\kappa}(W)$ also holds in this case.
\end{remark}

\begin{proof}  This amounts to translating Theorem~\ref{thm:surjectivereducerank1} into the present notation. 
Since $\p_{\reg} = \p_{\st}$ by Lemma~\ref{lem:symmetricstablepolar}, it follows from   Lemma~\ref{lem:slices}  that $\rank(G_p,S_p)\geq 1$
for $p \in \h \smallsetminus \h_{\reg}$. In particular, if $H\in \calA$, then  $\rank(G_p,S_p)= 1$
for all $p \in H^{\circ} = H \cap \h^{\circ}$ by Remark~\ref{rem:rankzeropolar}.

 Fix $\alpha \in R^+$   and pick $p \in
  H_\alpha^{\circ}$. Thus, $\alpha(p) = 0$ and $\mu(b) \neq 0$ for
  $\mu \in \Sigma \smallsetminus \{\alpha\}$. One can easily show
  that
\[
  \g_p = Z_\g(H_\alpha) = Z_{\g}(\h) \oplus \g^\alpha \oplus
  \g^{2\alpha} \ \quad\text{and}\quad \ \p_p = Z_\p(b) = \h
  \oplus \p^\alpha \oplus \p^{2\alpha},\] while
$\g \cdot p = \bigoplus_{\mu \in \Sigma \smallsetminus
  \{\alpha\}} \p^\mu.$ Using Remark~\ref{rem:parameters}, it
follows that the slice $(G_p,S_p)$ is given by
$G_p= Z_G(H_\alpha)$ and
$S_p = \p_p = \h \oplus \p^\alpha \oplus \p^{2\alpha}$ with Weyl
group $W_p = \{1, s_\alpha\}$. Let $\alpha^\vee$ be such that
$s_\alpha(\alpha^\vee)= -\alpha^\vee$ and write
$S_p = S_\alpha \oplus H_\alpha$ where
$S_\alpha= \C \alpha^\vee \oplus \p^\alpha \oplus \p^{2\alpha}$.
Then  
\[
  \C[S_p]^{G_p} = \C[H_\alpha] \otimes \C[S_\alpha]^{G_p} \cong
  \C[\h]^{W_\alpha} = \C[H_\alpha] \otimes \C[\C \alpha^\vee]^{W_p}=
   \C[H_\alpha] \otimes \C[\alpha^2].
\]
The representation $(G_p,S_\alpha)$ is polar of rank one with
\[\dim S_\alpha = \dim (\p^\alpha \oplus \p^{2\alpha}) + 1=
m_\alpha +1.\]  The restriction of the form $\varkappa$ to
$S_\alpha$ is $G_p$-invariant and non-degenerate and so
$\C[S_\alpha]^{G_p} = \C[u] \cong \C[\C \alpha^\vee]^{W_p} =
\C[\alpha^2]$, where $u(x) = \varkappa(x,x)$ for
$x \in S_\alpha$. Therefore the differential operator
$\Delta = u_*(D) \in (\Sym S_\alpha)^{G_p}$ used in
Proposition~\ref{thm:PHVbfunction} is the Laplacian operator on
the space $S_\alpha$; that is,
$\Delta= \sum_{i=1}^{m_\alpha +1} \frac{\!\partial^2}{\partial
  x_i^2}$ if $u =\sum_{i=1}^{m_\alpha +1} x_i^2$ in an orthonormal 
coordinate system. The $b$-function attached to $\Delta$ is easy
to compute:
\[b(s)= (s+1) (s+ \half \dim S_\alpha) = (s+1) (s + \half
(m_\alpha +1)).\] As in Notation~\ref{notation:chidot}, we set
$\lambda_0= 0$ and
$\lambda_1= \half (m_\alpha +1) - 1 = \half(m_\alpha -1)$. Then,
by Theorem~\ref{thm:surjectivereducerank1},  the
parameter $\kappa$ is given by
\[
\kappa_{H_\alpha,0} = 0, \quad \kappa_{H_\alpha,1} = \lambda_1 -
\half +1 = \half m_\alpha = k_\alpha,
\]
as claimed.  Finally, $\Im(\rad)=\Ak(W)$ by Theorem~\ref{thm:symmLSsurjective}.
\end{proof}

 \begin{remark}\label{kappa-elsewhere}
   The    parameter $\kappa$ defined in~\eqref{eq:kappa-alpha} appears
        frequently in the theory of symmetric spaces. For instance, it
        is  known that for this parameter the radial components of the
        invariant constant differential operators  $(\Sym \p)^G
        \subset \dd(\p)^G$ can be computed via the Dunkl operators
        defined by $\kappa$. To give more details it will be convenient to 
        identify  the rational Cherednik algebra $H_\kappa(W)$   with the
        subalgebra of $\dd(\h_{\reg})\rtimes W$ generated by
        $\C[\h]$, $W$ and the Dunkl operators $T_\kappa(y)= T_y$,
        $y \in \h$, as defined in \eqref{eq:Dunkl} for this
        choice of $\kappa$.
        The map $T_\kappa$ extends to a morphism on $\Sym \h$, and composed with
        $\rr^* : (\Sym \p)^G \cong (\Sym \h)^W$ gives 
        \begin{equation}\label{Dunkl-extension}
                \rad(D) = T_\kappa(\rr^*(D)) \in \Ak(W) \quad \text{for all $D \in (\Sym \p)^G \subset
                        \dd(\p)^G$;} 
        \end{equation}
        see, for example, \cite[(6.2)]{dJ}. If $D \in \dd(\p)^G$ then
        $\rad(D) \in eH_k(W)e \subset \dd(\h_{\reg})^W$ can
        be expressed as a polynomial in elements of $\p^*$ and Dunkl
        operators $T_\kappa(y)$ for  $y \in \h$. The   formula \eqref{Dunkl-extension} gives
        this expression when $D \in (\Sym \p)^G \subset
        \dd(\p)^G$. Therefore Proposition~\ref{kappa-symmetric} can be
        viewed as an extension of that formula.
 \end{remark}

 Let $(G,\p)$ be a  symmetric space.
Set
\begin{equation}
  \label{eq:kernel1}
  \KK(\p)= \bigl\{d\in \dd(\p) \, : \,  \forall \, f\in
  \C[\p]^G, \;  d(f)=0  \bigr\},  \quad  \LL(\p) = \KK(\p) / \dd(\p) \tau(\g)
\end{equation}
and note that $\KK(\p)\cap \eD(\p)^G  = \ker(\rad_0)$.

\subsection*{Special classes  of symmetric spaces}
Many important  properties of the diagonal case $(\g \oplus \g, \g)$
generalise to Sekiguchi's nice symmetric spaces, defined as
follows.

\begin{definition} \label{nice-space} If the $k_\alpha$ are defined by \eqref{eq:root-multiplicity2},
then the symmetric pair
  $(\gtilde,\g)$, or the symmetric space $(\g,\p)$, is {\it nice}~if
  \begin{equation}
    \label{eq:dagger} k_\alpha \le 1, \ \text{for all }
   \alpha\in R. 
\end{equation}
  \end{definition}
  
The diagonal case is obviously nice. The reader is
referred to \cite[Section~6]{Se} for further details and for a
classification of nice pairs to \cite[Lemma~6.2]{Se}. For a
version of these results using notation closer to that of this
paper, see \cite{LS3} and especially \cite[Theorem~2.5]{LS3}.
 
The main result from \cite{LS3} now gives:

\begin{theorem}\label{LS3-theorem}
  Suppose that $(\gtilde,\vt)$
is a nice symmetric pair with symmetric space $\p$ and define $\kappa$ by 
\eqref{eq:kappa-alpha}.  Then
$\KK(\p)=\eD(\p)\tau(\g)$ and so
$\ker(\rad_0) =\bigl(\eD(\p)\tau(\g)\bigr)^G$. Moreover, the ring $
\eD(\p)^G/\ker(\rad_0) \cong \Ak(W)$ is simple.
  \end{theorem}

  \begin{proof} 
    This forms parts of~\cite[Theorems~A and~D]{LS3}, combined with Proposition~\ref{kappa-symmetric}.
\end{proof}

The second class of simple spherical algebras arises
from \cite{BEG}. 
Let  $\mc{H}_q(W)$ be  the Hecke algebra associated to a general
parameter $\kappa$ as in \cite[p.~284]{BEG}. We also need the following definition.

\begin{definition}\label{gainly-defn} Define the symmetric pair  $(\gtilde,\vt)$ and its corresponding symmetric
space  $(G,\p)$ to be \emph{\integral} if  the parameter
 $\kappa$ from \eqref{eq:kappa-alpha} takes only integral values; equivalently, if $k_\alpha\in \mathbb{Z}$ for all
 $\alpha\in R$.   
 Define $(\gtilde,\vt)$  and $(G,\p)$  to be \emph{\gainly} if $(\gtilde,\vt)$  is a direct product of nice symmetric pairs and  integral symmetric pairs. 
 
 A  trivial symmetric space is both nice and \integral\ since in this case $R=
  \emptyset$, and  so the relevant conditions are   vacuously satisfied.
  
 \end{definition}

 The significance of integrality comes from the following  result. 
  
 \begin{theorem}\cite[Theorem~3.1]{BEG} \label{thm:BEG}
\begin{enumerate}
 \item    If $(\gtilde,\vt)$  is \integral\ then $\mc{H}_q(W)$ is semisimple.
 \item[(2)]   If the Hecke algebra $\mc{H}_q(W)$ is semisimple, then the algebra $\Ak(W)$ is
  simple. Indeed, even the Cherednik algebra $\Hk(W)$ is simple. 
  \end{enumerate}
 \end{theorem}

 \begin{proof}  We briefly explain the notation from \cite{BEG}. First, the parameters $c_\alpha$ from \cite{BEG} are just the negative of our $k_\alpha$ (compare \eqref{eq:kappa-alpha} with Remark~\ref{comparison}).  The hypothesis 
 $c\in \C[R]_{\mathrm{reg}}^W$  in \cite[Theorem~3.1]{BEG} simply means that the corresponding Hecke 
 algebra $\mc{H}_q(W)$ is semisimple. As observed immediately before \cite[Corollary~2.3]{BEG}, the hypothesis 
 $c\in \mathbb{Z}[R]^W$ (meaning that each $c_{\alpha}\in \mathbb{Z}$)  also implies that $c\in \C[R]_{\mathrm{reg}}^W$, which gives Part~(1).  Part~(2) of the theorem is then \cite[Theorem~3.1]{BEG}. \end{proof}

 \begin{remark}
   \label{Heckeremark}
   Retain the notation from the proof of Theorem~\ref{thm:BEG}. If $(\gtilde,\vt)$ is not integral, then 
   one of the parameters $ \kappa_\alpha = - c_\alpha$ belongs to $\half + \Z$. It then follows
   from \cite[\S3.8 and Theorem~3.9]{Gyojahecke} that the  associated Hecke
   algebra $\mc{H}_q(W)$ is not semisimple. Thus the converse of Theorem~\ref{thm:BEG}(1) also holds. 
  \end{remark}

\begin{corollary}
  \label{thm42}
  Assume that the symmetric  pair $(\gtilde,\vt)$ is \gainly. Then the algebra $\Ak=\Im(\rad)$ is simple.
\end{corollary}

\begin{proof}  By Proposition~\ref{kappa-symmetric},    $\Im(\rad)= A_\kappa(W)$.  
Now combine Lemma~\ref{reducible-spaces} with Theorems~\ref{LS3-theorem}
and~\ref{thm:BEG}.
\end{proof}

  In fact Corollary~\ref{thm42}  exhausts  the symmetric pairs for which $\Ak$ is simple, since we will prove the following stronger result.
 
 \begin{theorem}
   \label{simplicity-irred-symm-space}
 Let $(\gtilde,\vt)$ be a  symmetric pair.  Then the spherical algebra $\Ak$ is simple if and only if 
 $(\gtilde,\vt)$ is \gainly.
 \end{theorem}

\begin{proof}      The proof of consists of a case-by-case analysis to show
  that,   for the symmetric spaces not covered by Corollary~\ref{thm42},
  the algebra $\Ak$ is not simple. This is deferred to
  Appendix~\ref{appA}.  
  \end{proof}


\section{The kernel of the radial parts map  for symmetric spaces}\label{Sec:kernel}
  As has been remarked in  the introduction, given a polar representation $(G,V)$   then $\Ker(\rad_{0})$ always contains $(\dd(V)\tau(\g))^G$, but checking equality is an important but  much more subtle problem. 
  The   main aim of this section is 
  to examine this problem  for symmetric spaces and thereby to  complete the proof  
  Theorem~\ref{intro-mainthm}. In Corollaries~\ref{CORB} and ~\ref{maincor2} we give applications of this result to equivariant eigendistributions. The reader is referred to the discussion after Corollary~\ref{Harish-Chandra-map21} for  applications to  the  diagonal case.

   Since $\Ker(\rad)$ is described by Theorem~\ref{LS3-theorem} for nice symmetric spaces, it remains to consider 
    the integral symmetric spaces, as we do below.   Suppose that $\gtilde$ is semisimple. Then $(\gtilde,\g)$ is
isomorphic to a product of trivial pairs and irreducible
symmetric spaces $(\gtilde_i,\g_i)$ where either
$(\gtilde_i,\g_i) \cong (\fs \times \fs, \fs)$, with $\fs$ 
simple, or $\gtilde_i$ is simple.  From the tables  in Appendix~\ref{app-tables},
the  integral  irreducible
symmetric spaces $(\gtilde,\g)$  are the following:
\begin{enumerate}
\item $\mathsf{A II}_n = (\fsl(2n),\fsp(n))$, $ n\ge 2$, with  $k=2$; 
  \item $\mathsf{D II}_p = (\fso(2p),\fso(2p-1))$, $ p\ge 3$, with
    $k=p-1$;
    \item $\mathsf{E IV} = (\mathfrak{e}(6),\, \mathfrak{f}(4))$,
      with $k=4$;
      \item $(\fs \times \fs, \fs)$ with $\fs$ simple and $k = 1$.
\end{enumerate}

Equivalently:

\begin{lemma}\label{integra-equ}  A symmetric pair 
$(\gtilde,\vt)$ is \integral\ if and only if it is isomorphic 
to a product of pairs of type (1)--(4) and trivial pairs. \qed
\end{lemma}

In our finer analysis of \integral\ symmetric pairs it will be convenient
to exclude pairs $(\fs \times \fs, \fs)$ with $\fs$ simple; partly because,
by Remarks~\ref{distingexamples}(2), a useful technical condition~\eqref{equ:ast} usually fails
for such a pair.  This will not affect the results since the  pairs
$(\fs \times \fs, \fs)$ are also nice symmetric spaces. We therefore make
the following definition.

\begin{definition}
  \label{almostnice}
  The symmetric pair $(\gtilde,\vt)$ satisfies $(\dagger)$ if it is an
  \integral\ symmetric pair with no summands of the form
  $(\fs \times \fs, \fs)$ for $\fs$ simple. As usual, the
  corresponding symmetric space $(G,\p)$ satisfies $(\dagger)$ if
  $(\gtilde,\vt)$ satisfies $(\dagger)$.
  \end{definition}

For standard definitions   concerning
$\dd(\p)$-modules, in particular for the support $\Supp M$ and characteristic  
variety $\Ch M$  of  a $\dd$-module $M$, we refer the reader to
\cite[Section~3]{LS3}.
We will frequently identify $\p$ with its dual $\p^*$ through the non-degenerate 
$G$-invariant bilinear form $\varkappa$.  As in \cite[Section~2]{LS3},  the  \emph{commuting
variety of $\p$} is the closed
subvariety of $T^*\p = \p \times \p^* \cong \p \times \p$ defined by
$\euls{C}(\p) = \{(x,y) \in \p \times \p : [x,y] = 0\}$.

If $b \in \p$,  then the \emph{centraliser} of $b$ in a subset $X$ of
$\gtilde$  is denoted by $X_b$.
Recall that $x \in \p$ is called \emph{nilpotent} if $x \in
[\gtilde,\gtilde]$ and $\adj_{\gtilde}(x)$ is nilpotent; this is equivalent to
 $f(x) = 0$ for all $f \in \C[\p]_+^G$. We denote by $\NV$ the set of nilpotent elements.  

 By Lemma~\ref{lem:symmetricstablepolar}, any symmetric space
 $\p$ is stable. Thus, by the discussion before
 Corollary~\ref{cor:unstablepolaropen}, $\p_{\st}=\p_{\mathrm{st}}$ and
 hence, by~\eqref{eq:Ucomplement}
 and in the notation of \eqref{eq:msdefnorbit},
  $$s=m = \max_{v \in \p} \dim G \cdot v = \dim \p - \dim \h.
  $$
  We set $\p_m = \{v \in \p : \dim G\cdot v =m\}$.
   It is useful to note that, in the notation of Section~\ref{Sec:polarreps}, 
    $\p_{\reg} = \p_m \cap \p_{\mathrm{ss}}$  (written in our
    notation, this is the assertion  of \cite[Remark~6]{KR}).

 We say that
 $x$ is \emph{regular nilpotent} if $x \in \p_m\cap \NV$ and  write $
 \NVreg = \{x \in \NV: \text{$x$ is regular nilpotent}\}.  $ 
 We say that $x\in \NV$ is \emph{distinguished}, if  the centraliser $\p_x$ does not contain non-central
semisimple elements and put 
 $ \Vdist = \{x \in \NV : \text{$x$ is distinguished}\}.$
By  \cite[Lemma~1.3]{Pa} one has $\NVreg \subset \Vdist$.

\begin{remarks} \label{distingexamples}
  (0) If $(\gtilde,\vt)$ is trivial, then $\p= \h$ and so $\NV=\{0\}.$ 
  
  (1) If $(\gtilde,\g) = (\fso(n),\fso(n-1))$ for $n \ge 3$, then
  $(\gtilde,\g)$ is of rank one and
  $\NVreg = \Vdist = \NV \sminus \{0\}$ since $\p_x= \C x$ for
  all $x \in \p \sminus \{0\} = \C^{n-1} \sminus \{0\}$; see, for
  example, \cite[Proposition~4]{SY}.

  (2) If  $(\gtilde,\g)= (\fs \times \fs, \fs)$ with $\fs$
  simple, then $\p \cong \fs$ is the adjoint representation of $\fs$ and it
  is known that $\mathbf{N}(\p)^{\reg}= \Vdist$ only in the case where $\fs =
  \fsl(n)$, see~\cite{BaCa}.
\end{remarks}

Suppose that $b \in \p$ is semisimple. Then, as in \cite[I.6]{KR},  the
decomposition $\gtilde_b = \g_b \boplus \p_b$ defines a \emph{sub-symmetric
pair} $(\gtilde_b,\g_b) $  inside
$(\gtilde,\g)$ with \emph{sub-symmetric space} $(\g_b,\p_b)$. 

\begin{lemma}
  \label{distinguishedinduction}
 If the symmetric pair $(\gtilde,\g)$ satisfies $(\dagger)$ then
 the following condition holds:
 \begin{equation}\label{equ:ast}
    \text{\ $\mathbf{N}(\p_b)^{\mathrm{dist}}
   = \mathbf{N}(\p_b)^{\mathrm{reg}}$ for each  sub-symmetric pair
   $(\gtilde_b,\g_b)$ of $(\gtilde,\g)$.  }
 \end{equation}
 Furthermore, every sub-symmetric pair of $(\gtilde,\g)$ satisfies $(\dagger)$.
 \end{lemma}

\begin{proof}
  It suffices to prove the claims when $(\gtilde,\g)$ is
  irreducible, with the trivial case being obvious. For type $\mathsf{D II}_p$ it follows from
  Remarks~\ref{distingexamples}(1), since a sub-symmetric pair
      is either trivial (when $b \ne 0$) or equal to  $(\gtilde,\g)$ (if
  $b=0$). The assertion is proved in
  \cite[Theorem~3.2]{Pa} for type  $\mathsf{A
    II}_n$ and \cite[Proof of Proposition~3.4]{Pa} for type $\mathsf{E IV}$.
\end{proof}

Set $Z(\p) = \NV \sminus \Vdist$. Observe that,  for a symmetric space $(G,\p)$ satisfying $(\dagger)$,
 we have $Z(\p) = \NV \sminus \NVreg$.

\begin{proposition}
  \label{vanishing}
 Let $(\gtilde,\g)$ be a symmetric pair. Suppose that  $M$ is a
 holonomic $G$-equivariant $\dd(\p)$-module such that
\[
\Ch M \subseteq \euls{C}(\p)
 \cap (Z(\p) \times \NV),
\]
 where $\Ch M$ is the 
 characteristic variety  of $M$. Then $M=\{0\}$.  
\end{proposition}

\begin{proof}
  We follow the proof of \cite[Theorem~3.8]{LS3}. We may assume that
   $0\not=M$ is simple.  Then, as in  {loc.~cit.,} one shows that the support of $M$
  is the closure of a single nilpotent orbit $G\cdot
  x$.  
  By hypothesis $G\cdot x \subseteq Z(\p)$, thus $x$ is not
  distinguished. Pick an irreducible component $X$ of $\Ch M$
  containing a point of the form $(x,\xi)$. Then, it follows from
  \cite[Lemma~2.2(ii)]{LS3} that $\dim X < \dim \p$. But since
  $M$ is   holonomic, we must have $\dim X=\dim \p$,
  giving a contradiction.
\end{proof}

Recall from \eqref{eq:kernel1} the $\dd(\mf{p})$-modules $\KK(\p)$ and $\LL(\p)$. We will need the following general lemma.   
  
\begin{lemma}
  \label{LSresults}
  Let $(G,\p)$ be a symmetric space.
  
   {\rm (1)} Let $d \in \dd(\p)$ and suppose that  
  $dp \in \dd(\p )\tau(\g)$ for some $p \in (\Sym \p)^G$. Then 
  that $p^n d \in \dd(\p )\tau(\g)$ for some  $n >0$. In particular, if $F \subset
  (\Sym \p)^G$ is an ideal such that $dF \subset  \dd(\p )\tau(\g)$,
  then $F^t d \subset  \dd(\p )\tau(\g)$ for some $t >0$.

  {\rm (2)}   The support of $\LL(\p)= \KK(\p)/\dd(\p)\tau(\g)$
  in $\p$ is contained in $\p \sminus \p_m$ and hence in $\p\sminus \p_{\reg}$.
   In particular, for any $\theta\in \LL(\p)$ there exists $n> 0$ such that $\delta^n\theta=0$.
\end{lemma}

 \begin{proof}
  (1) The   proof of \cite[Lemma~ 5.4]{LS} proves the existence of  the
  integer~$n$. Let $F=\sum_{j=1}^r (\Sym \p)^G p_j$   and pick $n \in \N^*$ such that
  $p_j^n d \in \dd(\p )\tau(\g)$ for all $j$. Then  $F^{rm} d \in \dd(\p )\tau(\g)$.

  (2)  The idea behind the proof is standard (see, for example, \cite[Lemma~6.7]{AJM} or the proof of  \cite[Lemma~4.5]{LS3}) so some details will be left to the reader. 
  Since $\p_m \supseteq  \p_{\reg}$ by   \cite[Remark~6]{KR},   it suffices to prove the first assertion of Part~(2).    
  
 Let $v\in \p_m$.   If $M$ is a $\C[\p]$-module, denote by $M_v$ the localisation
  of $M$ at the maximal ideal $\mathbf{m}_v$ defined by~$v$. Pick algebraically independent homogeneous elements $p_j$ 
  such that $\C[\p]^G= \C[p_1,\dots,p_\ell]$. By
  \cite[Theorem~13]{KR} the morphism $\varpi: \p \to \C^\ell$ defined by 
  $\varpi(x) =(p_1(x),\dots,p_\ell(x))$, has rank~$\ell = \dim \h$
  at~$v$. It follows that
  $T_v(G \cdot v) = \Ker d_v \varpi = \C \tau(\g)_{\mid v} =
  \bigl\{[v, \xi] : \xi \in  \g \bigr\}$. Furthermore, there exist
  scalars $\lambda_j$ such that $\{z_j = p_j - \lambda_j : 
  1 \le j \le \ell\}$  forms part of a system of parameters of the
  local ring $\C[\p]_v$. Using \cite[Corollary~15.1.12]{MR} we can find
    derivations $\partial_i \in \Der \C[\p]_v$ such that
    $\partial_i(z_j) = \partial_i(p_j) = \delta_{i,j}$ for 
    $1 \le i,j \le \ell$. Set
    $N= \sum_{i=1}^\ell \C[\p]_v \partial_i$ and denote by
    $N_{\mid v} \subset T_v \p$ the space of tangent vectors
    defined by the elements $\partial_i$. From
    $\Ker d_v \varpi = \C \tau(\g)_{\mid v}$ one deduces that
\[
  \Der \C[\p]_v/ (\mathbf{m}_v\Der \C[\p]_v) = T_v \p = N_{\mid v}
  \oplus \C \tau(\g)_{\mid v}.
\]
Then, Nakayama's Lemma implies that $\Der \C[\p]_v = N + \C[\p]_v
\tau(\g)$. 

Let $P\in \KK(\p)_v:=\C[\p]_v\otimes_{\C[\p]}\KK(\p)$.  Then  one can write $P= P_0 + P_1$ with 
  $P_0 \in \dd(\p)_v \tau(\g)$ and (with the usual notation) $P_1 \in \sum_{\mathbf{i}\geq 0}
  \C[\p]_v \partial^{\mathbf{i}}$. If $P_1\not=0$, we can find some element $f\in \C[\p]^G=\C[p_1,\dots,p_{\ell}]$ such that $P_1\ast f\not=0$. Since $\tau(\g)\ast \C[\p]^G=0$, it follows that  $P\ast f\not=0$, contradicting the fact that $P\in \KK(\p)_v$.  
Thus, $P=P_0 \in \dd(\p)_v \tau(\g)$. This proves that $\LL(\p)_v= \{0\}$, as required.
   \end{proof}

    In order to understand the kernel of $\rad$,   and more generally $\KK(\p)$, we need to 
  reduce to the case of irreducible symmetric spaces and so we need to understand how $\KK(\p)$ relates to the corresponding objects for 
  summands of $\p$.  This will form the content of the next couple of results. 
  
   We begin with an abstract result  in which $\otimes$ will mean $\otimes_\C$.   For $i=1,2$, let $A_i$ 
   be  a $\C$-algebras with a left  ideal $J_i$. Given a left $A_1$-module $M$ and $0\not=f\in A_1$, then $M$ is called \emph{ $f$-torsionfree} 
    if $fm\not=0$ for all $0\not=m\in M$. 
    Set $A=A_1\otimes A_2$ and identify $A_1=A_1\otimes 1\subseteq A$ and $A_2=1\otimes A_2\subseteq A$. 
   Write  
   \[L=AJ_1+AJ_2=J_1\otimes A_2 +A_1\otimes J_2.\]
    
    \begin{lemma}\label{subsubsym} Keep the notation as above. For $i=1,2$, let $f_i\in A_i$ and assume that $A_i/J_i$ is $f_i$-torsionfree. Then:
    \begin{enumerate} 
    \item $A/L$ is $f_1$-torsionfree;
    \item if $f=f_1f_2\in A$ then $A/L$ is $f$-torsionfree.
    \end{enumerate}
    \end{lemma}
    
    \begin{proof}  (1)   As $\C$-vector spaces, choose a complementary summand  $J_2^\perp$ of $J_2$ inside $A_2$ and pick $\C$-bases 
    $\{\theta_{j}\}$ of $J_2 $, respectively $\{\phi_j\}$ of $J_2^\perp$. 
  Suppose that  $D\in A$ satisfies $f_1^sD\in L$ for some $s\geq 1$. We can write $D$ uniquely as 
  \[D=\sum_p \alpha_p\otimes \theta_p +\sum_q \beta_q\otimes \phi_q\qquad \text{ for some $\alpha_p,\beta_q\in A_1$.}\] 
  Now $L=A_1\otimes J_2+J_1\otimes J_2^\perp$ and  so we can  write 
  $f^sD =\sum_i\mu_i \theta_i + \sum_j \nu_j\phi_j$ for some $\mu_i\in A_1$ and $\nu_j\in J_1$. On the other hand, 
    \[f^sD=\sum_p (f^s\alpha_p)\otimes \theta_p +\sum_q (f^s\beta_q)\otimes \phi_q,\]
   and so, by uniqueness,  $f^s\alpha_p=\mu_p$ and $f^s\beta_q=\nu_q$ for all $p,q$. In particular each such $f^s\beta_q=\nu_q\in J_1$ and 
   hence $\beta_q\in J_1$ as $A_1/J_1$ is $f_1$-torsionfree. Therefore,
    \[ D=\sum_p \alpha_p\otimes \theta_p +\sum_q \beta_q\otimes \phi_q\in A_1\otimes J_2+J_1\otimes A_2=L;\]
    as required.
    
    (2) Suppose that $L\ni f^sD=f_1^sf_2^sD$ for some $D\in A$ and $s\geq 1$. Then part (1) implies that $f_2^sD\in L$. By the analogue of part (1) for 
    $f_2$, it follows that $D\in L$.
    \end{proof}
    
    We now return to symmetric spaces.  Let   $(G,\p)= (G_1,\p_1) \oplus (G_1,\p_2)$ be  a direct sum of symmetric spaces, and keep the resulting notation from Remark~\ref{rem:reducible-symm}.
 To simplify the notation, identify $\dd(\p_1)$ with $\dd(\p_1)\otimes 1$ inside $\dd(\p)=\dd(\p_1)\otimes \dd(\p_2)$ and similarly for $\dd(\p_2)$.   Define $\KK(\p_i)$ analogously to \eqref{eq:kernel1}; thus 
\[
\KK(\p_i)=\{ D\in \dd(\p_i) \ : \ D(\C[\p_i])^{G_i}=0\}.
\]  
   
   \begin{lemma}
  \label{sumsymmetricpairs}
 Let $(G,\p)= (G_1,\p_1) \oplus (G_2,\p_2)$ be a sum of symmetric
 spaces. Then 
  \[\KK(\p) = \Bigl(\KK(\p_1) \otimes \dd(\p_2)\Bigr) + \Bigl(\dd(\p_1)
  \otimes \KK(\p_2)\Bigr).\]
\end{lemma}

\begin{proof} Set $A_i=\dd(\p_i)$; thus $A=A_1\otimes A_2\cong \dd(\p)$. Writing $J_i=\KK(\p_i)$ for each $i$,  we need to prove that $L:=AJ_1+AJ_2$ equals $\KK(\p)$.
Since $\C[\p_1]^{G_1}\otimes\C[\p_2]^{G_2}=\C[\p]^G$, at least $L\subseteq \KK(\p)$.

If $0\not=f_i\in \C[\p_i]$ then we claim that the left $A_i$-module $ A_i/J_i$ is $f_i$-torsionfree. Indeed  suppose that $D\in A_i$ satisfies $f^s_i D\in J_i$ for some $s$;  thus $f^sD(\psi)=0$ for all $\psi\in \C[\p_i]^G$. Since $\C[\p_i]$ is a domain this implies that  $D(\psi)=0 $ for all such  $\psi$. In other words, $D\in J_i$ and so 
$ A_i/J_i$ is $f_i$-torsionfree, as claimed.

Now let $\deltav_i \in \C[\p_i]$ be the discriminant of the pair $(G_i,\p_i)$ and note that $\deltav=\deltav_1\deltav_2$ is the discriminant of  $(G,\p)$. By the previous paragraph each $A_i/J_i$ is $\deltav_i$-torsionfree. 
Now let $D\in \KK(\p)$. Then, by Lemma~\ref{LSresults}(2), 
there exists $s$ such that $\deltav^s D \in \dd(\p)\tau(\g)$. Since $\tau(\g_i)\subset J_i$ clearly $\dd(\p)\tau(\g)\subseteq L$ and hence $\deltav^s D\in L$. Thus Lemma~\ref{subsubsym}(2) implies that $D\in L$. Hence $L=\KK(\p)$, as required.
\end{proof}

  We can now prove the   main  result  of this section. 
  
\begin{theorem}\label{mainthm}
  Let $(G,\p)$ be a \gainly\ symmetric space. Then 
  $\KK(\p) = \dd(\p) \tau(\g),$ and so   $(\dd(\p)\tau(\g))^G = \ker(\rad_0). $
\end{theorem} 

\begin{proof}    It suffices to prove that $\LL(\p)=   0$. 
By Lemma~\ref{sumsymmetricpairs} it suffices to prove this when $(\gtilde,\vt)$ is irreducible.  Since  the nice
  symmetric spaces  are   covered by Theorem~\ref{LS3-theorem},  it
  remains to consider the case when $(\gtilde,\vt)$ satisfies $(\dagger)$.
 We argue by induction on the dimension $\dim \g$  of sub-symmetric pairs $(\gtilde_b,\g_b)$
  of $(\gtilde,\g)$.   

  If $(\gtilde_b,\g_b)$ is trivial, the
  assertion is clear, so assume not.   
  Suppose  that   $0\not=b\in \p$ is a semisimple element. By
  Lemma~\ref{integra-equ},   the centraliser $\g_b$ cannot equal
  $\g$.   Thus,  the slice $(\g_b,\p_b)$ is a proper sub-symmetric pair   and  so, by induction and 
  Lemma~\ref{distinguishedinduction}, the theorem holds for $(\g_b,\p_b)$.  Therefore,  \cite[Lemma~4.1]{LS3}(1) 
  implies that  $\Supp \LL(\p) \subseteq \NV$.  Now   the characteristic variety $\Ch \LL(\p)$ is contained in $  \Ch \bigl(\dd(\p)/\dd(\p)\tau(\g)\bigr)$, and hence in  $\euls{C}(\p)$.  Thus
  $\Ch \LL(\p) \subseteq \euls{C}(\p) \cap (\NV \times \p)$. 
  Since $(\gtilde,\g)$  satisfies $(\dagger)$,  Lemma~\ref{LSresults}(2) then implies that 
      \[\Supp \LL(\p)\ \subseteq \ \NV \cap (\p \sminus \p_m) \ = \  \NV \sminus \NVreg \ = \  Z(\p).  \]
     Therefore
  $\Ch \LL(\p) \subseteq \euls{C}(\p) \cap (Z(\p) \times \p)$. 
  
  If we
  show that 
  \begin{equation}\label{eq:mainthm}
  \Ch \LL(\p) \subseteq \euls{C}(\p) \cap (Z(\p) \times
  \NV),\end{equation}
then   Proposition~\ref{vanishing} will force $\LL(\p)=~0$, thereby proving the theorem.

Set  $\mathbf{m}=(\Sym \p)_+^G$. Recall that $x\in \NV$ if and only if $f(x)=0$ for all $f\in \C[\p]^G_+$. However, we have identified $\p \times \p^* = \p \times \p$ using $\varkappa$. Therefore, in order to prove \eqref{eq:mainthm}, it remains    to prove that any element of $\LL(\p)$ is annihilated, on the right, by a power of $\mathbf{m}$. To prove this we use the argument of   \cite[Corollary~3.9(2)]{LS3}.

  Observe that both $\dd(\p)/\dd(\p)\tau(\g)$ and    $\LL(\p)$ are  right $\dd(\p)^G$-modules. Denote by $I$ the annihilator of   $\LL(\p)$ as a $\dd(\p)^G$-module; it is a graded ideal in the
  naturally graded ring $\dd(\p)^G$, see 
  \cite[page~1725]{LS3}. Since 
  $\Ch \LL(\p) \subseteq \euls{C}(\p) \cap (\NV \times \p)$ 
  by~\cite[Lemma~2.2(i)]{LS3}, $\LL(\p)$ is a holonomic left $\dd(\p)$-module.
  Thus ${}_{\dd(\p)}\LL(\p)$ has
  finite length and so 
  $\mathrm{End}_{\dd(\p)}( \LL(\p))$ is finite dimensional. From
  $\dd(\p)^G/I \subseteq \mathrm{End}_{\dd(\p)}( \LL(\p))$ we 
  conclude that $I \cap (\Sym \p)^G$ has finite codimension in
  $(\Sym \p)^G$. Since $I$ is a graded ideal,  so is   $I \cap (\Sym \p)^G$ and   hence  $I \cap (\Sym
  \p)^G\supseteq F=\mathbf{m}^s$ for some $s$. Hence   $d F\in
  \dd(\p)\tau(\g)$ for all $d\in \KK(\p)$. Lemma~\ref{LSresults}(1) then  implies  that $F^t
  d \in \dd(\p)\tau(\g)$ for some $t$. In other words, \eqref{eq:mainthm} holds, which is enough to prove the theorem.
\end{proof}

\subsection*{Applications of Theorem~\ref{mainthm}}

The detailed description of  $\KK(\p)$ given by Theorem~\ref{mainthm} has  a number of applications, which we now describe.

First, as an immediate consequence of Theorem~\ref{mainthm} one obtains a
generalization of a fundamental result of Harish-Chandra 
 \cite[Theorem~5]{HC3}   from the
diagonal case to \gainly\
symmetric pairs. For nice symmetric spaces this was proved in \cite{LS3}.

\begin{corollary}
\label{CORB}  
 Let $(\gtilde,\g)$ be a  \gainly\ symmetric pair  that is the
complexification of a real symmetric pair $(\gtilde_0,\g_0)$. Write $G_0$
for the connected algebraic group satisfying $\Lie(G_0) = \g_0$. Let $U
\subset \p_0$ be a $G_0$-stable open subset and $T$ be a
$G_0$-invariant distribution on $U$.  Then $\KK(\p)^G\cdot T = 0$.  \qed
\end{corollary}

With a little extra work we are able to generalise a second result of Harish-Chandra \cite[Theorem~3]{HC2}
to \gainly\ symmetric spaces. We begin with a subsidiary result.

\begin{corollary}
  \label{maincor1}
Let $(G,\p)$ be a \gainly\ symmetric space. If  $\mathbf{n}
\subset (\Sym \p)^G$ is a maximal ideal and $f \in \C[\p]^G \sminus
\{0\}$, then $\dd(\p) = \dd(\p) f + \dd(\p) \tau(\g) + \dd(\p) \mathbf{n}$.  
\end{corollary}

\begin{proof}  By Corollary~\ref{thm42}, $\dd(\p)^G/\ker(\rad_0) \cong \Ak$ is simple and so any non-zero $\Ak$-module $Z$  has annihilator $\ann_{\Ak}Z=0$. Thus    \cite[Theorems~1.1 and~1.3]{LosevBernstein} imply that $\GKdim_{\Ak} Z \geq \dim \h$  for any finitely generated $\Ak$-module~$Z$.
        
  Set $N= \dd(\p)/(\dd(\p) f + \dd(\p) \tau(\g) + \dd(\p)
  \mathbf{n})$. Then  Theorem~\ref{mainthm} implies that  $  N^G =
  \Ak/(\Ak f + \Ak\mathbf{n})$.     
 By the first paragraph of this proof, in order to prove that  $N^G= 0$ and therefore $N=0,$  it suffices to prove that
  $\GKdim M < \dim \h$. 
  
  We prove this by mimicking  the proof of \cite[Corollary~5.7]{LS3}.  Under the order filtration from 
  Remark~\ref{rem:orderfiltrationH}, the associated graded ideal  is $\gr \mathbf{n}=(\Sym \p)^G_+$ and so 
  \[\Ch N \  \subseteq \  \euls{X} \ := \euls{C}(\p) \cap \bigl(\p\times \NV\bigr) \cap \bigl((f=0)\times \p \bigr).\]
 By \cite[Proposition~2.1]{LS3}, $\euls{X}\git G$ identifies with the subvariety $(f=0)$ of $\p\git G$. 
 Since $\p\git G$ is an irreducible variety of dimension $\ell=\dim \h$, this means that $\dim \euls{X}\git G<\ell$ and hence that  the Gelfand-Kirillov dimension
 $$\GKdim N^G  \ = \ \dim \Ch N\git G \  \leq \  \dim \euls{X}\git G  \ < \  \ell,$$
 as required. \end{proof}

The  significance of Corollary~\ref{maincor1} is that,    for \gainly\ symmetric pairs, it proves a result proved in \cite[Corollary~5.8]{LS3} for nice pairs, which in turn generalises   \cite[Theorem~3]{HC2}.  We just state the result; the notation and proof are essentially the same  as those used in  \cite{LS3}. In the case when $I=\C[\p]^G_+$, this reduces to Corollary~\ref{intro-maincor2}.

\begin{corollary}\label{maincor2} 
        Assume that the symmetric space $(G,\p)$ is \gainly\ and pick a real form $ (G_0,\p_0)$ of $(G,\p)$. Let $T$ be a $G$-invariant distribution on an open subset $U\subseteq \p_0$. Suppose that:
        
        \begin{enumerate}
                \item there exists an ideal $I\subset \C[\p]^G$ of finite codimension such that $I\cdot T=0$;
                
                \item $\mathrm{Supp}(T) \subseteq U\smallsetminus U\cap \p_{\reg}$.  \end{enumerate}
        Then $T=0$.  \qed
        
\end{corollary}

We end the section with several comments and applications of the earlier results. First, by
combining Proposition~\ref{Harish-Chandra-map} with  Theorems~\ref{simplicity-irred-symm-space} 
and~\ref{mainthm}, we deduce the following. 

\begin{corollary}\label{cor:Harish-Chandra-map}
        Assume that $(\gtilde,\g)$ is a \gainly \ symmetric pair such that the
        centre of $\g$ acts trivially on $\p$. Then, $\Im(\rad_\vs)$ is simple and $\ker(\rad_{\vs})=(\dd(\g) \tau(\g))^G$ for all choices of  $\vs$.   \qed
\end{corollary}

In particular, suppose that 
$(\gtilde,\g) = (\g \oplus \g, \g)$ is of diagonal type with $\g$ reductive.  Then the centre of $\g$ certainly acts trivially on $\g=\p$ and so we conclude: 

   \begin{corollary}   \label{Harish-Chandra-map21}
   If $(\gtilde,\g) = (\g \oplus \g, \g)$ is of diagonal type
   then $\Im(\rad_\vs) \cong \Im(\rad_0)$ is a simple ring for all
   choices of $\vs$. Also,  $\ker(\rad_{\vs})=\ker(\rad_0) =
   (\dd(\g) \tau(\g))^G$. \qed  
  \end{corollary}
  
     Corollary~\ref{Harish-Chandra-map21}    is  significant because the radial parts map that appears in the literature for the diagonal case is  not $\rad_0$.
      In more detail, the
   {Harish-Chandra homomorphism} $\HC : \dd(\g)^G\to \dd(\h/W)$
   studied in \cite{HC2, LSSurjective, LS, Sc2, Wallach}  differs from  $\rad$ by  conjugation with the element~$\delta^{ \half}$  
   thus, in our notation, $\HC = \rad_{-\half}$.   The details can be found,  for example, in  \cite[Theorem~1]{HC2} or
    \cite[Theorem~3.1]{Wallach}. This is important since,  as Harish-Chandra 
    noticed in \cite{HC2}, it is only after taking this shift that one obtains $\Im(\HC)\subseteq \dd(\h)^W$. 
      It was then  proved in \cite{LSSurjective, LS}
 that $\Im(\HC)= \dd(\h)^W$ and $\KK(\g)= \dd(\g) \tau(\g)$, thus
 $\Ker(\HC) = \bigl(\dd(\g) \tau(\g)\bigr)^G$.
 Up to an isomorphism, Corollary~\ref{Harish-Chandra-map21}  recovers these result.

We end this section with a remark about a possible generalisation of the previous results. A large class of polar representations is given by Vinberg's  $\theta$-rep\-resent\-ations \cite{Vinberg}. These are defined as
  follows. Let $\Gtilde$ be as above and
  $\theta \colon \gtilde \to \gtilde$ be an automorphism of order
  $2\leq m < \infty$. Then $\theta$ defines a
  $(\mathbb{Z} / m \mathbb{Z})$-grading
  $\gtilde = \bigoplus_{k=0}^{m-1} \gtilde_k$, where $\gtilde_0$ is a
  reductive subalgebra of $\gtilde$. Let $G_0 \subseteq \Gtilde$ be the
  connected subgroup with Lie algebra $\gtilde_0$. The group $G_0$ acts on
  $V := \gtilde_1$ and we define $G$ to be the image of $G_0$ in $\GL(V)$.
  The pair $(G,V)$ is the \emph{$\theta$-representation} associated to
  $(\gtilde,\theta)$. By \cite[Theorem~4]{Vinberg} and \cite[p.~245 and
  247]{PopovVinberg}, $\theta$-representations are  visible polar
  representations. When $m = 2$, $\theta$-representations are the same as
  symmetric spaces and are therefore stable but for $m \ge 3$ they need not
  be stable; see \cite[p.~245]{PopovVinberg}.  For a large  family of stable polar
   representations that are   not
   $\theta$-representations, see \cite[Section~8]{BLT}.
  
  It would be interesting to know to what extent the results of the last several sections generalise to 
  $\theta$-rep\-resent\-ations. 
  Recall from the proof of Theorem~\ref{thm:symmLSsurjective}
  that Tevelev's Theorem \cite{Te} says that the morphism $\C[V \oplus
  V^*]^G \to \C[\h \oplus \h^*]^W$ is surjective for symmetric
  spaces.  We ask:
  
  \begin{question} \label{Tevelev-question} Let $V$ be a
    $\theta$-representation with a corresponding radial parts map
    $\rad_{\vs} : \dd(V)^G\to \Ak$. Then:
  
  (1) Is $\rad_{\vs}$ always surjective?

(2)    Does Tevelev's Theorem   extend to $\theta$-representations?

\noindent
If  Question (2) holds, then  so does Question~(1), as can be seen by examining the proof of Theorem~\ref{thm:symmLSsurjective}.
\end{question}
 
 \appendix

\section{Non-simplicity of certain  spherical algebras} \label{appA}

The main aim of this appendix is to complete the proof of
Theorem~\ref{simplicity-irred-symm-space}, by showing that $\Ak$
is not simple for  the symmetric spaces not covered by
Corollary~\ref{thm42}.

\subsection*{Definitions}

  The computations in this appendix become much simpler if we adopt the notation of \cite{BEG,EG}, where
the rational Cherednik algebra $H_c(W)$ is defined in terms of a
multiplicity function $c$ as opposed to  $\kappa$ and so we begin
with the relevant notation.

We assume that $W$ is the Weyl group associated to a finite root
system, hence $\h$ is the complexification of an euclidean vector
space and $W$ is generated by orthogonal reflections.
 Let $R\subset \h^*$ be the set of  roots and make  a choice of positive roots $R^+ \subset R$ so that
$\calS= \{s_\alpha : \alpha \in R^+\}$  is the set of reflections in~$W$. If
$s = s_\alpha$ with $\alpha \in R$, we set $\alpha_s =\alpha$. Let $c: \calS \to \C$  
be a $W$-invariant function, setting $c_s= c(s)$ for $s\in \calS$ and $c_s=c_\alpha$ if
$s= s_\alpha$. The set of such maps,  called multiplicity
functions, is denoted by~$\calC(W)$ (it is called $\C[R]^W$ in
\cite{BEG}). The rational Cherednik algebra
$H_c(W) = \C \langle \h^*, \, \h, \, W \rangle$ associated to 
$c \in \calC(W)$ is  defined in 
\cite[Section~1]{BEG}.  For our purposes it suffices to note that  the Dunkl operator $T_y(c)$ associated to $y \in \h$,     as  defined in \cite[page~288]{BEG} and \cite[\S 2.1]{CE}, is 
\begin{equation} \label{eq10}
      T_y(c) = \partial_y - \! \sum_{^{_{s \in \calS}}}
      c_s \scal{\alpha_s}{y} \frac{1 -s}{\alpha_s} \in \End_\C \C[\h].
\end{equation}

\begin{remark} \label{comparison}
      Comparing the element $T_y$ from~\eqref{eq:Dunkl} with ~\eqref{eq10} shows that $\kappa_{H,0} = 0$ and $\kappa_{H,1} = - c_{\alpha}$ for $H = \ker(\alpha)$ and $\alpha \in R^+$. 
\end{remark}

\subsection*{Twist by a character}\label{twist}

Let $\chi \in \Hom(W,\C^*)$ be a linear character of~$W$. The
idempotent $e_\chi \in \C W$ associated to $\chi$ is  $e_\chi =
\frac{1}{|W|} \sum_{w \in W} \chi^{-1}(w) w$. 
Let $M$ be a $W$-module; denote by $M[\chi]$ the
$\chi$-isotypic component of $M$. Then:
\[
M[\chi] := \{x \in M \, | \, \forall\, w \in W, \;
wx = \chi(w)x\} = \{e_\chi x \, : \, x \in M\} =: e_\chi M.
\]
Observe that, since $\C W \subset \cheW{c}$, the above can be applied to
any $\cheW{c}$-module.

In particular, if $\chi = \triv$, the trivial character, one gets the
trivial idempotent $\etriv$, also denoted by $e_W$, and the set of
$W$-invariants in $M$:
\[
\etriv = \frac{1}{|W|} \sum_{w \in W}  w, \quad 
M[\triv] = \etriv M = M^W.
\]
The sign character $\sgn : W \to \{\pm 1\}$ is given by
$\sgn(w)= \det_\h(w)$.

If $c \in \calC(W)$ is a multiplicity function, one defines
$c^\chi \in \calC(W)$ by $c^\chi_s = \chi(s)c_s$ {for all $s \in \calS$.}  When
$\chi = \sgn$ we therefore have $c^{\sgn} =-c$.

The next lemma is a standard result in the literature; see \cite[Section~5.1]{BerestChalykhQuasi}.

\begin{lemma}
        \label{lemtwist}
        Let $\chi \in \Hom(W,\C^*)$. There exists an algebra isomorphism
        \[
        f_\chi : \cheW{c} \longisomto \cheW{\cchi}
        \]
        given by $f_\chi(x)=x$, $f_\chi(y) = y$, $f_\chi(w)
        = \chi(w)w$, for all $x \in \h^*$, $y \in \h$ and $w \in W$.
        The inverse of
        $f_\chi$ is $f_{\chi^{-1}} : \cheW{\cchi} \isomto  \cheW{c}$.\qed
\end{lemma}

\begin{remark}  
        Note that $f_\chi(e_\chi)= \etriv$ and  that in terms of Dunkl operators we have $f_\chi(T_y(c)) =
        T_y(\cchi)$.
\end{remark}

Let $M$ be a $\cheW{\cchi}$-module. The twist of $M$ by $\chi$
is the $\cheW{c}$-module $M^{f_\chi}$ defined by $M^{f_\chi}
= M$ as a vector space and $a\cdot m = f_\chi(a) m$ for all $a \in
\cheW{c}$ and $m \in M$. Using this notation one easily proves the following. 

\begin{lemma}
        \label{lem22}
        {\rm (1)} The functor $M \mapsto M^{f_\chi}$ is an equivalence of
        categories between $\cheW{\cchi}$-$\mathrm{Mod}$ and
        $\cheW{c}$-$\mathrm{Mod}$. Its inverse is  given by $N \to
        N^{f_{\chi^{-1}}}$.
        
        {\rm (2)} Let $M$ be an $\cheW{\cchi}$-module. Then, the
        vector subspace $M^W= M[\triv]$ of the $W$-module $M$ is equal
        to the isotypic component $M^{f_\chi}[\chi]$ of the $W$-module
        $M^{f_\chi}$. In particular:
        \[
        M^W = M^{f_{\sgn}}[\sgn] \ \; \text{for any $\cheW{-c}$-module $M$.} \qed
        \]
\end{lemma}

Let $\chi \in \Hom(W,\C^*)$  and set:
\begin{equation}
        \label{eq1}
        A_{c,\chi}(W)= e_\chi \chec e_\chi, \quad A_c(W) =
        A_{c,\triv}(W) = \etriv \chec \etriv
\end{equation}

Each $A_{c,\chi}(W) \subset \chec$ is an algebra whose identity
element is the idempotent $e_\chi$ and  $A_c(W)$ is the usual spherical subalgebra of $H_c(W)$.  Note that if $M$ is an $\chec$-module, the isotypic component $M[\chi] = e_\chi M$ is a
module over the algebra $A_{c,\chi}(W)$.  In particular, $M^W = \etriv M = e_W M$ is an
$A_c(W)$-module  and $M[\sgn]$ is a module
over $A_{c,\sgn}(W)$.

Recall that $f_{\sgn}(\esgn) = \etriv$ and $c^{\sgn}= -c$; via the isomorphism
$f_{\sgn}$  one gets the isomorphism of algebras:
\begin{equation}
        \label{eq2}
        f_{\sgn} \, :  \, A_{c,\sgn}(W)= \esgn \chec \esgn \, \longisomto
        \, A_{-c}(W) 
        = \etriv \cheW{-c} \etriv.
\end{equation}

We are  interested in the (non) simplicity of the
spherical algebra for certain Weyl groups and multiplicity functions.
We will need some  results from~\cite{LosevTotally}. First, recall that the
multiplicity $c \in \calC(W)$ is said to be \emph{totally aspherical} if one has the
following equivalence:
$M$ is a module of the category  $\calO_c(W)$ which
is torsion as an $S(\h^*)$-module if and only if  $M^W = \etriv M=
\{0\}$; see \cite{BEG} or \cite{LosevTotally} for a definition of $\calO_c(W)$.

Let $W' \subset W$ be a parabolic subgroup; that is, a stabiliser in $W$ of
some $p \in \h$. Decompose $\h$ as $\h=\h^{W'} \boplus \h'$ where
$\h' \cong \h/\h^{W'}$ is
the unique $W'$-complement of the set of fixed points $\h^{W'}$.  Set
$\calS'= \calS \cap W'$ and define a $W'$-invariant function $c' : \calS' \to \C$ by
$c' = c_{\mid \calS'} \in \calC(W')$. One can then construct a Cherednik algebra
$\cher{c'}{W'}$ from  the datum $(W',\h',c')$, cf.~\cite{BE}.

\begin{proposition}[\cite{LosevTotally}, Proposition~2.7, Lemmata~2.8 \&~2.9]
        \label{thm24}
        The following  are equivalent:
        \begin{enumerate} 
                \item $c \in \calC(W)$ is  totally aspherical;
                \item $e_{W'}M' = \{0\}$ for any parabolic subgroup $W' \ne \{1\}$ and
                any finite dimensional $\cher{c'}{W'}$-module, where $e_{W'}$ is the
                trivial idempotent in $\C W'$; 
                \item the spherical algebra $A_c(W)$ is simple.\qed
        \end{enumerate}
\end{proposition}

This proposition, combined with  Lemma~\ref{lem22} and~\eqref{eq2}, implies the following
result.

\begin{corollary}\label{thm24-cor}
If there exists a parabolic subgroup $W' \subset W$ and a finite
dimensional $\cher{c'}{W'}$-module $L$ such that $L[\sgn] \ne \{0\}$, then
the spherical algebra $A_{-c}(W)$ is not simple.\qed
\end{corollary}


\subsection*{Finite dimensional modules in the $\Asf_1$ case}
\label{sec0}

We assume in this subsection that $R$ is a root system of type
$\Asf_1$. 
 Therefore $\h^*= \C x$,  $W = \{1,s\}$ and $c\in \calC(W)$ is determined by $c=c_s$.   The algebra $\chec$ is isomorphic to the $\C$-algebra generated
by $x,s, T$ where
\[
T=  \frac{\partial}{\partial x} - \frac{c}{x}(1-s) \in \End_\C \C[x]. 
\]
We  now make the hypothesis:

\begin{mainhypo}
        \label{hypo0}
        $c= m+ \half$ \, with $m \in \N$.
\end{mainhypo}

Then $T(x^{2m+1}) = 0$ and the standard module $M_c(\triv) \cong
\C[x]$ has a unique irreducible quotient, namely: $L_c(\triv) = \C[x] \big/\left(\boplus_{ j \ge 2m +1} \C
x^j\right) = \boplus_{j=0}^{2m} \C t^j$,
where $t$ is the class of $x$. Since $s(t) = -t$ we get:
\begin{gather*}
        L_c(\triv)^W = L_c(\triv)[\triv] =  \boplus_{i=0}^{m} \C
        t^{2i},
        \\
        L_c(\triv)[\sgn] =
        \begin{cases}
                \{0\} & \ \text{if $m=0$,}
                \\
                \boplus_{i=0}^{m-1} \C t^{2i+1} & \ \text{if $m\ge 1$.}
        \end{cases}
\end{gather*}

Therefore, if we set
$M = L_c(\triv)^{f_{\sgn}}$, then the $\cheW{-c}$-module $M$ satisfies
$M^W = \{0\}$ if $c= \half$ and $M^W \ne \{0\}$ for $m \ge 1$ by Lemma~\ref{lem22}.
Thus, we have:

\begin{corollary}
        \label{A1-case}
        Let $R$ be a root system of type $\Asf_1$ and, in the above notation.
        assume that $m \ge 1$, equivalently if $c= \frac{d}{2}$ for an odd integer  $d\geq 3$.  Then there exists a nonzero finite dimensional module over
        $A_{-c}(W)$. \qed
\end{corollary}

When $(\gtilde,\vt)$ is an irreducible symmetric pair of rank~$1$, the
parameter $k$ (given by $\kappa)$ is always of the form $k=m$ or $k= m+ \half$ for some $m \in \N$ (see the
tables in Appendix~\ref{app-tables}). When $k = \half$, it is easy to see
that $A_\kappa$ is simple, see \cite [\S6]{LS3}, or observe that
$(\gtilde,\vt)$ is nice.  We therefore deduce the following result.

\begin{corollary}
        \label{A1-case2}
        If $(\gtilde,\vt)$ is an irreducible symmetric pair of rank~$1$, then
        $A_\kappa(W)$ is not simple if and only if $k=m+ \half$ with $m \ge 1$. \qed
\end{corollary}

\subsection*{Finite dimensional modules in the $\Bsf_p$ case}
\label{sec1}

In this subsection we recall some results from~\cite{CE}.  Let $\fS_p$ be the
$p$th symmetric group and set $\mu_2 = \{\pm 1\}$.  We consider here the
Weyl group $W = \fS_p \ltimes \mu_2^p$ of a root system of type $\Bsf_p$ (or
$\Csf_p$) with $p \ge
2$. It acts on the $p$-dimensional space $\h=\C^p$. The set $\calS \subset W$
consists of the elements $s_i$, $1 \le i \le p$, and $\sigma_{i,j}^{(m)}$,
$1 \le i < j \le p$, defined by:
\begin{gather*}
        s_i(x_1, \dots, x_i, \dots, x_p) = 
        (x_1, \dots, -x_i,\dots, x_p), 
        \\
        \sigma_{i,j}^{(m)}(x_1 ,\dots, x_i ,\dots, x_j, \dots, x_p)=
        (x_1, \dots, (-1)^m x_j, \dots, (-1)^m x_i ,\dots, x_p),  
\end{gather*}
for $m=0,1$. A multiplicity function $c \in \calC(W)$ is given by  two complex numbers $\{c_1,c_2\}$:
\begin{equation} \label{multiplicity-1}
        \text{$c_1=c_{s_i}$ for all $i$ and
                $c_2= c_{\sigma_{i,j}^{(m)}}$   for all $i, j, m$}.
\end{equation}

\begin{remark}
        \label{multiplicity-2}
        When $R$ is a root system of type $\Bsf_p$, respectively~$\Csf_p$, $s_i$ is the
        reflection associated to a short, respectively~long,  root and $\sigma_{i,j}^{(0)}$
        is associated to a long, respectively~short, root.  See \cite[Planches II \& III]{Bou}.
\end{remark}

We  will work under the following hypothesis, see \cite[\S 4.1]{CE}.

\begin{mainhypo}
        \label{hypo1}
        Set $
        r= 2 (c_1 + (p-1)c_2).$
        We assume that $r\geq 1$ is an odd integer and write
        $r= 2m -1$ for $m \ge 1$.
\end{mainhypo}

Let $\tilde{u}$ be an indeterminate over $\C$ and define an
$r$-dimensional vector space by
\[
U_r = \C[\tilde{u}]/(\tilde{u}^r) = \boplus_{i=0}^{r-1} \C u^i
\]
where $u$ is the class of $\tilde{u}$ modulo~$(\tilde{u}^r)$.
Let
$X_r= U_r^{\otimes p}$, thus $\dim X_r= r^p$.
A basis of the vector space $X_r$ is given by the $u^{\mathbf{j}} = u^{j_1}
\otimes \cdots \otimes u^{j_p}$ where $\mathbf{j}=(j_1,\dots,j_p)
\in \{0,\dots,r-1\}^p$. 
The space $X_r$ has a natural
structure of graded $W$-module which can be described as follows.

The group $\mu_2$ acts  on $U_r$ by
$(\pm 1) \cdot u^j= (\pm 1)^{j} u^j$ for $j=0,\dots,r-1$, and
$(\epsilon_1,\dots,\epsilon_p) \in \mu_2^p$ acts on
$a=a_1 \otimes \cdots \otimes a_p \in X_r$ via
$\epsilon \cdot a= \epsilon_1\cdot a_1 \otimes \cdots \otimes \epsilon_p
\cdot a_p$.  Let $w=(\sigma,\epsilon) \in W$ with $\sigma \in \fS_p$ and
$\epsilon =(\epsilon_1,\dots,\epsilon_p) \in \mu_2^p$, then the action of
$w$ on $a=a_1 \otimes \cdots \otimes a_p \in X_r$ is: $w \cdot a= \sigma\cdot(\epsilon \cdot a)= \epsilon_{\sigma(1)} \cdot
a_{\sigma(1)} \otimes \cdots \otimes \epsilon_{\sigma(p)} \cdot a_{\sigma(p)}$.

Recall that $M_c(\triv) = \chec \otimes_{\C W \ltimes S(\h)} \C_{\triv}$ is the
standard representation of $\chec$ with trivial lowest weight; this is the
polynomial representation, on which the elements of $\h$ act via  Dunkl
operators.  

\begin{proposition}[\cite{CE}, Theorem~4.2]
        \label{thm31}
        Under Hypothesis~\ref{hypo1}, there exists a quotient $Y_r$ of  
        $M_c(\triv)$ which is isomorphic to $X_r$ as a graded $W$-module. \qed
\end{proposition}

In order to apply Lemma~\ref{lem22}, we want to describe
$X_r[\sgn]$, the isotypic component of type $\sgn$ of the
$W$-module~$X_r$.

The reflection $s_k$ identifies with the element
$(1,\dots,1,-1,1,\dots 1) \in W$, where $-1$ is in position
$k$. Therefore, $s_k$ acts on $u^{\mathbf{j}}$ by
$s_k \cdot u^{\mathbf{j}} = (-1)^{j_k} u^{\mathbf{j}}$. Let
$x= \sum_{\mathbf{j}} \lambda_{\mathbf{j}} u^{\mathbf{j}} \in
X_r$.  Then, $s_k \cdot x= \sgn(s_k) x = -x$ if and only if  
$\sum_{\mathbf{j}} (-1)^{j_k} \lambda_{\mathbf{j}}u^{\mathbf{j}}
= -\sum_{\mathbf{j}} \lambda_{\mathbf{j}} u^{\mathbf{j}}$.
Equivalently,
$(-1)^{j_k +1}\lambda_{\mathbf{j}} = \lambda_{\mathbf{j}}$ for
all $\mathbf{j}$; that is $\lambda_{\mathbf{j}} =0$ when $j_k$ is
even. This shows that the elements $u^{\mathbf{j}}$ such that $j_k$
is odd for all $k$ give a basis of the space
$X_r'= \{x \in X_r : s_k \cdot x = -x, \ k=1,\dots, p\}$. Since
$j \in \{0,\dots,r-1\}$ is odd if only if 
$j=2i -1$ with $i\in\{1,\dots,m-1\}$, we have
$\dim X'_r= (m-1)^p$. Note that $X'_r$ is an $\fS_p$-submodule
of $X_r$ and, if one sets $U'_r= \boplus_{i=1}^{m-1} \C u^{2i-1} \subset U_r$,
we have $X'_r= (U'_r)^{\otimes p} \subset X_r$.  Since
\[
X_r[\sgn] = \{x \in X'_r \, : \, \sigma \cdot x= \sgn(\sigma) x \
\text{for all $\sigma \in \fS_p$}\},
\]
the $W$-module $X_r[\sgn]$ identifies with the space of
antisymmetric tensors in $ (U'_r)^{\otimes p}$, that is to say
with the $p$-th exterior product $\bigwedge^p U'_r$.

We can summarise the previous discussion in the next three corollaries.

\begin{corollary}
        \label{cor32}
        Under  Hypothesis~\ref{hypo1} and the notation of
        Proposition~\ref{thm31}:
        \begin{enumerate} 
                \item  the isotypic component $Y_r[\sgn]$ is isomorphic to
                $\bigwedge^p U'_r$ (with the above $W$-module structure);
                \item $Y_r[\sgn] = \{0\}$ if and only if  $m < p+1$ and $\dim
                Y_r[\sgn] = \binom{m-1}{p}$ if $m \ge p+1$.\qed
        \end{enumerate}
\end{corollary}

\begin{remark} \label{eq4}
        In terms of the multiplicity function $c$, the condition
        $m \ge p+1$ can be rewritten as
        $  c_1 + (p-1) c_2 \ge p + \half.$
\end{remark}

\begin{corollary}
        \label{cor33}
        Assume that $c_1 + (p-1)c_2 \in \half+\mathbb{Z}$ 
        and $c_1 + (p-1) c_2 \ge p + \half$. Then:
        
        \begin{enumerate}
                \item there exists a finite dimensional
                $\cheW{-c}$-module $M$ such that
                \[M^W =\etriv M \ne \{0\};\] 
                \item the spherical subalgebra
                $A_{-c}(W) \cong \esgn A_c(W) \esgn$ is not simple.\qed
\end{enumerate}  \end{corollary}

\begin{proof}
        Recall that there exist isomorphisms of algebras:
        \[
        f_{\sgn} : \cheW{-c} \, \isomto \, \cheW{c}, \quad f_{\sgn} :
        A_{-c}(W) \, \isomto \, \esgn A_c(W) \esgn. 
        \]
        Set $M= Y_r^{f_{\sgn}}$, where $Y_r$ is an in
        Proposition~\ref{thm31}. Then, $M$ is an $\cheW{-c}$ module; using
        Lemma~\ref{lem22}  and Corollary~\ref{cor32} we get $M^W = Y_r[\sgn] \ne \{0\}$.
        This proves (1), from which (2) is an immediate consequence.
\end{proof}

When $p=2$ the expression $r= 2(c_1 + (p-1) c_2) =2(c_1+c_2)$ is symmetric in
$c_1$ and $c_2$, and so we obtain:

\begin{corollary}
        \label{ex34}
        Let $R$ be a root system of type $\Bsf_2$ or $\Csf_2$. Write
        $W= \langle s_{\alpha}, s_{\beta} \rangle$ where $\alpha$ and
        $\beta$ have different length.  Suppose that
        $r= 2(c_{s_\alpha} + c_{s_\beta}) = 2m -1$ where $m$ is an
        integer~$\ge 3$. Then there exists a finite dimensional
        $\cheW{-c}$-module $M$ such that $e_{\triv} M = M^W \ne
        \{0\}$. \qed
\end{corollary}


\subsection*{Application}
\label{sec2}

 We give here examples of multiplicity
functions such that the spherical algebra $A_{-c}(W)$ is
\emph{not simple}. Our notation for root systems
follows \cite{Bou}.

Let $(\alpha_1,\dots,\alpha_\ell)$ be a basis of $R$. If $I$ is a
subset of $\{1,\dots,\ell\}$, denote by $R_I$ the root system
generated by the $\alpha_i$, $i \in I$, and let $W'=W_I$  be the parabolic subgroup generated by the $\{s_{\alpha_i} :i \in I\}$. 
Prior to Proposition~\ref{thm24}, for each  multiplicity function $c \in \calC(W)$ we   defined
a  multiplicity function $c'= c_{\mid \calS \cap W'} \in \calC(W')$  and hence we obtain  a Cherednik algebra $\cher{c'}{W'}$.

\begin{proposition}
        \label{lem21}
        Adopt the previous notation. Suppose that one of the following
        hypothesis holds.
        \begin{enumerate}
                \item[(a)]  Let $I=\{j\}$ and suppose that
                $c_{s_{\alpha_j}}= m + \half$ with $m \in \N^*$.  
                
                \item[(b)] Let $I=\{j, j+1\}$ with $R_I$ of type $\Bsf_2$ or
                $\Csf_2$ and
                suppose, for some integer $m\geq 3$,  that $r= 2( c_{s_{\alpha_j}} + c_{s_{\alpha_{j+1}}}) = 2m -1$.
        \end{enumerate}
        
        Then the spherical algebra $A_{-c}(W)$ is not simple.
\end{proposition}

\begin{proof}
        If   (a) holds, the result follows from 
        Corollary~\ref{A1-case} and Proposition~\ref{thm24}.
        
        If (b) holds, apply Corollary~\ref{ex34} and Proposition~\ref{thm24}. 
\end{proof}

\begin{examples} \label{computations}
        For each of the eleven following
        cases, we indicate the type of $R$, the conditions on the
        integers $p,q,n$, and the values of the multiplicity $c$. We write
        $c_{\mathrm{sh}}$, resp.~$c_{\mathrm{lg}}$, for $c_{\alpha}$
        with $\alpha$ short, respectively~long. We then explain how to use  
        Proposition~\ref{lem21} to obtain the non-simplicity of
        $A_{-c}(W)$.
        \smallskip
        
        \emph{The choice of the multiplicities (and the numbering of the
                different cases) will be justified in  Remark~\ref{rem:unknown-cases}.}


        \smallskip
        
        (\ref{computations}.1) $\Asf \mathrm{III}_{p,q}$: $R= \Bsf_p$, $p+q= n+1$, $2 \le p
        \le \half n <q$; $\csh = n -2p + \frac{3}{2} = q-p + \half$,
        $\clg = 1$.

        The non simplicity of $A_{-c}(W)$ follows from
        Proposition~\ref{lem21}(a) applied with $I=\{\alpha_j\}$ where $\alpha_j$
        is short, since $c_{s_{\alpha_j}} = q -p + \half$ with $q-p \ge 1$.

        \smallskip
        (\ref{computations}.2) $\Bsf \mathrm{I}_{p,q}$: $R= \Bsf_p$, $p+q= 2n+1$,
        $2 \le p < n$; $\csh = n -p + \half = \half(q-p)$, $\clg= \half $.
        
        Take $I=\{\alpha_j\}$ where $\alpha_j$ is short. Observe that $p\pm q$ is
        odd and $p <n$ implies $q \ge p+3$. Then,
        $c_{s_{\alpha_j}} = \half(q -p) = m + \half$ with $m = \half(q-p -1)$ integer
        $\ge 1$. Therefore, Proposition~\ref{lem21}(a) shows  that $A_{-c}(W)$ is
        not simple.

        \smallskip
        (\ref{computations}.3) $\Csf \mathrm{II}_{p,q}$: $R= \Bsf_p$, $p+q= n$, $2 \le p
        \le \half(n-1) <q$; $\csh = 2n -4p + \frac{3}{2} = 2q-2p
        +\frac{3}{2}$, $\clg= 2$. 
        
        The non simplicity of $A_{-c}(W)$ follows from
        Proposition~\ref{lem21}(a) applied with $I=\{\alpha_j\}$ where $\alpha_j$
        is short, since $c_{s_{\alpha_j}} = 2q -2p + 1 + \frac{1}{2}$ with
        $2q-2p +1 > 1$.

        \smallskip
        (\ref{computations}.4) $\Csf \mathrm{II}_{p,p}$: $R= \Csf_p$,  $2 \le p$; $\clg=
        \frac{3}{2}$, $\csh= 2$.  
        
        The non simplicity of $A_{-c}(W)$ follows from 
        Proposition~\ref{lem21}(a) applied with $I=\{\alpha_j\}$ where $\alpha_j$
        is long, since $c_{s_{\alpha_j}} = \frac{3}{2}$.

        \smallskip
        (\ref{computations}.5) $\Dsf \mathrm{I}_{p,q}$: $R= \Bsf_p$, $p+q= 2n$,
        $2 \le p \le n -2 <q$; $\csh= n -p = \half(q-p)$, $\clg = \half $.
        Note that $q-p$ is even and $q \ge p+4$ (since
        $2p \le 2n-4= p+q-4)$.
        
        If $I=\{p-1,p\}$, the root system $R_I$
        is of type $\Bsf_2$; one has $c_{s_{\alpha_{p}}} = \half(q-p)$
        and $c_{s_{\alpha_{p-1}}}= \half$.  Apply Proposition~\ref{lem21}(b): we
        have $r= q-p+1= 2m -1$ where $m = \half(q-p) +1$ is an integer
        $\ge 3$. Thus $A_{-c}(W)$ is not simple.

        \smallskip
        (\ref{computations}.6) $\Dsf \mathrm{III}_{4p}$: $R= \Csf_p$,  $2 \le p$; $\clg =
        \frac{1}{2}$, $\csh = 2$.  
        
        If $I=\{p-1,p\}$, the root system $R_I$
        is of type $\Csf_2$; one has $c_{s_{\alpha_{p}}} = 2$
        and $c_{s_{\alpha_{p-1}}}= \half$.  Apply Proposition~\ref{lem21}(b): we
        have $r= 5 = 2m -1$ where $m = 3$. Thus $A_{-c}(W)$ is not simple.
        \smallskip
        
        (\ref{computations}.7) $\Dsf \mathrm{III}_{4p+2}$: $R= \Bsf_p$, $2 \le p$;
        $\csh = \frac{5}{2}$, $\clg = 2$.
        
        The non simplicity of $A_{-c}(W)$ is obtained by applying
        Proposition~\ref{lem21}(a) with $I=\{\alpha_p\}$, since
        $c_{s_{\alpha_p}} = \frac{5}{2}$.
        
        \smallskip
        (\ref{computations}.8) $\Esf \mathrm{III}$: $R= \Bsf_2$, $p=2$; $\csh = \frac{9}{2}$,
        $\clg = 3$.

        The non simplicity of $A_{-c}(W)$ is obtained by
        applying Proposition~\ref{lem21}(a) with $I=\{\alpha_2\}$, since
        $c_{s_{\alpha_2}} = \frac{9}{2}$.
        
        \smallskip
        (\ref{computations}.9) $\Esf \mathrm{VI}$: $R= \Fsf_4$, $p=4$; $\csh = 2$,
        $\clg = \half$.
        
        Take $I=\{2,3\}$, then $R_I$ is of type $\Bsf_2$; we have
        $c_{s_{\alpha_2}} = \half$ and $c_{s_{\alpha_3}} = 2$. Therefore,
        $r= 5 =2m -1$ with $m=3$ and
        Proposition~\ref{lem21}(b) gives the non simplicity of $A_{-c}(W)$.
        
        \smallskip
        (\ref{computations}.10) $\Esf \mathrm{VII}$: $R= \Csf_3$, $p=3$; $\clg =
        \frac{1}{2}$, $\csh = 4$.  
        
        Take $I=\{2,3\}$, then $R_I$ is of type $\Csf_2$; we have
        $c_{s_{\alpha_2}} = 4$ and $c_{s_{\alpha_3}} = \half$. Therefore,
        $r= 9 =2m -1$ with $m=5$ and
        Proposition~\ref{lem21}(b) gives the non simplicity of
        $ A_{-c}(W)$.
        \smallskip
        
        (\ref{computations}.11) $\Esf \mathrm{IX}$: $R= \Fsf_4$, $p=4$; $\csh = 4$, $\clg = \half$.  
        
        Take $I=\{2,3\}$, then $R_I$ is of type $\Bsf_2$; we have
        $c_{s_{\alpha_2}} = \half$ and $c_{s_{\alpha_3}} = 4$. We get $r= 9
        =2m -1$ with $m=5$; thus, by Proposition~\ref{lem21}(b), $A_{-c}(W)$ is not simple.
\end{examples}

\subsection*{Radial components}
\label{sec4}

Let $\gtilde$ be a finite dimensional semisimple complex Lie
algebra. We adopt the notation introduced in  Section~\ref{Sec:examples}
for symmetric pairs $(\gtilde, \vt) = (\gtilde,\g)$ and the
associated stable polar representation $(G,V=\p)$.

Recall the multiplicity function $k : R \to \half \N$ defined by:
\begin{equation} \label{k-multiplicity}
        \forall \, \alpha \in R^+, \ \; k_\alpha =
        \half\bigl(\dim \gtilde^\alpha + \dim \gtilde^{2\alpha}\bigr).
\end{equation}

        \label{rem41}
        Define the multiplicity $c : \calS \to \C$ by $c_{s_\alpha} = - k_\alpha$ for all $\alpha \in R$.  
    Then the rational Cherednik
        algebra $H_k(W)$ associated to the symmetric pair $(\gtilde,\vartheta)$
        and the multiplicity $k$ is isomorphic to the Cherednik algebra
        $H_{-c}(W)$ and, for the spherical algebras, we have
        $A_k(W) \cong A_{-c}(W)$.

\begin{remark} \label{rem:unknown-cases} Suppose that
        $(\gtilde,\vt)$ is irreducible. From the classification of
        symmetric spaces, see table in Appendix~\ref{app-tables}, the pairs for which the
        simplicity of $\Im(\rad)$ cannot be determined by
        Corollary~\ref{thm42}, or the study of the rank one case,  are
        the eleven cases  (\ref{computations}.1--\ref{computations}.11). \end{remark}

By combining Lemma~\ref{reducible-spaces},   Corollary~\ref{thm42}  and Appendix~\ref{app-tables}
 with the non simplicity of the spherical
subalgebra $A_k(W) \cong A_{-c}(W)$ in the cases (1) to (11) above, we get
the following result, announced in
Theorem~\ref{simplicity-irred-symm-space}.

\begin{theorem}
        \label{appendix-irred-symm-space}
        Let $(\gtilde,\vt)$ be an irreducible symmetric space.
        Then the algebra $\Im(\rad) \cong A_k(W)$ is simple if and
        only if  $(\gtilde,\vt)$ is \gainly.  \qed
\end{theorem}

   
\section{Tables of symmetric pairs} \label{app-tables}
  
We adopt the notation of~\cite[Chapter~X]{He1} for the
classification of irreducible symmetric pairs $(\gtilde, \vt) =
(\gtilde,\g)$. The multiplicity $k : R \to \half \N$ is  defined
in~\eqref{k-multiplicity}.  

The columns of the next two  tables give the following information (with minor exceptions in the diagonal case).
 \begin{enumerate}
\item The first column gives the type of $(\gtilde, \vt)$, using the notation from~\cite{He1}, and the type  of the root system $R$.
\item Columns (2)--(4) are self-explanatory.
    \item The  fifth column gives the value  $k_\lambda$ of the multiplicity $k$ on the root
      $\lambda $. When there are two values for $k_\lambda$, the
        top number gives   $k_\lambda$ for the long roots
and the bottom number gives it for the short roots.
      \item The final column  describes, by ``Y'' or ``N'',  whether 
        $A_k(W)$ is simple or not  and gives a justification of this statement;
        namely, ``nice'' if the pair $(\gtilde, \vt)$ is nice, ``BEG'' if
        the simplicity follows from Theorem~\ref{thm:BEG}, ``$\rk 1$'' or ``(x)''
        if the non simplicity is consequence of Corollary~\ref{A1-case2}, respectively 
         is proved in case  (\ref{computations}.x)  from Examples~\ref{computations}.  
\end{enumerate}

\vfill

\vspace*{0.5cm}

\begin{center}

\begin{small}  
 
     \begin{tabular}{|c|c|c|c|c|c|}\hline
$\begin{matrix}  \text{Type}\,  (\gtilde,\vt)\hskip -.02truein
  \\  \text{Type}\ R \hfill \end{matrix}\strut$ 
& $ \begin{matrix}   \gtilde,\g   
  \end{matrix} $ &  $\rk$ & $\dim \p$  & $ k_\lambda$ & 
 simple 
  \\ 
\hline  \hline
$\begin{matrix} \!\!\!\!\!\ \text{ diagonal} \strut \\ \text{type}\hfill
\end{matrix}  \strut  \hfill$ & 
$\begin{matrix} \mathfrak{s}\times\mf{s},\, \mathfrak{s} \strut \\
\mf{s} \ \text{simple}\end{matrix}$ 
& $\rank \fs$ & $\dim \mf{s}$   & $1$
  & $\begin{matrix} \text{Y} \\ \text{nice} \end{matrix}$
  \\ 
\hline  
$\begin{matrix} \!\!\!\!\!\Asf\Isf_n   \strut \\ \mathsf{A}_{n-1}\hfill
\end{matrix}  \strut  \hfill$ & 
$\begin{matrix} \mathfrak{sl}(n),\, \mathfrak{so}(n ) \strut
  \\ 2 \le n  \end{matrix}$ 
& $n-1$ & $ \frac{1}{2}{(n-1)(n+2)}$   & $\frac{1}{2}$ 
  & $\begin{matrix} \text{Y} \\ \text{nice} \end{matrix}$
\\
 \hline
  $ \begin{matrix}  \!\!\!\!\Asf\II_n  \strut \\
  \Asf_{n-1}\hfill \end{matrix}  \strut   \hfill$ & 
$\begin{matrix} \mathfrak{sl}(2n),\, \mathfrak{sp}(n ) \strut  \\
  2 \le n  \end{matrix}$
& $n-1$ & $ {(n-1)(2n+1)}$ &  $2$ &  $\begin{matrix} \text{Y} \\ \text{BEG} \end{matrix}$
\\
  \hline
  $ \begin{matrix} \Asf\III_{p,q}
 \strut \\ \Bsf_p\hfill \end{matrix}  \strut  \hfill$ &
$\begin{matrix} \mathfrak{sl}(p\hskip -.02truein + \hskip -.02truein q),\, \mathfrak{gl}(p)\times\mathfrak{sl}(q)
  \\ p+q=n+1 \\
2\leq p \leq \frac{n}{2} < q \end{matrix}$
& $p$ & $2pq$    &
$  \begin{matrix} 1 \\  n \hskip -.02truein - \hskip -.02truein  2p\hskip -.02truein 
 + \hskip -.02truein \frac{3}{2}    \end{matrix}$  &
$\begin{matrix} \text{N} \\ \text{(1)} \end{matrix}$
\\
\hline 
$ \begin{matrix} \Asf\III_{n,n}    
 \strut \\ \Csf_{n}\hfill \end{matrix}  \strut  \hfill$ &
$\begin{matrix} \mathfrak{sl}(2n),\, \mathfrak{gl}(n)\times\mathfrak{sl}(n)
   \\ 2 \le n   \end{matrix}$ 
& $n$ & ${2}n^2$    &
$  \begin{matrix} \frac{1}{2} \\  1    \end{matrix}$  &
 $\begin{matrix} \text{Y} \\ \text{nice} \end{matrix}$
\\  
  \hline
  $ \begin{matrix} \Asf\IV_n   \strut \\ \Asf_{1}\hfill \end{matrix}  \strut  \hfill$ &
$\begin{matrix} \mathfrak{sl}(n+1),\, \mathfrak{gl}(n) \strut  \\
  2 \le n  \end{matrix}$
& $1$ & $2n$ &   $n-\frac{1}{2}$ & $\begin{matrix} \text{N} \\
  \text{$\rank 1$} \end{matrix}$
\\
  \hline
  $ \begin{matrix} \Bsf\Isf_{p,q}   \strut \\ \Bsf_{p}\hfill \end{matrix}  \strut  \hfill$ &
$\begin{matrix} \mathfrak{so}({ 2n\hskip -.02truein + \hskip -.02truein 1}),\,
 \mathfrak{so}(p )\times\mathfrak{so}(q) \strut  \\ p+q=2n
 \hskip -.02truein +  \hskip -.03truein 1,\\
 \   2\leq p < n  \end{matrix}$
& $p$ & $pq $  &
 $\begin{matrix} \frac{1}{2} \\    \frac{1}{2}(q-p)  \end{matrix}
  $ & $\begin{matrix} \text{N} \\
  \text{(2)} \end{matrix}$
  \\
 \hline
  $ \begin{matrix} \Bsf\Isf_{p,p+1}  \strut \\ \Bsf_{p}\hfill \end{matrix}  \strut  \hfill$ &
$\begin{matrix} \mathfrak{so}({ 2p\hskip -.02truein + \hskip -.02truein 1}),\,
 \mathfrak{so}(p )\times\mathfrak{so}(p+1) \\ 2 \le p \end{matrix}$
& $p$ & $p(p+1) $  &
 $\begin{matrix} \frac{1}{2} \\    \frac{1}{2}  \end{matrix}
  $ & $\begin{matrix} \text{Y} \\
  \text{nice} \end{matrix}$
  \\ 
  \hline
  $ \begin{matrix} \Bsf\II_n   \strut \\ \Asf_{1}\hfill \end{matrix}  \strut  \hfill$ &
$\begin{matrix} \mathfrak{so}(2n+1),\, \mathfrak{so}(2n ) \strut
  \\ 2 \le n  \end{matrix}$
& $1$ & $2n$   & $n-\frac{1}{2}$ & 
$\begin{matrix} \text{N} \\
  \text{$\rank 1$} \end{matrix}$
  \\
  \hline
  $ \begin{matrix} \Csf\Isf_n   \strut \\ \Csf_{n}\hfill \end{matrix}  \strut  \hfill$ &
$\begin{matrix} \mathfrak{sp}(n),\, \mathfrak{gl}(n ) \strut  \\
  2 \le n  \end{matrix}$
& $n$ & $n(n+1)$ 
  & $\begin{matrix}  \frac{1}{2} \\   \frac{1}{2}  \end{matrix}
  $ &
 $\begin{matrix} \text{Y} \\ \text{nice} \end{matrix}$     
  \\
\hline
 $ \begin{matrix} \Csf\II_{p,q}
 \strut \\ \Bsf_p\hfill \end{matrix}  \strut  \hfill$ &
$\begin{matrix} \mathfrak{sp}(p\hskip -.02truein + \hskip -.02truein q),\, \mathfrak{sp}(p)\times\mathfrak{sp}(q)
  \\ p+q=n \\
2\leq p \leq \frac{1}{2}(n \hskip -.02truein - \hskip -.03truein 1) <q\end{matrix}$
& $p$ & $4pq$ & 
$  \begin{matrix} 2 \\    2n\hskip -.02truein - \hskip -.02truein 4p
\hskip -.02truein + \hskip -.02truein \frac{3}{2}    \end{matrix}$  &
$\begin{matrix} \text{N} \\
  \text{(3)} \end{matrix}$
  \\
  \hline
 $ \begin{matrix} \Csf\II_{1,q}
 \strut \\ \Asf_1\hfill \end{matrix}  \strut  \hfill$ &
$\begin{matrix} \mathfrak{sp}(q\hskip -.02truein + \hskip -.02truein 1),\, \mathfrak{sp}(1)\times\mathfrak{sp}(q)
  \\ 2 \le q \end{matrix}$ 
& $1$ & $4q$ & 
$ 2q - \frac{1}{2} $  &
$\begin{matrix} \text{N} \\
  \text{$\rank 1$} \end{matrix}$
  \\
       \hline
       $ \begin{matrix} \Csf\II_{p,p}
 \strut \\ \Csf_p\hfill \end{matrix}  \strut  \hfill$ &
$\begin{matrix} \mathfrak{sp}(2p),\, \mathfrak{sp}(p)\times\mathfrak{sp}(p)
  \\ 2\leq p  \end{matrix}$
& $p$ & $4p^2$   &
$  \begin{matrix} \frac{3}{2} \\  2  \end{matrix}$  &
$\begin{matrix} \text{N} \\
  \text{(4)} \end{matrix}$
  \\
       \hline
        $ \begin{matrix} \Dsf\Isf_{p,q}
 \strut \\ \Bsf_p\hfill \end{matrix}  \hfill$\hfill &
$\begin{matrix} \mathfrak{so}(p \hskip -.02truein +  \hskip -.02truein q),\,
 \mathfrak{so}(p)\times\mathfrak{so}(q)
  \\ p+q=2n, \\  2\leq p\leq n-2  <q\end{matrix}$
& $p$ & $pq$    & $\begin{matrix}  \frac{1}{2} \\ n-p    \end{matrix}$  &
$\begin{matrix} \text{N} \\
  \text{(5)} \end{matrix}$
 \\
       \hline
       $ \begin{matrix} \Dsf\Isf_{p-1,p+1}
 \strut \\ \Bsf_{p-1}\hfill \end{matrix}  \hfill$\hfill &
$\begin{matrix} \mathfrak{so}(2p),\,
\mathfrak{so}(p \hskip -.02truein  -  \hskip -.03truein 1) \hskip -.02truein \times \hskip -.02truein 
\mathfrak{so}(p \hskip -.02truein + \hskip -.03truein 1)
  \\ 3 \le p \end{matrix}$
& $p-1$ & $p^2-1$   & $\begin{matrix} \frac{1}{2}\\ 1 \end{matrix}$  &
 $\begin{matrix} \text{Y} \\
  \text{nice} \end{matrix}$
 \\
       \hline
       $ \begin{matrix} \Dsf\Isf_{p,p}
 \strut \\ \Dsf_{p}\hfill \end{matrix}   \hfill $\hfill &
$\begin{matrix} \mathfrak{so}(2p),\, \mathfrak{so}(p)\times\mathfrak{so}(p)
  \\ 3 \le p \end{matrix}$
& $p$ & $p^2$ &
  $ \frac{1}{2} $  &
$\begin{matrix} \text{Y} \\
  \text{nice} \end{matrix}$
 \\
 \hline
        $\begin{matrix} \Dsf\II_{p}
 \strut \\ \Asf_{1}\hfill \end{matrix}   \hfill $\hfill &
$\begin{matrix} \mathfrak{so}(2p),\, \mathfrak{so}(2p-1) 
  \\  3 \leq p \end{matrix}$
& $1$ & $2p-1$    & $ p-1 $  &
$\begin{matrix} \text{Y} \\
  \text{BEG} \end{matrix}$
 \\
       \hline
       $ \begin{matrix} \Dsf\III_{4p}
 \strut \\ \Csf_p\hfill \end{matrix}  \strut  \hfill$ &
$\begin{matrix} \mathfrak{so}(4p),\, \mathfrak{gl}(2p)
  \\ 2\leq p  \end{matrix} $& 
  $p$ &  $2p(2p-1)$    &
$  \begin{matrix} \frac{1}{2} \\  2  \end{matrix}$  &
$\begin{matrix} \text{N} \\
  \text{(6)} \end{matrix}$
  \\
       \hline
       $ \begin{matrix} \Dsf\III_{4p+2}
 \strut \\ \Bsf_p\hfill \end{matrix}  \strut  \hfill$ &
$\begin{matrix} \mathfrak{so}(4p \hskip -.02truein +  \hskip -.03truein 2),\,
\mathfrak{gl}(2p \hskip -.02truein +  \hskip -.03truein 1 )
  \\ 2\leq p  \end{matrix} $& 
  $p$ &   $2p(2p+1)$   &
$  \begin{matrix} 2 \\  \frac{5}{2}  \end{matrix}$  &
$\begin{matrix} \text{N} \\
  \text{(7)} \end{matrix}$
       \\
       \hline
     \end{tabular}

 \vfill \newpage

\vspace*{0.5cm}

     \begin{tabular}{|c|c|c|c|c|c|}\hline
$\begin{matrix}  \text{Type}\,  (\gtilde,\vt)\hskip -.02truein
  \\ \text{Type}\ R \hfill \end{matrix}\strut$ 
& $ \begin{matrix}   \gtilde,\g  
  \end{matrix} $ &  $\rk$ & $\dim \p$  & $ k_\lambda$ & 
 simple 
  \\ 
       \hline  \hline
$ \begin{matrix} \Esf\Isf   \strut \\ \Esf_6\hfill \end{matrix}  \strut  \hfill$ &
$\begin{matrix} \mathfrak{e}(6),\, \mathfrak{sp}(4 ) \strut  \\   \end{matrix}$
       & $6$ & $42$ &   $\frac{1}{2}$ &
  $\begin{matrix} \text{Y} \\
  \text{nice} \end{matrix}$                                      
  \\
 \hline
$ \begin{matrix} \Esf\II   \strut \\ \Fsf_4\hfill \end{matrix}  \strut  \hfill$ &
$\begin{matrix} \mathfrak{e}(6),\, \mathfrak{sl}(6 )\times \mathfrak{sl}(2) \strut  \\   \end{matrix}$
& $4$ & $40$    &
                  $  \begin{matrix} \frac{1}{2} \\  1  \end{matrix}$  &
$\begin{matrix} \text{Y} \\
  \text{nice} \end{matrix}$                                                                        
\\
 \hline       
 $ \begin{matrix} \Esf\III  \strut \\ \Bsf_2\hfill \end{matrix}  \strut  \hfill$ &
$\begin{matrix} \mathfrak{e}(6),\, \mathfrak{so}(10 )\times \mathfrak{gl}(1) \strut  \\   \end{matrix}$
& $2$ & $32$   &
 $  \begin{matrix} 3 \\  \frac{9}{2}  \end{matrix}$  &
$\begin{matrix} \text{N} \\
  \text{(8)} \end{matrix}$                                                                       
\\
 \hline 
$ \begin{matrix} \Esf\IV   \strut \\ \Asf_2\hfill \end{matrix}  \strut  \hfill$ &
$\begin{matrix} \mathfrak{e}(6),\, \mathfrak{f}(4 ) \strut  \\   \end{matrix}$
& $2$ & $26$ &   $4$ &
$\begin{matrix} \text{Y} \\
  \text{BEG} \end{matrix}$
       \\
 \hline
$ \begin{matrix} \Esf\Vsf   \strut \\ \Esf_7\hfill \end{matrix}  \strut  \hfill$ &
$\begin{matrix} \mathfrak{e}(7),\, \mathfrak{sl}(8 ) \strut  \\   \end{matrix}$
       & $7$ & $70$ &   $\frac{1}{2}$ &
$\begin{matrix} \text{Y} \\
  \text{nice} \end{matrix}$                                        
\\
 \hline
 $ \begin{matrix} \Esf\VI   \strut \\ \Fsf_4\hfill \end{matrix}  \strut  \hfill$ &
$\begin{matrix} \mathfrak{e}(7),\, \mathfrak{so}(12 )\times \mathfrak{sl}(2) \strut  \\   \end{matrix}$
& $4$ & $64$ &  
               $  \begin{matrix}  \frac{1}{2}\\ 2   \end{matrix}$  &
$\begin{matrix} \text{N} \\
  \text{(9)} \end{matrix}$                                                                      
\\
 \hline 
 $ \begin{matrix} \Esf\VII \strut \\ \Csf_3\hfill \end{matrix}  \strut  \hfill$ &
$\begin{matrix} \mathfrak{e}(7),\, \mathfrak{e}(6 )\times \mathfrak{gl}(1) \strut  \\   \end{matrix}$
& $3$ & $54$ &  
$  \begin{matrix}  \frac{1}{2}\\ 4   \end{matrix}$  &
$\begin{matrix} \text{N} \\
  \text{(10)} \end{matrix}$
       \\
 \hline 
$ \begin{matrix} \Esf\VIII 
  \strut \\ \Esf_8\hfill \end{matrix}  \strut  \hfill$ &
$\begin{matrix} \mathfrak{e}(8),\, \mathfrak{so}(16 ) \strut  \\   \end{matrix}$
       & $8$ & $128$ &   $\frac{1}{2}$ &
$\begin{matrix} \text{Y} \\
  \text{nice} \end{matrix}$                                         
\\
 \hline   
  $ \begin{matrix} \Esf\IX \strut \\ \Fsf_4\hfill \end{matrix}  \strut  \hfill$ &
$\begin{matrix} \mathfrak{e}(8),\, \mathfrak{e}(7)\times \mathfrak{sl}(2) \strut  \\   \end{matrix}$
& $4$ & $112$ &  
                $  \begin{matrix}  \frac{1}{2}\\ 4   \end{matrix}$  &
$\begin{matrix} \text{N} \\
  \text{(11)} \end{matrix}$                                                                      
\\
 \hline 
  $ \begin{matrix} \Fsf\Isf  \strut \\ \Fsf_4\hfill \end{matrix}  \strut  \hfill$ &
$\begin{matrix} \mathfrak{f}(4),\, \mathfrak{sp}(3)\times \mathfrak{sl}(2) \strut  \\   \end{matrix}$
& $4$ & $28$ &  
$  \begin{matrix}  \frac{1}{2}\\  \frac{1}{2}   \end{matrix}$
       &
$\begin{matrix} \text{Y} \\
  \text{nice} \end{matrix}$         
\\
\hline
$ \begin{matrix} \Fsf\II   \strut \\ \Asf_1\hfill \end{matrix}  \strut  \hfill$ &
$\begin{matrix} \mathfrak{f}(4),\, \mathfrak{so}(9 ) \strut  \\   \end{matrix}$
       & $1$ & $16$ &  $\frac{15}{2}$ &
$\begin{matrix} \text{N} \\
  \text{$\rank 1$} \end{matrix}$                                        
\\
 \hline   
  $ \begin{matrix} \Gsf  \strut \\ \Gsf_2\hfill \end{matrix}  \strut  \hfill$ &
$\begin{matrix} \mathfrak{g}(2),\, \mathfrak{sl}(2)\times \mathfrak{sl}(2) \strut  \\   \end{matrix}$
& $2$ & $8$ &  
$  \begin{matrix}  \frac{1}{2}\\  \frac{1}{2}  \end{matrix}$
       &
 $\begin{matrix} \text{Y} \\
  \text{nice} \end{matrix}$        
\\
\hline
\end{tabular}

\end{small}
\end{center}


\section{Detailed proofs for Section~\ref{Sec:Cherednik}}\label{App-C}
  
  In this appendix we give the details for a couple of proofs from Section~\ref{Sec:Cherednik}.
  
  \medskip
\emph{ Detailed proofs for Lemma~\ref{lem:BEembedlocal}.}
We keep the notation from the lemma. We first check that the map $\Psi$  is well-defined.

   First, let $u_1, u_2, w \in W$ and $f \in F$. Then 
        \begin{align*}
          (\Psi(u_1u_2)f)(w) & = f(w(u_1u_2)) = f((wu_1) u_2) \\
        & = (\Psi(u_2)f)(wu_1) = (\Psi(u_1)\Psi(u_2)f)(w). 
        \end{align*}
        Next, if $u,w \in W$ and $D \in \dd(\h_{\reg})$ then 
        \begin{align*}
        (\Psi((1 \o u)(D \o 1)(1 \o u^{-1}))f)(w) & = \Psi(u\cdot D \o 1 )f(w) \\
        & = (w\cdot u \cdot D) f(w) = (wu)\cdot D f(w).
        \end{align*}
        and
        \begin{align*}
           (\Psi(1 \o u) \Psi(D \o 1) \Psi(1 \o u^{-1})f)(w) & = (\Psi(D \o 1) \Psi(1 \o u^{-1})f)(wu) \\
        & = (wu)\cdot D (\Psi(1 \o u^{-1})f)(wu) \\
        & =  (wu)\cdot D f(w).
        \end{align*}
        This shows that $\Psi$ is well-defined. The fact that it is then a homomorphism of rings is left to the reader.\qed
 
Next, we check that the diagram in the statement of that lemma is indeed commutative, for which the reader should recall the definition of $f_0$ and $\varphi$ from Lemma~\ref{lem:varphi}. In particular, $\{g_i\}$ is a choice of left coset representatives of $W_H$ in $W$ and $f_0$ is defined by $f_0(w)=e_0$ for $w\in W$. Thus, let $D\in \dd(V_{\reg}^W)$ and trace its image going clockwise round the diagram. Thus, $j(D)=eDe$ and then 
\begin{align*}
 \varphi(e\Phi(eDe)e) \ &= \ e_0 \alpha\bigl(\Phi(D\otimes 1)(f_0)\bigr)
\ = \ e_0 \sum_i \frac{|W_H|}{|W|} \Bigl(\Phi(D\otimes 1)(f_0 )\Bigr)(g_i)e_0  \\
& = \ e_0  \sum_i  \frac{|W_H|}{|W|} (g_i\cdot D)f_0(g_i) e_0  \ = \ 
 \ e_0  \sum_i  \frac{|W_H|}{|W|} D e_0e_0  \ = \ e_0De_0,
  \end{align*}
 where in the last line we have used that  $g_i\cdot D=D$ since $D\in \dd(V_{\reg})^W$.  
  Since $e_0De_0=j_0(D)$,  this proves that the diagram commutes. \qed

\medskip
\emph{Proof of Equation~\ref{eq:imagedunkl}.}
   We remind the reader that we are interested in proving the following formula, where the notation  is set up prior to Lemma~\ref{prop:sphericalinrankone1}.    
  For $y \in \h$ and $w \in W$,
        \begin{equation}\label{eq:imagedunkl1}
                \left(\Psi\left(T_y^W\right) f\right)(w) = T_{w(y)}^{W_H} f(w) +
                \sum_{\begin{smallmatrix} H'\in \calA \\ H' 
                                \neq H\end{smallmatrix}} \frac{\langle w(y), \alpha_{H'}
                        \rangle}{\alpha_{H'}} \sum_{i = 0}^{\ell_{H'} - 1}
                \ell_{H'} \kappa_{{H'},i} f( e_{{H'},i} w),
        \end{equation}
        where, with a slight abuse of  notation, we write 
        \[f( e_{{H'},i} w) := \frac{1}{\ell_{H'}} \sum_{s \in
                W_{H'}} \mr{det}_{\h}(s)^i f(sw).\]

    \begin{proof}  
\begin{align*}
          &     \left(\Psi\left(T_y^W\right) f\right)(w) \ =   \\
          &=  \Psi\left(\partial_y\right) (f)(w) + \sum_{H' \in \calA} \langle y, \alpha_{H'} \rangle \sum_{i = 0}^{\ell_{H'} - 1} \ell_{H'} \kappa_{H',i}\Psi\left(\alpha_{H'}^{-1} e_{H',i} \right) (f)(w) \\
               & = \partial_{w(y)} f(w) + \sum_{H' \in \calA} \langle y, \alpha_{H'} \rangle \sum_{i = 0}^{\ell_{H'} - 1} \ell_{H'} \kappa_{H',i}\Psi(\alpha_{H'}^{-1})(\Psi(e_{H',i}) f)(w) \\
               & = \partial_{w(y)} f(w) + \sum_{H' \in \calA} \langle y, \alpha_{H'} \rangle \sum_{i = 0}^{\ell_{H'} - 1} \ell_{H'} \kappa_{H',i} \ w(\alpha_{H'}^{-1}) \bigg( \frac{1}{\ell_{H'}} \sum_{s \in W_{H'}} \mr{det}_{\h}(s)^{i} f(w s) \bigg)\\
               & = \partial_{w(y)} f(w) + \sum_{H' \in \calA} \langle y, \alpha_{H'} \rangle \sum_{i = 0}^{\ell_{H'} - 1} \ell_{H'} \kappa_{H',i} \ \alpha_{w(H')}^{-1} \bigg( \frac{1}{\ell_{H'}} \sum_{s \in W_{H'}} \mr{det}_{\h}(w s w^{-1})^{i} f(w s w^{-1} w)\bigg) \\
               & = \partial_{w(y)} f(w) + \sum_{H' \in \calA} \langle w(y), \alpha_{w(H')} \rangle \sum_{i = 0}^{\ell_{H'} - 1} \ell_{H'} \kappa_{H',i} \ \alpha_{w(H')}^{-1} \bigg( \frac{1}{\ell_{H'}} \sum_{s \in W_{w(H')}} \mr{det}_{\h}(s)^{i} f(s  w)\bigg) \\  
               & = \partial_{w(y)} f(w) + \sum_{H' \in \calA} \langle w(y), \alpha_{H'} \rangle \sum_{i = 0}^{\ell_{H'} - 1} \ell_{H'} \kappa_{H',i}\ \alpha_{H'}^{-1} \bigg( \frac{1}{\ell_{H'}} \sum_{s \in W_{H'}} \mr{det}_{\h}(s)^{i} f(s w)\bigg) \\
               & = \partial_{w(y)} f(w) + \sum_{H' \in \calA} \frac{\langle w(y), \alpha_{H'} \rangle}{\alpha_{H'}} \sum_{i = 0}^{\ell_{H'} - 1} \ell_{H'} \kappa_{H',i} f(e_{H',i} w) \\
               &  = 
              \partial_{w(y)} f(w) +     \frac{\langle w(y), \alpha_{H} \rangle}{\alpha_{H}} \sum_{i = 0}^{\ell_{H} - 1} \ell_{H} \kappa_{H,i} e_{H,i} f(w)  \ + \\ 
              & \qquad\qquad\qquad  \qquad\qquad\qquad+ \sum_{H' \neq H} \frac{\langle w(y), \alpha_{H'} \rangle}{\alpha_{H'}} \sum_{i = 0}^{\ell_{H'} - 1} \ell_{H'} \kappa_{H',i} f(e_{H',i} w)
               \\
               & = T_{w(y)}^{W_H} f(w) + \sum_{H' \neq H} \frac{\langle w(y), \alpha_{H'} \rangle}{\alpha_{H'}} \sum_{i = 0}^{\ell_{H'} - 1} \ell_{H'} \kappa_{H',i} f(e_{H',i} w)
\end{align*}    
    
  This  concludes the proof of  Equation~\ref{eq:imagedunkl} and hence that of Lemma~\ref{prop:sphericalinrankone1}.
\end{proof}



\end{document}